\documentclass{amsart}[12pt]
\usepackage{graphicx,url,amsthm,xspace,enumitem,subfig}
\usepackage[dvipsnames]{xcolor}
\usepackage{amscd}
\usepackage{multirow}

\usepackage{xy}
\xyoption{all}

\usepackage{pinlabel}
\usepackage[linktocpage]{hyperref}

\usepackage[margin=3cm, marginpar=2.5cm]{geometry}

\newcommand{\C}{\mathbb{C}}
\renewcommand{\P}{\mathbb{P}}
\newcommand{\R}{\mathbb{R}}

\newcommand{\Z}{\mathbb{Z}}
\newcommand{\cptwo}{\C\P^2}
\newcommand{\cpone}{\C\P^1}
\newcommand{\cptwobar}{\overline{\C\P}\,\!^2}

\newcommand{\Ac}{\mathcal{A}}
\newcommand{\Bc}{\mathcal{B}}

\newcommand{\Dc}{\mathcal{E}_6}
\newcommand{\Ec}{\mathcal{E}_3}

\newcommand{\HFplus}{\HF^+}

\newcommand{\hb}{{\rm hb}}

\newcommand{\spinc}{spin$^c$~}
\newcommand{\Spinc}{{\rm Spin}^c}
\newcommand{\ft}{\mathfrak{t}}
\newcommand{\fs}{\mathfrak{s}}

\DeclareMathOperator{\HF}{HF}
\DeclareMathOperator{\tb}{tb}
\DeclareMathOperator{\rot}{rot}

\DeclareMathOperator{\lk}{lk}
\DeclareMathOperator{\PD}{PD}
\DeclareMathOperator{\Sing}{Sing}
\DeclareMathOperator{\Fr}{Fr}

\theoremstyle{plain}
\newtheorem{theorem}{Theorem}[section]

\newtheorem{lemma}[theorem]{Lemma}
\newtheorem{prop}[theorem]{Proposition}

\theoremstyle{definition}
\newtheorem{question}[theorem]{Question}

\theoremstyle{remark}
\newtheorem{remark}[theorem]{Remark}

\numberwithin{equation}{section}

\title{Rational cuspidal curves and symplectic fillings}

\author{Marco Golla}
\email{marco.golla@univ-nantes.fr}
\address{CNRS, Laboratoire de Math\'ematiques Jean Leray, Nantes, France}

\author{Laura Starkston}
\email{lstarkston@math.ucdavis.edu}
\address{Department of Mathematics, UC Davis, One Shields Ave, Davis, CA 95616, U.S.A.}

\begin{document}

\maketitle

\begin{abstract}
A symplectic rational cuspidal curve with positive self-intersection number admits a concave neighborhood, and thus a corresponding contact manifold on the boundary. In this article, we study symplectic fillings of such contact manifolds, providing a complementary perspective to our earlier article on symplectic isotopy classes of rational cuspidal curves. We explore aspects of these symplectic fillings through Stein handlebodies and rational blow-downs. We give examples of such contact manifolds which are identifiable as links of normal surface singularities, other examples which admit no symplectic fillings, and further examples where the fillings can be fully classified.
\end{abstract}


\section{Introduction}
{
In~\cite{bigGS}, we studied classifications of \emph{symplectic rational cuspidal curves} in closed symplectic $4$-manifolds up to equisingular symplectic isotopy (see also~\cite{GKutle} for an extensions of those results for curves in $\cptwo$). Rational means the geometric genus is zero, and cuspidal means that the singularities are locally irreducible and modeled on the singularities of complex curves. 
The link of a cuspidal singularity is an algebraic knot, and we primarily focus on cases where the link is a torus knot $T(p,q)$.

In this article we explore the symplectic and contact geometry associated with rational cuspidal curves from a complementary perspective. Namely we look at the symplectic fillings found as the complement of a neighborhood of such a curve in some closed symplectic $4$-manifold.  We showed in~\cite[Theorem 2.13]{bigGS} that if $C$ is a singular symplectic curve of positive self-intersection, it has small concave neighborhoods, with induced contact boundary we will denote by $(Y_C,\xi_C)$. Note that this contact structure depends only on the topological types of the singularities, geometric genus, and self-intersection number of $C$. In this article, we will study properties of these contact manifolds and their symplectic fillings for many examples of (neighborhood germs of) rational curves $C$, specified by their singularity types and self-intersection number.
We describe some situations where such contact manifolds admit no fillings, and others where the fillings can be completely classified and described explicitly.

Note that any symplectic filling of $(Y_C,\xi_C)$ can be viewed as the complement of a neighborhood of a symplectic embedding of $C$ in a closed symplectic $4$-manifold (obtained by gluing the concave neighborhood of $C$ to the filling). Although our prior work gives classification results for embeddings for many such rational cuspidal curves and thus yields abstract classifications of the symplectic fillings of $(Y_C,\xi_C)$, more work is required to determine concrete diagrammatic presentations for these symplectic fillings and to explore the relations between different fillings. We take on these problems in this article.

	We start with a study of the fillings of $(Y_C,\xi_C)$ where $C$ is a rational cuspidal curves with a unique singularity modeled on $\{x^a=y^b\}$ (whose link is the torus knot $T(a,b)$) which embeds into $\cptwo$ in the homology class $d[\cpone]\in H_2(\cptwo)$. Note in this case, the corresponding self-intersection number is $d^2$. According to~\cite[Theorem~2.3]{Liu} (see also~\cite[Remark~6.18]{BCG}), all such $C$ belong to the list of~\cite[Theorem 1.1]{FdBLMHN}; namely, $(p,q;d)$ is one of:
\begin{align*}
	\Ac_p&: (p,p+1;p+1), 	&& (F_{j-2},F_{j+2};F_j), 					&& \Ec: (3,22;8),\\
	\Bc_p&: (p,4p-1;2p), 		&& (F_j^2, F_{j+2}^2; F_jF_{j+2}), 		&& \Dc: (6,43;16).
\end{align*}	
	Here $F_i$ denotes the $i^{\rm th}$ Fibonacci number. The two Fibonacci families correspond to curves $C$ where $(Y_C,\xi_C)$ is either a universally tight lens space or a connected sum of universally tight lens spaces.
	In these cases, the fillings have been previously classified, some of them have been presented through Stein handle diagrams and Lefschetz fibrations, and their rational blow-down relations have been established~\cite{Honda, Li, LekiliMaydanskiy, BhupalOzbagci}.
	We denote the remaining two infinite families with $\mathcal{A}_p$ and $\mathcal{B}_p$, corresponding to $(a,b) = (p,p+1)$ (and degree $p+1$) and $(a,b) = (p,4p-1)$ (and degree $2p$), respectively. The two exceptional cases, corresponding to $(a,b) = (3,22), (6,43)$ (of degrees $8$ and $16$) are denoted with $\mathcal{E}_3$ and $\mathcal{E}_6$. We find presentions for each of the corresponding contact manifolds $(Y_C,\xi_C)$ as the boundary of a Stein handlebody diagram.

\begin{theorem}\label{thm:unicuspidal}
	Let $C$ be a rational cuspidal curve in $\cptwo$ with one cusp of torus knot type.
	\begin{itemize}
		\item If $C$ is of type $\Ac_p$, then the unique minimal symplectic filling of $\xi_C$ has the Stein handlebody diagrams shown in Figure~\ref{fig:ApStein}, so $(Y_C,\xi_C)$ is specified by the corresponding contact surgery diagram.
		
		\item If $C$ is of type $\Bc_p$, then the two minimal symplectic fillings are given by Stein handlebodies in Figures~\ref{fig:BpStein} and~\ref{fig:BpStein2}, so $(Y_C,\xi_C)$ is specified by the corresponding contact surgery diagrams.
		
		\item If $C$ is of type $\Ec$ or $\Dc$, then $\xi_C$ is the canonical contact structure on $Y_C$ given as the link of a normal complex surface singularity whose minimal resolution is symplectic deformation equivalent to the Stein manifolds whose handle diagrams are shown in Figures~\ref{fig:E3Stein} and~\ref{fig:E6Stein}, thus yielding contact surgery diagrams for the corresponding contact manifolds.
	\end{itemize}
\end{theorem}

In the $\Ac_p$ and $\Bc_p$ cases, to prove the Stein handlebodies we present have contact boundaries agreeing with $(Y_C,\xi_C)$, we use Kirby calculus to identify the smooth boundary of the handlebody with $Y_C$, and an argument using Gompf's $\Gamma$-invariant~\cite{Gompf2} of contact structures and results of Matkovi\v{c}~\cite{Matkovic} to identify the contact boundary of the Stein handlebody with $\xi_C$. The key technical tool is a general computation of Gompf's $\Gamma$-invariant for the contact manifolds $(Y_C,\xi_C)$ when $C$ is an arbitrary rational cuspidal curve (Theorem~\ref{thm:Gamma(xi,s0)}).

In the $\Ec$ and $\Dc$ cases, to identify the contact structure as the canonical structure from a complex surface singularity we use classification results of Ghiggini~\cite[Theorem~1.3]{Ghiggini-Lclassification} (see also~\cite{Tosun} for a more general statement) and Heegaard Floer correction terms.
	
%

Next we consider relations between different symplectic fillings of the same $(Y_C,\xi_C)$. We focus on when such fillings are related by \emph{rational blow-down}, a surgery operation introduced by Fintushel and Stern~\cite{FintushelStern} and generalized by~\cite{Park} and~\cite{SSW}. Because its effect on Seiberg--Witten invariants has been established, rational blow-down has proven a very useful tool to produce small exotic $4$-manifolds (see \cite{ParkExotic,StipsiczSzabo,ParkStipsiczSzabo,FintushelSternDN} as a small sample of the many examples in the literature).
When we have multiple symplectic fillings of a given contact manifold, we can ask whether they are related via a known rational blow-down operation, or whether the substitution of fillings yields a new symplectic cut-and-paste operation. In the case of lens spaces with the canonical contact structure Bhupal and Ozbagci proved that all symplectic fillings are obtained from a plumbing by a sequence of symplectic rational blow-downs~\cite{BhupalOzbagci}.
We prove the analogous result for the cases of $(Y_C,\xi_C)$ considered above: every minimal symplectic filling, can be obtained from the largest filling by a sequence of symplectic rational blow-downs (where we include generalized rational blow-downs of~\cite{Park} and~\cite{SSW}).

\begin{theorem}\label{thm:qbdintro}
Let $C$ be a rational unicuspidal curve in $\cptwo$ whose singularity link is a torus knot.
\begin{itemize}
\item If $C$ is of type $\Bc_p$, the two fillings of $(Y_C,\xi_C)$ differ by rationally blowing down a symplectic $-4$-sphere (or rationally blowing up a Lagrangian $\mathbb{RP}^2$).

\item If $C$ is of type $\Ec$ or $\Dc$, one filling is given by a symplectic plumbing of spheres and every other filling is obtained from this plumbing by a sequence of symplectic rational blow-downs.
\end{itemize}
\end{theorem}

In the $\Ec$ and $\Dc$ cases, we specify precisely which pairs of fillings are related by a rational blow-down. In each case, there are pairs of fillings which cannot be related by rational blow-downs---instead one must first rationally blow-up one of the fillings to get to a larger filling and then rationally blow-down to get to the second filling.

Generalizing from the case of unicuspidal $C$ which embed in $\cptwo$, we can consider more general rational cuspidal curves and the corresponding contact manifolds $(Y_C,\xi_C)$. We specify the curve data determining $(Y_C,\xi_C)$ through the singularity types of $C$ and the self-intersection number, $s$, of $C$. In this more general setting, we find a large class of such contact structures which admit no fillings at all. In Section~\ref{ss:bounds} we will define two effectively computable invariants of cuspidal singularities, $M$ and $\ell$, and we will prove a slight generalization of the following theorem.

\begin{theorem}\label{thm:bounds}
Let $C$ be a rational cuspidal curve with $\Sing(C) \neq \emptyset$ such that $(Y_C,\xi_C)$ is weakly fillable. Then
\[
C\cdot C \le \sum_{p \in \Sing(C)} M(p_i) + 2\min_{p \in \Sing(C)} \ell(p_i) + 1.
\]
\end{theorem}

Finally, we look at symplectic filling classification results in some low complexity cases. In a rational cuspidal curve $C$, each singularity has an underlying arithmetic genus. This can be thought of as the genus that would be added to the curve $C$ by smoothing the cusp singularity. Adding up the arithmetic genus contributions from each singularity gives us the \emph{arithmetic genus} of the singular curve $C$. To give a sample of the filling classifications for these more general examples $(Y_C,\xi_C)$ we investigate examples where $C$ has low arithmetic genus, varying $s$ within a range the allows us to give complete classifications up to the bounds where Theorem~\ref{thm:bounds} applies. We include these classification results in Section~\ref{s:lowgenus}. In some cases, there is a unique minimal symplectic filling and in others there are multiple minimal symplectic fillings yielding potentially new symplectic cut-and-paste operations.

\subsubsection*{Organization}
In Section~\ref{s:AB}, we give Stein handlebodies for the fillings of $(Y_C,\xi_C)$ in the $\Ac_p$ and $\Bc_p$ cases, and prove the first two items of Theorem~\ref{thm:unicuspidal}. Next, in Section~\ref{s:E3E6}, we identify $(Y_C,\xi_C)$ with canonical contact manifolds in the $\Ec$ and $\Dc$ cases, and provide their Legendrian surgery diagrams, thus completing the proof of Theorem~\ref{thm:unicuspidal}. In Section~\ref{s:background}, we provide a brief summary of background from~\cite{bigGS} that we will need for the remaining sections. In Section~\ref{s:qbd} we determine which fillings of $(Y_C,\xi_C)$ are related by rational blow-downs when $C$ is a unicuspidal curve in $\cptwo$, proving Theorem~\ref{thm:qbdintro}. Next, we look at more general examples of $(Y_C,\xi_C)$. In Section~\ref{ss:bounds}, we prove Theorem~\ref{thm:bounds}. In Section~\ref{s:lowgenus}, we give classifications of fillings of $(Y_C,\xi_C)$ for cases when $C$ has small arithmetic genus. Finally, in Section~\ref{s:speculations} we provide some questions and conjectures that build off of the results we established.

\subsubsection*{Acknowledgements}
Early progress on this project was made during the BIRS conference in Oaxaca \emph{Thirty years of Floer Theory for 3-manifolds}, and the authors appreciate this excellent conference and opportunity to collaborate.
MG acknowledges hospitality from UC Davis and LS is grateful for hospitality from Universit\'{e} de Nantes. We also thank Fabien K\"utle and B\"ulent Tosun for helpful conversations related to this project. LS has been supported by NSF grants DMS 1904074 and DMS 2042345 during this project.
}

\section{Stein handlebodies for fillings in the $\Ac_p$ and $\Bc_p$ families}\label{s:AB}

Our first goal is to explicitly present Stein handlebody diagrams for the fillings of $(Y_C,\xi_C)$ in the case that $C$ is a rational unicuspidal curve in either the family $\Ac_p$ or $\Bc_p$. Such presentations are critical to study properties of the fillings (such as the types of surfaces they contain) and to compute their invariants. Additionally, the Stein handlebody induces a Legendrian surgery diagram for the boundary. These provide explicit ways to understand the corresponding cuspidal contact structures. 

In order to verify that our Stein handlebodies have the correct contact structure on the boundary, we will use a result of Matkovi\v c~\cite{Matkovic} which classifies contact structures on Seifert fibered L-spaces in terms their underlying \spinc structure. To encode the \spinc structure concretely, we will utilize Gompf's $\Gamma$-invariant~\cite{Gompf2}.
Gompf provided formulas to compute this invariant for contact manifolds appearing on the boundary of a Stein handlebody in terms of the Stein handlebody diagram. 
We will be interested in comparing contact structures arising on the boundary of a given Stein handlebody, with the contact structures $(Y_C,\xi_C)$ induced concavely on the boundary of a neighborhood of a rational unicuspidal curve. To facilitate this comparison, we will compute Gompf's $\Gamma$-invariant for the concavely induced structure $(Y_C,\xi_C)$. 

\subsection{Gompf's $\Gamma$-invariants for rational cuspidal contact manifolds} \label{ss:Gamma}

Gompf's $\Gamma$-invariant measures the $2$-dimensional obstruction to uniqueness of homotopy classes of $2$-plane fields on a $3$-manifold. Any $2$-plane field $\xi$ on a $3$-manifold $Y$ determines a \spinc structure on $Y$ (by determining a complex structure on $TY\oplus \R$). $\Spinc$ structures on $Y$ are affinely in bijection with $H^2(Y;\Z)$ through the corresponding free and transitive action. For a more canonical identification, one can let $\Fr(Y)\to Y$ denote the principal $SO(3)$-bundle of oriented orthonormal frames. Then $\Spinc(Y)$ is identified with the set of elements in $H^2(\Fr(Y);\Z)$ whose restrictions to the $SO(3)$ fibers yield the non-zero element of $H^2(SO(3);\Z)=\Z/2\Z$ (see~\cite[Chapter XI.1.2]{Turaev}). A choice of trivialization $\tau$ of $TY$ gives an identification of $\Fr(Y)$ with $Y\times SO(3)$, and thus yields an identification of $\Spinc(Y)$ with $H^2(Y;\Z)$ (note this is not canonical as it depends on the trivialization $\tau$). Gompf defines $\Gamma_\tau(\xi)\in H_1(Y;\Z)$ to be the Poincar\'{e} dual of the class in $H^2(Y;\Z)$ identified using $\tau$ with the \spinc structure determined by $\xi$. In~\cite[Proposition 4.1]{Gompf2}, he proves various properties including that $2\Gamma_\tau(\xi)$ is Poincar\'{e} dual to $c_1(\xi)$.
A trivialization $\tau$ of $TY$ determines a spin structure on $Y$ (since a spin structure is a trivialization of the tangent bundle over the 2-skeleton). For each spin structure $s$ on $Y$, Gompf defines an invariant $\Gamma(\xi,s)\in H_1(Y;\Z)$. Proposition 4.8 of~\cite{Gompf2} proves that if $s$ is the spin structure determined by a trivialization $\tau$ then $\Gamma_\tau(\xi) = \Gamma(\xi,s)$ (i.e. $\Gamma_\tau(\xi)$ only depends on the spin structure induced by $\tau$).

	We recall some background on spin and \spinc structures. When $Y$ is a rational homology 3-sphere, spin structures on $Y$ are in bijective correspondence with self-conjugate \spinc structures on $Y$. To see this, following Turaev~\cite[Chapter XI]{Turaev}, first identify spin structures on $Y$ with the elements of $H^1(\Fr(Y);\Z/2\Z)$ whose restrictions to the $SO(3)$ fibers are the non-trivial element ($\Fr(Y)$ is the $SO(3)$ bundle of oriented orthonormal frames over $Y$). Then the short exact sequence $\Z\to \Z\to \Z/2\Z$ induces the long exact sequence on cohomology
\[
\xymatrix{
	\ar[r] &H^1(\Fr(Y);\Z)\ar[r] &H^1(\Fr(Y);\Z/2\Z)\ar[r]^{\beta}& H^2(\Fr(Y);\Z) \ar[r]^{\times 2}& H^2(\Fr(Y);\Z) \ar[r]&
}
\]
The Bockstein homomorphism $\beta: H^1(\Fr(Y);\Z/2\Z)\to H^2(\Fr(Y);\Z)$ maps spin structures on $Y$ to \spinc structures on $Y$. Since $Y$ is a rational homology $3$-sphere, $H^1(\Fr(Y);\Z)=0$, so this map is injective. The image of the spin structures is precisely the \spinc structures $\fs_0$ such that $2\fs_0=0$, or equivalently $\fs_0=-\fs_0$ (self-conjugate).

Let $N$ be a standard concave neighborhood of a rational cuspidal curve $C$ of arithmetic genus $g$ and self-intersection $n>0$, with boundary $(Y,\xi)$ (we will omit the $C$ subscripts for the rest of this section to have cleaner notation). Note, the orientation on $Y$ \emph{disagrees} with the boundary orientation from $N$. Let $K$ denote the connected sum of the links of the singularities of $C$ (in the $\Ac_p$ and $\Bc_p$ cases, $K$ will be the corresponding torus knot). Then $N$ is diffeomorphic to the knot trace $X_{n}(K)$ (the manifold obtained by attaching a single $2$-handle to a $0$-handle along $K$ with framing $n$). Therefore $-Y$ has a surgery diagram consisting of the knot $K$ with framing $n$. Correspondingly $Y$ has a surgery diagram consisting of the mirror knot $m(K)$ with framing $-n$.

A spin structure on a $3$-manifold with a given integral surgery diagram can be encoded by a \emph{characteristic sublink} of the surgery diagram~\cite{Kaplan} (see~\cite[Proposition 5.7.11]{GompfStipsicz}). In a $3$-manifold obtained by integral surgery on $S^3$ along a link $L$, a sublink $L'\subset L$ is \emph{characteristic} when, for each component $K_i$ of $L$, the framing of $K_i$ is congruent modulo $2$ to $\lk(K_i,L')$. In the surgery diagram for $-Y$, which consists of a single component with framing $n$ the only sublinks are the full link and the empty link. In the case that $n$ is even, both of these are characteristic and there are two spin structures on $-Y$. In the case that $n$ is odd, only the non-empty sublink is characteristic and there is a unique spin structure. In either case, let $s_0$ denote the spin structure on $-Y$ which corresponds to the non-empty characteristic sublink.

\begin{theorem}\label{thm:Gamma(xi,s0)}
	Let $(Y,\xi) = (Y_C,\xi_C)$, where $C$ is a rational cuspidal curve of arithmetic genus $g$ and self-intersection $n$.
	Let $K$ be the connected sum of the links of the singularities as above. Consider the surgery diagram for $-Y$ given by $n$-surgery on $K$. Then
	\[
	\Gamma(\xi,s_0) = (1-g)\mu \in H_1(-Y),
	\]
	where $s_0$ is the unique spin structure on $-Y$ represented by the characteristic sublink $K$, and $\mu$ is the homology class represented by the meridian of $K$. 
\end{theorem}

\begin{proof}
	
	To set notation, we begin with some basic facts about homology and cohomology. $H_2(N;\Z)\cong \Z$ with intersection form $\langle n \rangle$ and is generated by $[C]$. $H^2(N;\Z)\cong \Z$ and is generated by a class $G$ where $\langle G, [C]\rangle =1$. 
	On the boundary, $H^2(-Y;\Z)\cong H_1(-Y;\Z)\cong \Z/n\Z$, generated by the meridian $\mu$ of the surgery knot. 
	The restriction map $i^*: H^2(N;\Z)\to H^2(-Y;\Z)$ sends $G$ to the Poincar\'e dual of $\mu$.
	
	Since, spin$^c$ structures on a manifold $M$ are affinely identified with $H^2(M;\Z)$, given any two \spinc structures $\fs$ and $\fs'$ on $N$, their difference $\fs'-\fs$ represents a class in $H^2(N;\Z)$. Under this identification, $c_1(\fs')-c_1(\fs) = 2(\fs'-\fs)$.
	
	We will first fix a \emph{reference \spinc structure}, $\fs_0$ on $N$, defined by the property that
	\[
	\langle c_1(\fs_0), [C]\rangle =n.
	\]
	One important property of this reference \spinc structure is that its restriction to $\partial N=-Y$ is self-conjugate. To see this, observe that
	\[
	\langle c_1(\overline{\fs}_0), [C]\rangle =-n,
	\]
	so $\langle c_1(\fs_0)-c_1(\overline{\fs}_0), [C]\rangle = 2n$. Therefore $\fs_0-\overline{\fs}_0 = n G \in H^2(N;\Z)$. Let $i^*: H^2(N;\Z)\to H^2(-Y;\Z)$ denote the restriction map to the boundary. Then $i^*(\fs_0-\overline{\fs}_0) = n\PD(\mu) = 0\in H^2(-Y;\Z)$. Therefore, $i^*(\fs_0)$ is self-conjugate so it corresponds to a spin structure $s_0$ on $-Y$.
	
	Recall the correspondence described above between spin structures on $-Y$ and characteristic sublinks of the surgery diagram for $-Y$ given by $n$-surgery on $K$. 
	To see that the empty sublink does not correspond to $s_0$ when $n$ is even, consider the unique \spinc structure $\ft$ on $N$ which satisfies $\langle c_1(\ft), [C]\rangle =0$. This \spinc structure is self-conjugate on all of $N$, since $\ft-\overline{\ft}=0\in H^2(N;\Z)$. Therefore $\ft$ is the image of the unique spin structure on $N$ under the corresponding Bockstein homomorphism $\beta: H^1(\Fr(N);\Z/2\Z)\to H^2(\Fr(N);\Z)$. 
	Next, observe that $\langle c_1(\fs_0)-c_1(\ft), [C]\rangle =n$, so $\fs_0-\ft=\frac{n}{2}G\in H^2(N;\Z)$, and the restriction to $-Y$, $i^*(\fs_0-\ft)=\frac{n}{2}\PD(\mu)\in H^2(-Y;\Z)$. Since $\frac{n}{2}\PD(\mu)$ is a non-zero element in $H^2(-Y;\Z)$, $i^*(\fs_0)\neq i^*(\ft)$. Therefore these two self-conjugate \spinc structures on $-Y$ correspond to distinct spin structures on $-Y$. Since $i^*(\ft)$ corresponds to the empty sublink (by definition, the empty sublink of an integral surgery diagram corresponds to the spin structure which extends over the corresponding 4-manifold described as that 2-handlebody), $s_0$ must correspond to the non-empty sublink.
		
	Now, we consider two \spinc structures on $N$, our reference \spinc structure $\fs_0$ and the \spinc structure $\fs_J$, coming from an almost complex structure on $N$ compatible with the symplectic structure, such that the complex tangencies of $J$ along $Y$ give the contact structure $\xi$. The difference of these two \spinc structures corresponds to a class $\fs_J-\fs_0\in H^2(N;\Z)$. We know that $2(\fs_J-\fs_0) = c_1(\fs_J)-c_1(\fs_0)$, so since $H^2(N;\Z)\cong \Z$ has no $2$-torsion, we can determine the class $\fs_J-\fs_0$ uniquely from computing the Chern classes.
	
	By definition, we chose $\fs_0$ to have $c_1(\fs_0)=nG$ (where $G$ is the generator dual to $[C]$). We can use the adjunction formula to compute
	\[
	\langle c_1(\fs_J), [C]\rangle = 2-2g+n.
	\]
	(This can be seen by replacing $C$ by a curve which smooths its singularities symplectically, thus increasing the genus of $C$ to its arithmetic genus without changing the homology class it represents.)
	
	Therefore $c_1(\fs_J)-c_1(\fs_0) = (2-2g+n-n)[G]$, so 
	\[
	\fs_J-\fs_0 = (1-g)G \in H^2(N;\Z).
	\]
	
	Restricting this class to the boundary, $\partial N = -Y$, we have $i^*(\fs_J)-i^*(\fs_0) = (1-g)\PD(\mu)$.
	
	Now recall the definition of $\Gamma(\xi,s_0)$ from the beginning of this subsection. Since $\xi$ is given by the complex tangencies to $-Y$ under the almost complex structure $J$, $\fs_J$ restricts to the \spinc structure on $-Y$ determined by $\xi$. If $\tau$ is a trivialization that induces the spin structure $s_0$, then $\Gamma(\xi,s_0)$ is the Poincar\'{e} dual to the class in $H^2(-Y;\Z)$ corresponding to the \spinc structure $i^*(\fs_J)$ (where the correspondence depends on $\tau$). Note that under this correspondence the \spinc structure $i^*(\fs_0)$ corresponds to the $0$ element in $H^2(-Y;\Z)$. Therefore, $\Gamma(\xi,s_0)=\PD(i^*(\fs_J)-i^*(\fs_0))$, and the result follows.
\end{proof}

\subsection{Stein handlebodies for fillings}

Observe that any symplectic filling of $(Y_C,\xi_C)$ can be capped by a concave neighborhood of the cuspidal curve $C$ to form a closed symplectic manifold. Equivalently, every symplectic filling of $(Y_C,\xi_C)$ arises as the complement of a neighborhood of a symplectic embedding of $C$ into a closed symplectic manifold. When $C$ is a rational and unicuspidal curve in $\cptwo$, we classify such embeddings in~\cite[Theorem 6.5]{bigGS}. The goal of this section is to turn this perspective upside down, by translating our understanding of the embedding of $C$ into a concrete description of the Stein domain in its complement. 

For smooth handle decompositions, we can literally turn handle diagrams upside down (negating the Morse function) to obtain handle decompositions for many complements. However, it is not clear that a given smooth handle decomposition for a complement will support a Stein structure. Since there is not a notion of symplectic handle attachment for closed $4$-manifolds, the strategy of turning things upside down does not suffice to identify a Stein handle structure on the complement, so we will use additional arguments to show we have identified the correct geometric structure.

We start with the case where $C$ is in the family $\Ac_p$. In this case $C$ has a unique singularity whose link is $T(p,p+1)$ so its arithmetic genus is $\frac{p(p-1)}{2}$. The self-intersection number is $(p+1)^2$. Then~\cite[Theorem 6.5]{bigGS} tells us that the only minimal symplectic embedding of $C$ into a closed symplectic manifold is into $\cptwo$ and that there is a unique symplectic isotopy class of such an embedding. As a consequence $(Y_C,\xi_C)$ has a unique minimal symplectic filling, which is an affine complex surface. Our aim is to give a Stein handlebody description of this object.
	
\begin{theorem}\label{thm:ApStein}
Fix $p > 1$, and let $C$ be a curve of type $\Ac_p$. Then the unique filling of $(Y_C,\xi_C)$ is given by the Stein handlebody $W_\hb$ given by the diagram in Figure~\ref{fig:ApStein}  (up to symplectic deformation and symplectomorphism).
\end{theorem}

\begin{figure}[h!]
	\centering
	\includegraphics[scale=.6]{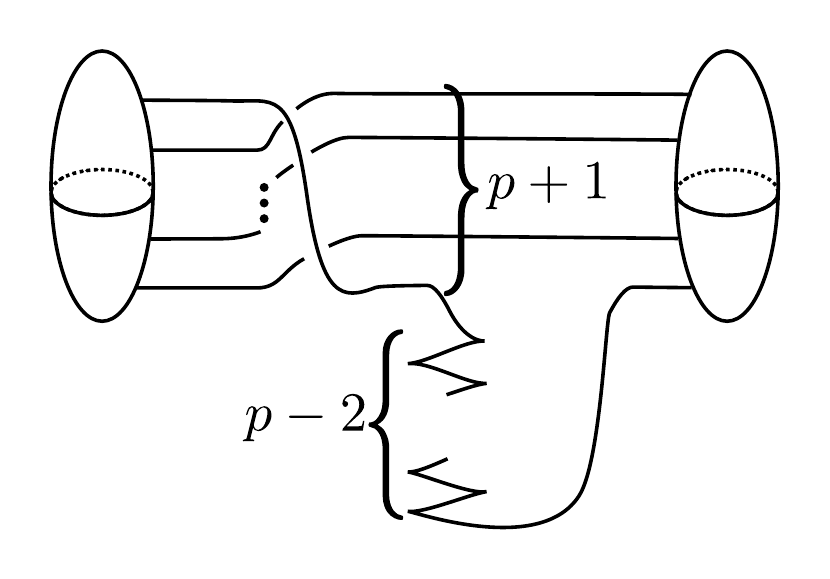}
	\caption{Stein handlebody for the filling corresponding to the $\Ac_p$ family. The framing on the $2$-handle is $tb-1 = 1$.}
	\label{fig:ApStein}
\end{figure}

That the filling has a handle decomposition with only one 1-handle and one 2-handle was to be expected from the proof of~\cite[Theorem~6.5]{bigGS}. Indeed, in that proof we show that, letting $T$ be the tangent to the cusp of $C$, the configuration $C \cup T$ is birationally equivalent to the configuration of two lines in $\cptwo$, and in particular its complement is diffeomorphic to $S^1\times D^3$. The handle decomposition corresponding to adding a neighborhood of $T \setminus C$ corresponds to attaching a 2-handle. In fact, using this argument one obtains the (smooth) Kirby diagram underlying Figure~\ref{fig:ApStein}.

\begin{proof}
To simplify notation, we will write $(Y,\xi)$ instead of $(Y_C,\xi_C)$. Since $Y$ is the boundary of a neighborhood of the rational cuspidal curve of type $\Ac_p$ (with the orientation reversed), $Y = -S^3_{(p+1)^2}(T(p,p+1))$; it follows, for instance from~\cite[Lemma~4.4]{OwensStrle}, that $Y$ is a small Seifert fibered space with Seifert invariants $(-1;\frac{p}{p+1},\frac1p,\frac1{p+1})$ (see Figure~\ref{f:Seifert} for the Seifert fibered notation).

Since $(p^2+p+1)$-surgery along $T(p,p+1)$ is a lens space~\cite{Moser}, the torus knot $T(p,p+1)$ is an L-space knot. Since $(p+1)^2 \ge 2g(T(p,p+1)) - 1$, by~\cite[Proposition~4.1]{LiscaStipsicz-I} $Y$ is a Heegaard Floer L-space.

Matkovi\v{c} proved in~\cite[Theorem~1.3]{Matkovic} that two tight contact structures on a small Seifert fibered L-space are isotopic if and only if they induce the same spin$^c$ structure. We will compute Gompf's $\Gamma$-invariant (with respect to the same spin structure on $Y$) for the contact structure $\xi_{\rm hb}$ described by the Stein handlebody of Figure~\ref{fig:ApStein} and use Theorem~\ref{thm:Gamma(xi,s0)} to say that $\Gamma(\xi,s_0) = \Gamma(\xi_{\rm hb}, s_0)$; that is, $\xi_{\rm hb}$ and $\xi$ induce the same spin$^c$ structure. Since both $\xi_{\rm hb}$ and $\xi$ are tight, by Matkovi\v{c}'s result they are isotopic. Note that when $p+1$ is odd computing the first Chern class of $\xi_{\rm hb}$ and $\xi$ would suffice. However, when $p+1$ is even, $H_1(Y)$ contains 2-torsion, so knowing that the first Chern classes agree does not suffice to prove that the spin$^c$ structures are the same.


\begin{figure}
	\centering
	\includegraphics[scale=.5]{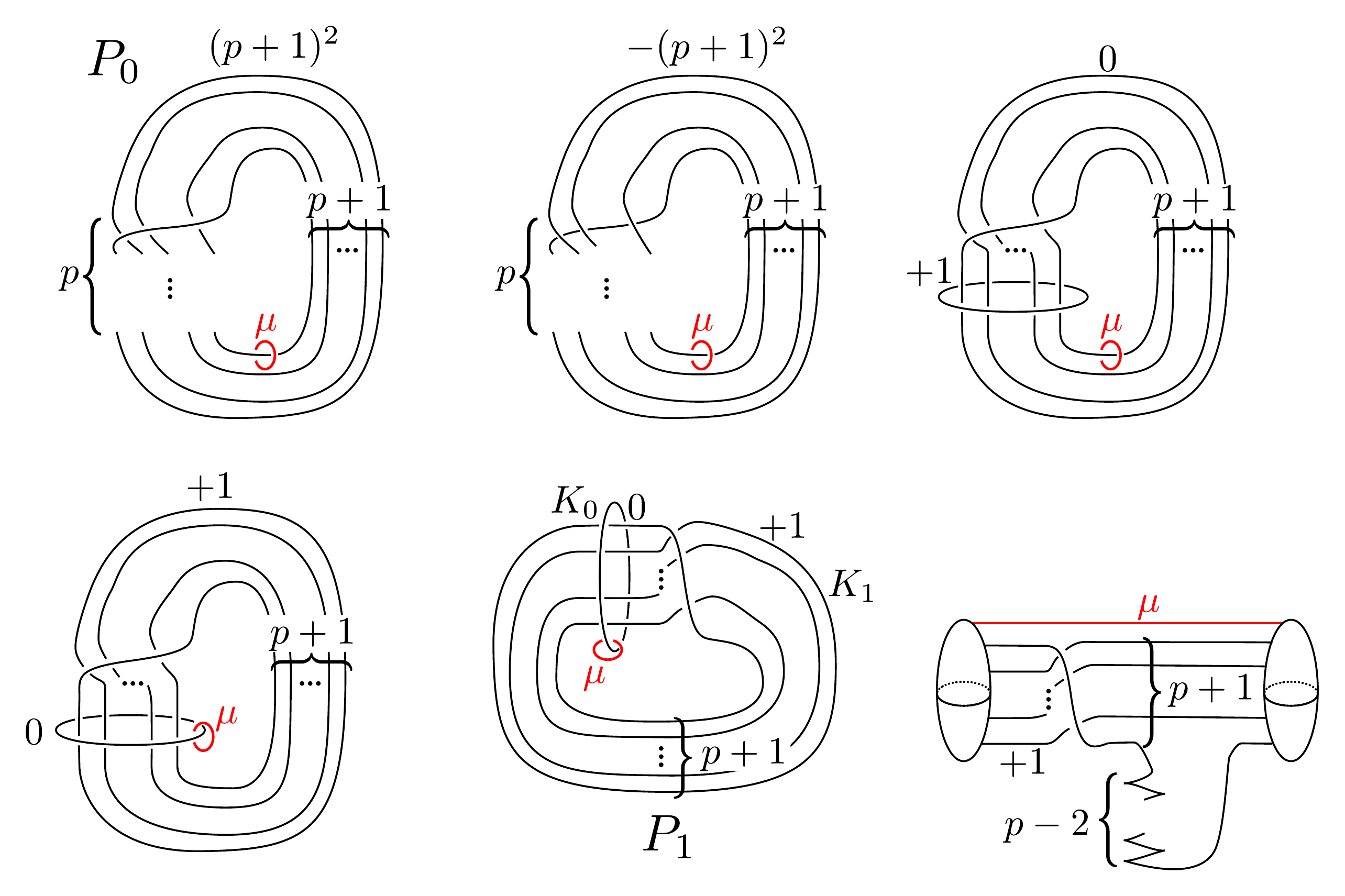}
	\caption{Kirby calculus moves relating surgery presentation $P_0$ to surgery presentation $P_1$, where the surgery presentation $P_1$ is induced by the Stein handlebody in the bottom right.}
	\label{fig:Kirbycalc}
\end{figure}

We have one surgery presentation, $P_0$, for $-Y$ given by $(p+1)^2$-surgery on $T(p,p+1)$. In order to compare $\Gamma$-invariants, we need to relate this to the surgery diagram for $Y$ given as the boundary of $W_\hb$. Note that in order to obtain a surgery diagram for $Y$ from the handle diagram for $W_\hb$ of Figure~\ref{fig:ApStein}, we need to switch to dotted circle notation for the $1$-handle, and then replace the dotted circle with a $0$-framed circle. The resulting diagram represents a smooth $4$-manifold $W^*_\hb$ obtained by surgering out the $S^1\times D^3$ going around the $1$-handle and replacing it by $S^2\times D^2$, but its boundary is unchanged. We will refer to this surgery diagram for $Y$ as $P_1$, and will denote its two components by $K_0$ and $K_1$ where $K_0$ is $0$-framed and $K_1$ is $+1$-framed. (To see the framing on $K_1$ from Figure~\ref{fig:ApStein}, calculate $\tb-1$ using~\cite[Equation (1.1)]{Gompf2}, using the framing convention of~\cite[page~634]{Gompf2}.) To relate $P_0$ and $P_1$, we first reverse the orientation of $P_0$, and then perform a sequence of Kirby calculus moves as in Figure~\ref{fig:Kirbycalc}. The first Kirby calculus move is an anti-blow-up, which introduces a $+1$-framed unknot and adds a full positive twist into the strands of the mirror of $T(p,p+1)$ and changes its framing to $0$. This results in a symmetric link, so there is an isotopy which exchanges the two components, which results in the final surgery presentation $P_1$. This sequence of Kirby calculus moves provides an explicit diffeomorphism identifying the two surgery presentations. 
We can carry the curve in $Y$ given by the meridian $\mu$ of $T(p,p+1)$ through the Kirby calculus moves to see its image under this diffeomorphism. Following Figure~\ref{fig:Kirbycalc}, we see that $\mu$ is sent to the meridian of $K_0$ in the surgery presentation $P_1$, or equivalently a curve which passes once through the $1$-handle in the boundary of $W_\hb$.


In the statement of Theorem~\ref{thm:Gamma(xi,s0)}, the spin structure $s_0$ is represented in the surgery presentation $P_0$ by the non-empty characteristic sublink, i.e. $T(p,p+1)$. 
We would like to identify this characteristic sublink in the surgery presentation $P_0$ with a characteristic sublink of the surgery presentation $P_1$. Note that, when $p+1$ is odd, there is a unique characteristic sublink $K_0$ in the surgery diagram $P_1$.
When $p+1$ is even, there are two characteristic sublinks of $P_1$: $K_1$ and $K_0\cup K_1$, so we need to determine which of these corresponds to $s_0$. 
It is explained in~\cite[Pages 190--191]{GompfStipsicz} how to track the characteristic sublink through Kirby calculus moves, and specifically through a blow-up. The characteristic sublink during a blow-up is unchanged except that the new $+1$-framed unknotted component is included if and only if it has even linking number with the previous characteristic sublink. Therefore, after we push $s_0$ through the Kirby calculus moves, it is represented by the sublink $K_0\cup K_1$ when $p+1$ is even, and by $K_0$ when $p+1$ is odd. See Figure~\ref{fig:Kirbyspin}.

\begin{figure}
	\centering
	\includegraphics[scale=.5]{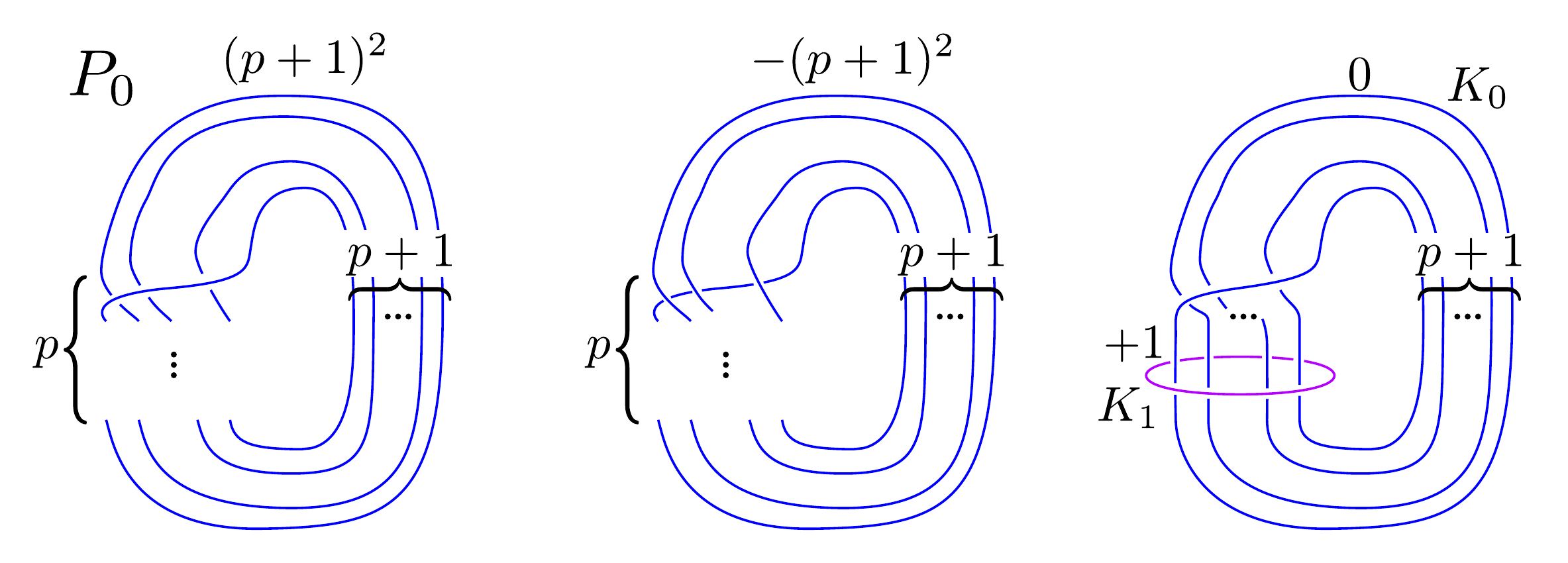}
	\caption{The characteristic sublink in $P_0$ represented by the colored $T(p+1,p)$ is carried through the Kirby calculus moves. Reversing orientation keeps the same component in the characteristic sublink. The anti-blow-up adds the $+1$-framed component $K_1$ to the characteristic sublink if and only if $\lk(K_1,K_0)=p+1$ is even.}
	\label{fig:Kirbyspin}
\end{figure}

Given a Stein handlebody diagram in standard form, Gompf provides a formula to determine the $\Gamma$-invariant for the contact structure induced on the boundary. Let $\xi_{\hb}$ denote the contact structure induced on the boundary of the Stein handlebody $W_\hb$. Recall that $P_1$ is a handlebody diagram for a $4$-manifold $W^*_\hb$ obtained by surgering $W_\hb$. The $2$-handles attached to $K_0$ and $K_1$ determine homology classes $\alpha_0$ and $\alpha_1$ respectively which form a basis for $H_2(W^*_\hb;\Z)$ ($\alpha_i$ is represented by a surface obtained by capping off a Seifert surface for $K_i$). By~\cite[Theorem~4.12]{Gompf2}, $\Gamma(\xi_\hb,s_0)$ is Poincar\'e dual to the cohomology class $j^*\rho \in H^2(Y;\Z)$, where $j\colon Y \to W^*_\hb$ is the inclusion and $\rho \in H^2(W^*_\hb;\Z)$ satisfies:
\begin{equation}\label{eq:computerho}
\left\{
\begin{array}{l}
\vphantom{\displaystyle\frac12} \langle \rho, \alpha_0 \rangle = \frac12 \lk(K_0, L' + K_0),\\
\vphantom{\displaystyle\frac12} \langle \rho, \alpha_1 \rangle = \frac12\left( \rot(K_1) + \lk(K_1, L' + K_0) \right),
\end{array}
\right.
\end{equation}
where $L'$ is the characteristic sublink of $W^*_\hb$ corresponding to $s_0$ (namely if $p+1$ is odd $L'=K_0$ and if $p+1$ is even $L'=K_0\cup K_1$). 
Note that, as observed in~\cite[Section~4]{Gompf2}, the link $L' + K_0$ is to be interpreted modulo 2; in particular, if $K_0 \subset L'$ then $L' + K_0 = L' \setminus K_0$.

$H^2(W^*_\hb;\Z)$ is Poincar\'e dual to $H_2(W^*_\hb,Y;\Z)$, which is generated by the classes of the meridian disks $D_0$ and $D_1$ to $K_0$ and $K_1$, respectively. 
Since $\langle \rho, \alpha_i \rangle = \PD(\rho)\cdot \alpha_i$, and $D_i$ is a geometric dual of $\alpha_i$, we have that
\[
\PD(\rho) = \frac12 \lk(K_0, L' + K_0)[D_0] + \frac12\left( \rot(K_1) + \lk(K_1, L' + K_0) \right)[D_1].
\]

We also have that $\PD(j^*(\rho)) = \partial(\PD(\rho))$ as in the following commutative diagram:
\[
\xymatrix{  H^2(W^*_\hb;\Z) \ar[r]^{j^*} \ar[d]^{\PD} & H^2(Y;\Z) \ar[d]^{\PD}\\ H_2(W^*_\hb,Y;\Z) \ar[r]^\partial & H_1(Y;\Z)   }
\]

The boundary map $\partial\colon H_2(W^*_\hb,Y;\Z) \to H_1(Y;\Z)$ sends $[D_0]$ to $[\partial D_0] = \mu \in H_1(Y;\Z)$. Using the fact that a presentation of $H_1(Y;\Z)$ can be obtained from the surgery diagram $P_1$, with generators $[\partial D_0]$ and $[\partial D_1]$ and relations presented by the intersection matrix $\left(\begin{smallmatrix}0 & p+1\\ p+1 & 1\end{smallmatrix}\right)$, we see that
\[
[\partial D_1] = -(p+1)[\partial D_0] = -(p+1)\mu \quad \textrm{and}\quad (p+1)^2\mu=0.
\]

Therefore
\begin{align*}
\Gamma(\xi_\hb,s_0) &= \PD(j^*\rho) = \partial \PD(\rho) = \frac12 \lk(K_0, L' + K_0)[\partial D_0] + \frac12\left( \rot(K_1) + \lk(K_1, L' + K_0) \right)[\partial D_1]\\
& = \frac12 \left( \lk(K_0, L' + K_0) -(p+1)\left( \rot(K_1) + \lk(K_1, L' + K_0) \right) \right) \mu.
\end{align*}

Finally, by counting the number of downward cusps minus the number of upward cusps in the Stein handlebody diagram of Figure~\ref{fig:ApStein}, we see that $\rot(K_1) = p-2$.

Let us now look at the two cases, according to the parity of $p+1$.

If $p+1$ is odd, $L' = K_0$, so we can take the empty link as a mod 2 representative of $L'+K_0$. Therefore,
\begin{align*}
\Gamma(\xi_\hb,s_0) &= \frac12 \left( \lk(K_0, \emptyset) -(p+1)\left( \rot(K_1) + \lk(K_1, \emptyset) \right) \right) \mu\\
& = -\frac{(p+1)(p-2)}{2}\mu = \frac{-p^2+p+2}{2}\mu.
\end{align*}

If $p+1$ is even $L' = K_0\cup K_1$, so we can take $K_1$ as a mod 2 representative of $L'+K_0$. Therefore,
\begin{align*}
\Gamma(\xi_\hb,s_0) &= \frac12 \left( \lk(K_0, K_1) -(p+1)\left( \rot(K_1) + \lk(K_1, K_1) \right) \right) \mu\\
& = \frac{p+1-(p+1)(p-2+1)}{2}\mu = \frac{-p^2+p+2}{2}\mu.
\end{align*}

Comparing this with the calculation from Proposition~\ref{thm:Gamma(xi,s0)} with $g=\frac{p(p-1)}{2}$ (the arithmetic genus of the singularity with link $T(p,p+1)$), we see that
\[
\Gamma(\xi,s_0) = \left(1-\frac{p(p-1)}{2}\right)\mu = \frac{-p^2+p+2}{2}\mu = \Gamma(\xi_{\hb},s_0).
\]
Therefore, by Matkovi\v{c}'s result, $\xi$ and $\xi_{\hb}$ are contactomorphic. Since $\xi$ has a unique minimal symplectic filling, $W_\hb$ must be this filling (up to symplectic deformation and symplectomorphism).
\end{proof}

For symplectic curves in the $\Bc_p$ family, we prove in~\cite[Theorem 6.5]{bigGS}, that there are exactly two relatively minimal symplectic embeddings of $C$ into closed symplectic manifolds up to symplectic isotopy: one embedding is into $\cptwo$ and the other is into $S^2\times S^2$. Consequently, there are exactly two minimal symplectic fillings of $(Y_C,\xi_C)$ (up to symplectic deformation and symplectomorphism): one is a rational homology ball and the other has $b_2=1$.
		
\begin{theorem}\label{p:Bp}
Fix $p > 1$, and let $C$ be the rational cuspidal curve of type $\Bc_p$. Then the unique rational homology ball filling of $(Y_C,\xi_C)$ is given by the Stein handlebody depicted in Figure~\ref{fig:BpStein}, and the unique filling of $(Y_C,\xi_C)$ with $b_2=1$ is given by the Stein handlebody of Figure~\ref{fig:BpStein2}.
\end{theorem}

\begin{figure}[h]
	\centering
	\includegraphics[scale=.75]{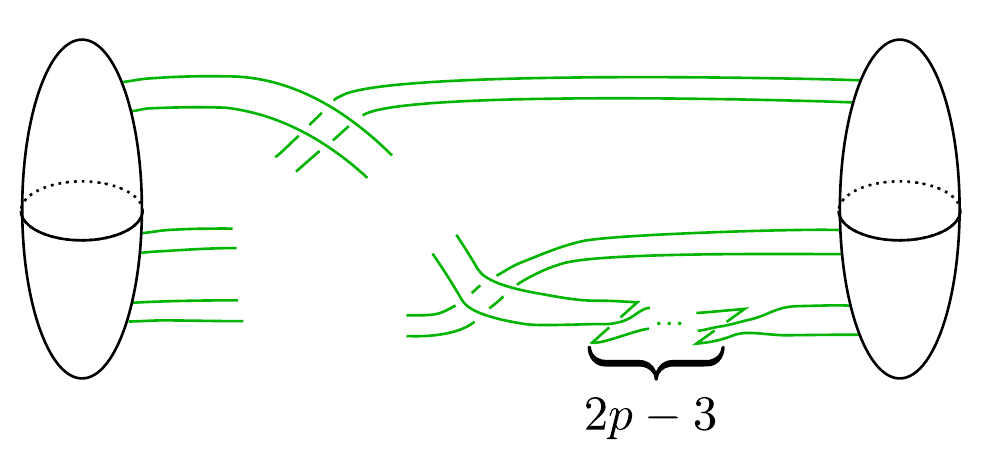}
	\caption{Stein handlebody diagram for the rational homology ball filling of $(Y_C,\xi_C)$ where $C$ is a $\Bc_p$-type rational unicuspidal curve. The attaching curve for the $2$-handle passes through the $1$-handle $2p$ times. Observe that for the Legendrian attaching curve for the $2$-handle, $\tb = 4(p-1)-(2p-3)-(2p-3)=2$ and $\rot = 2p-3$.}
	\label{fig:BpStein}
\end{figure}

\begin{figure}[h]
	\centering
	\includegraphics[scale=.75]{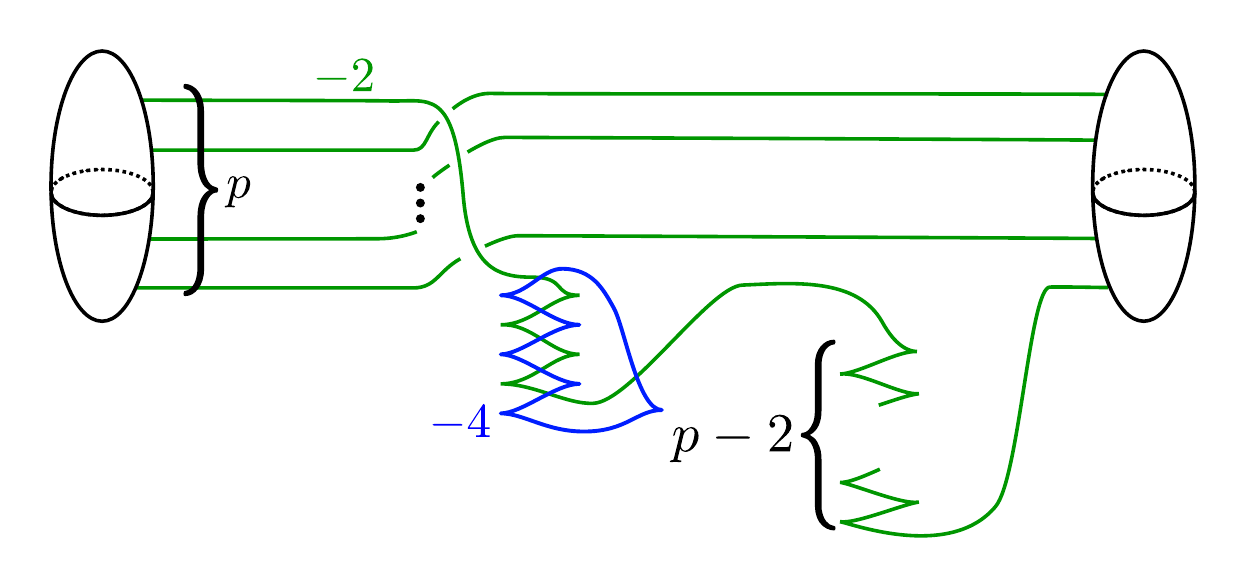}
	\caption{Stein handlebody diagram for the $b_2=1$ filling of $(Y_C,\xi_C)$ where $C$ is a $\Bc_p$-type rational unicuspidal curve. Observe that the blue Legendrian attaching curve has $\tb = -1$ and $\rot = p$, and the green Legendrian attaching curve has $\tb=-3$ and $\rot=2$.}
	\label{fig:BpStein2}
\end{figure}

The proof will follow the same method as Theorem~\ref{thm:ApStein}. A heuristic argument, similar to that we gave for the $\Ac_p$ family, for why the filling has a decomposition with one 1-handle and one 2-handle is found by considering $C \cup T$, where $T$ is the tangent to $C$ at the cusp.

\begin{proof}
	The boundary of the handlebody is $Y_C$ as seen by the sequence of Kirby calculus moves in Figure~\ref{fig:BpKirby}, which are explained in the caption.
	The isotopy from the bottom left diagram to the bottom right one can be thought of in terms of the following observation. As a link, the diagram on the bottom left is the obtained from the Hopf link by taking the $(2,1)$-cable of one component and the $(p,1)$-cable of the other. Since the link obtained by doing the $(p,1)$-cable on the Hopf link is symmetric, we can choose to ``discharge'' the cabling on the first component, so that cabling the two components separately is the same as cabling one component twice. One needs to keep track of the framing, using the writhe, to see that there are $(p-1)$ negative twists.
	
	We also track the class $\mu$ and the spin structure $s_0$ (represented by the non-empty characteristic sublink) through these diagrams. Since the (green) $+1$-unknot of the blow-up links the (black) knotted component $2p$ times, by~\cite[Pages~190--191]{GompfStipsicz}, the $+1$-framed unknot is included in the characteristic sublink corresponding to $s_0$. That is, in the last diagram the spin structure $s_0$ is represented by the characteristic sublink consisting of \emph{both} components of the surgery diagram.
	
	\begin{figure}
		\centering
		\includegraphics[scale=.58]{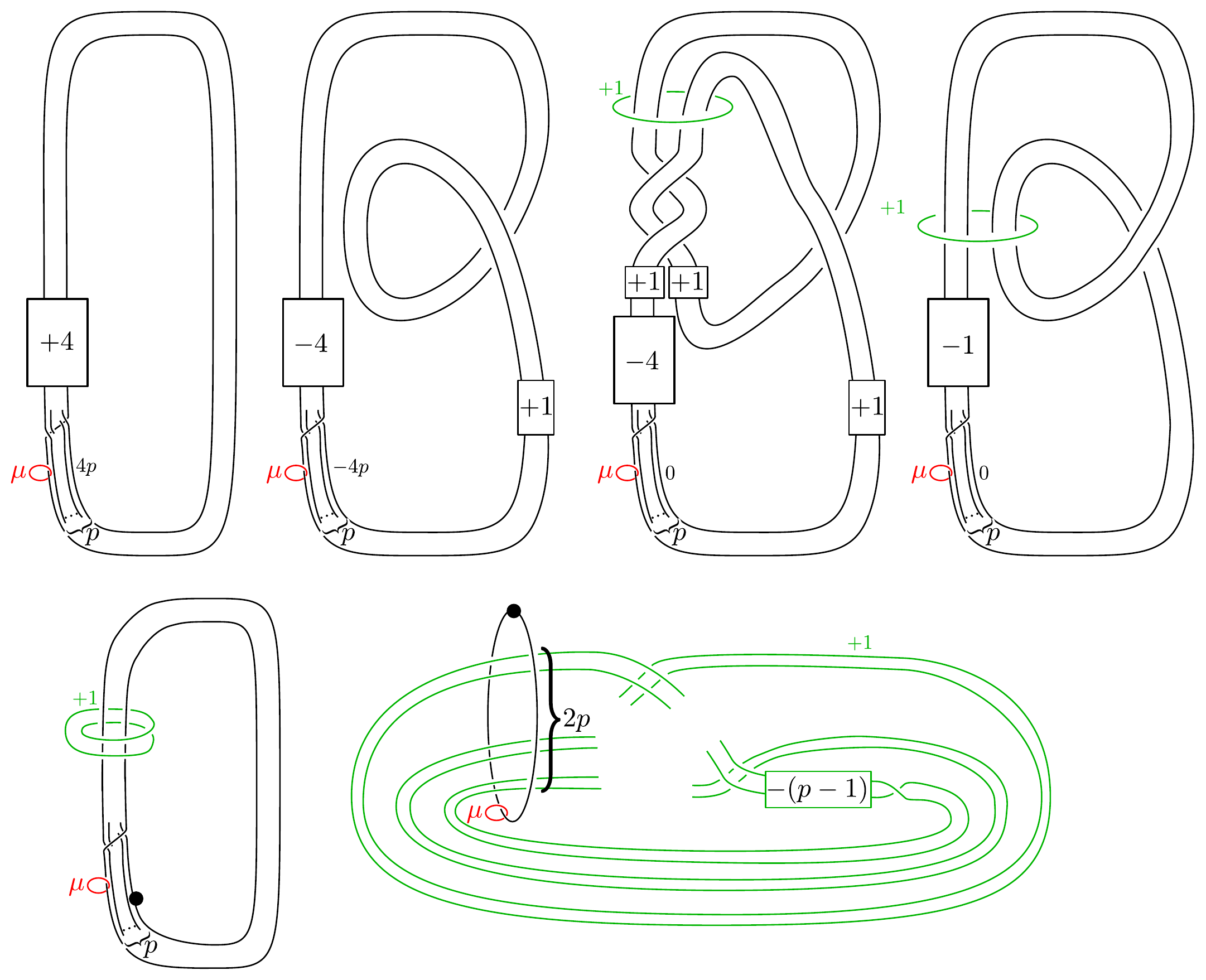}
		\caption{A Kirby calculus sequence from the surgery presentation of $Y$ to a rational homology ball. Along the top, going from the left most figure to the next corresponds to reversing orientation and performing an isotopy, then the next figure is related by an anti-blow-up, and the rightmost is related by an isotopy. The bottom left figure is related by an isotopy and switching the $0$-framed component to a dotted circle, and the bottom right figure is obtained by an isotopy to put it in standard notation. This bottom right figure is the same underlying smooth handlebody diagram as Figure~\ref{fig:BpStein} (the signed number of crossings between the two strands at the bottom is $-2(p-1)+1=-(2p-3)$, and the framing $+1=\tb-1$ for the Legendrian realization in Figure~\ref{fig:BpStein}).}
		\label{fig:BpKirby}
	\end{figure}

	Since the $\Bc_p$ curve has arithmetic genus $\frac{(p-1)(4p-2)}{2}$, using Theorem~\ref{thm:Gamma(xi,s0)}, we see that
	\[
	\Gamma(\xi_C,s_0) = (1-(p-1)(2p-1))\mu = (3p-2p^2)\mu.
	\]
	
	Let $\xi_\hb$ be the contact structure represented by the Legendrian surgery diagram of Figure~\ref{fig:BpStein}.
	We compare $\Gamma(\xi_C,s_0)$ and $\Gamma(\xi_{\hb},s_0)$, following the same computation as in the $\Ac_p$ case.
	
	As in the proof of Theorem~\ref{thm:ApStein}, call $W^*_\hb$ the surgery diagram on the top right of Figure~\ref{fig:BpKirby} (i.e. the one obtained from doing a zero-dot surgery on $W_\hb$), and call $K_0$ and $K_1$ the components of the link, where $K_0$ is the 0-framed component (black) and $K_1$ is the 1-framed component. With a slight abuse of notation, we also call $K_1$ the Legendrian attaching curve of Figure~\ref{fig:ApStein}, which we orient as going from left to right. We orient $K_0$ so that $\lk(K_0,K_1) = 2p$. Once again, we use the conventions from~\cite{Gompf2} to compute Thurston--Bennequin and rotation numbers in the presence of 1-handles.
	
	Using these conventions, we compute the classical invariants of $K_1$. There are $4(p-1)$ positive crossings coming from the $1/p^{\rm th}$ of twist on the left, $2p-3$ negative crossings coming from the 2-cabling on the right, and $4p-6$ cusps, all of which are upward-pointing. Summing up:
	
	\begin{align*}
	\tb(K_1) &= 4(p-1)-(2p-3)-(2p-3) = 2;\\
	\rot(K_1) &= {\textstyle -\frac{4p-6}2} = 2p-3.
	\end{align*}
	Note that $\tb(K_1) - 1 = +1$, so that the framing coming from the Legendrian surgery picture of Figure~\ref{fig:BpStein} agrees with the smooth surgery framing of Figure~\ref{fig:BpKirby}.
	
	Call $D_0$ and $D_1$ the co-cores of the handles associated to the surgeries along $K_0$ and $K_1$, respectively. The intersection matrix of $W^*_\hb$ is $\left(\begin{smallmatrix}0 & 2p\\ 2p & 1\end{smallmatrix}\right)$, so we have $[\partial D_0] = \mu \in H_1(Y)$ and $[\partial D_1] = -2p\mu \in H_1(Y)$.
	
	Recall that $s_0$ corresponds to the characteristic sublink $L' := K_0\cup K_1$ containing both components of the link in Figure~\ref{fig:BpKirby}. The same computation as in the proof of Theorem~\ref{thm:ApStein} now yields:
	
	\begin{align*}
	\Gamma(\xi_\hb,s_0) & =\PD(j^*\rho) = \partial \PD(\rho) = \frac12 \lk(K_0, L' + K_0)[\partial D_0] + \frac12\left( \rot(K_1) + \lk(K_1, L' + K_0) \right)[\partial D_1]\\
	&=\frac12 \lk(K_0,K_1)\mu +\frac12 (\rot(K_1)+\lk(K_1,K_1)) (-2p)\mu =p\mu-p(2p-3+1)\mu\\
	&=(3p-2p^2)\mu.
	\end{align*}
	
	Therefore, we have that $\Gamma(\xi_\hb, s_0) =\Gamma(\xi,s_0)$. That is, $\xi_\hb$ and $\xi$ induce the same \spinc structure on $Y$.
	Now, $-Y$ is obtained as $4p^2$-surgery along $T(p,4p-1)$, so $Y$ is a small Seifert fibered space. Moreover, since $4p^2 > 2g(T(p,4p-1)) - 2$, by~\cite[Proposition~4.1]{LiscaStipsicz-I} $Y$ is a also a Heegaard Floer L-space.
	Since both $\xi$ and $\xi_\hb$ are tight, Matkovi\v{c}'s result~\cite[Theorem~1.3]{Matkovic} implies that $\xi_\hb$ and $\xi$ are isotopic.
	
	For the filling with $b_2=1$, the proof goes similarly, comparing the $\Gamma$-invariant calculation for the Stein handlebody of Figure~\ref{fig:BpStein2} with that of Theorem~\ref{thm:Gamma(xi,s0)}. We leave this computation to the reader, but provide assistance with Figure~\ref{fig:BpKirby2}, which gives a Kirby calculus sequence between the surgery diagram for $Y_C$ as $-S^3_{4p^2}(T(p,4p-1))$ and one which is equivalent to Figure~\ref{fig:BpStein2} by and isotopy plus surgering a $0$-framed $2$-handle to a $1$-handle.
	\end{proof}
	
	\begin{figure}
		\centering
		\includegraphics[scale=1]{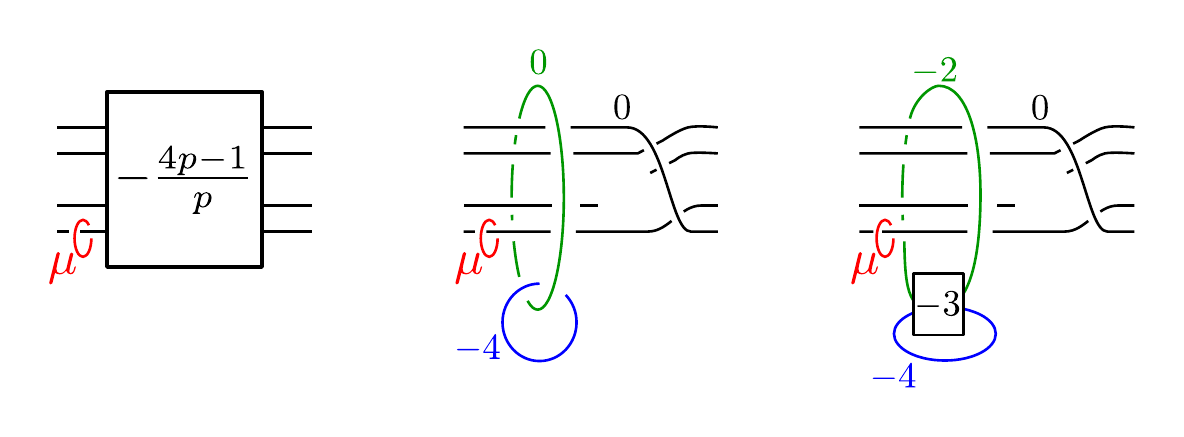}
		\caption{An abbreviated Kirby calculus sequence for $Y_C$ in the $\Bc_p$ case to help compare the $\Gamma$-invariant for the second filling. The right most diagram can be isotoped so that after exchanging the $0$-framed $2$-handle for a dotted circle and switching to the other $1$-handle notation, we get a diagram which is smoothly equivalent to Figure~\ref{fig:BpStein2}. To get from the center figure to the left, handle slide all $p$ strands the black curve over the blue curve, and then cancel the blue and green curves. To get from the center figure to the right, handle slide the green curve once over the blue curve. The spin structure represented by the non-empty characteristict sublink on the left corresponds to the spin structure represented by the characteristic sublink in the center and right diagrams which consists of the black component only when $p$ is even, and the black and blue components when $p$ is odd.}
		\label{fig:BpKirby2}
	\end{figure}

\section{Canonical contact structures in the exceptional $\Ec$ and $\Dc$ cases}\label{s:E3E6}
In this section, we focus on $(Y_C,\xi_C)$ for the two exceptional cases of rational cuspidal curves in $\cptwo$ with a single $T(a,b)$ cusp, with $(a,b) = (3,22)$ or $(a,b) = (6,43)$.
Our goal is to identify the contact manifolds in these two cases as the canonical contact manifolds arising as links of complex normal surface singularities. This will complete the proof of Theorem~\ref{thm:unicuspidal}.

In the proofs, we will make use of the Ozsv\'ath--Szab\'o contact invariant~\cite{OSz-contact} in Heegaard Floer homology~\cite{OSz-HF}.
The relevant properties of the theory are the non-vanishing of the contact invariant for fillable contact structures~\cite[Theorem~1.5]{OSz-contact}, the fact that large surgeries on torus knots are L-spaces (i.e. they have the smallest possible Heegaard Floer homology)~\cite[Proposition~4.1]{LiscaStipsicz-I}, and the absolute grading on Heegaard Floer homology~\cite{OSz-correctionterms}.

We first consider the contact 3-manifold $(Y_C,\xi_C)$ associated to the curve $\Ec$, which is associated to the triple $(3,22;8)$.
Then topologically, $Y_C = -S^3_{64}(T(3,22))$.
In particular, it is a small Seifert fibered manifolds with Seifert parameters $(-2;1/2,1/3,15/22)$;
see Figure~\ref{f:Seifert} for an explanation of the notation.

\begin{figure}
	\labellist
	\pinlabel $e_0$ at 136 46
	\pinlabel $-\frac{1}{r_1}$ at 14 7
	\pinlabel $-\frac{1}{r_2}$ at 50 7
	\pinlabel $-\frac{1}{r_3}$ at 86 7
	\endlabellist
	\centering
	\includegraphics[scale=1.2]{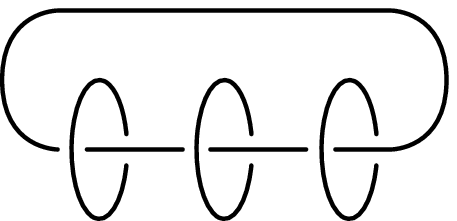}
	\caption{A surgery description for the small Seifert fibered manifold with parameters $(e_0;r_1,r_2,r_3)$.}
	\label{f:Seifert}
\end{figure}

Expanding the rational surgeries with their continued fraction expansions, we see that $Y_C$ is smoothly the boundary of the following plumbing graph.
\begin{equation}\label{e:64/7}
	\begin{aligned}
		\xygraph{
			!{<0cm,0cm>;<0.9cm,0cm>:<0cm,0.9cm>::}
			!~-{@{-}@[|(2.5)]}
			!{(0,0.8) }*+{\bullet}="c"
			!{(1.5,0.8) }*+{\bullet}="d"
			!{(3,0.8) }*+{\bullet}="e"
			!{(4.5,0.8) }*+{\bullet}="f"
			!{(6,0.8) }*+{\bullet}="x"
			!{(1.5,-0.7) }*+{\bullet}="h"
			!{(0,1.2) }*+{-2}
			!{(1.5,-1.1) }*+{-3}
			!{(1.5,1.2) }*+{-2}
			!{(3,1.2) }*+{-2}
			!{(4.5,1.2) }*+{-2}
			!{(6,1.2) }*+{-8}
			"c"-"d"
			"d"-"e"
			"e"-"f"
			"f"-"x"
			"d"-"h"
		}
	\end{aligned}
\end{equation}

The \emph{canonical} contact structure of a Seifert fibered space is the one arising as convex boundary of a plumbing of symplectic spheres according to a negative definite graph. Equivalently, it is the contact structure that arises as the link of a complex surface singularity whose minimal normal crossing resolution is that plumbing. Note that the graph~\eqref{e:64/7} does have a negative definite intersection form.

\begin{prop}\label{p:canonical322}
	When $C$ is the rational cuspidal curve with a single $T(3,22)$ cusp and self-intersection number $64$, the contact manifold $(Y_C,\xi_C)$ is contactomorphic to the canonical contact structure on $Y_C$ associated with the plumbing~\eqref{e:64/7}. It can be presented as a Legendrian surgery diagram as in Figure~\ref{fig:E3Stein}.
\end{prop}

\begin{figure}
	\centering
	\includegraphics[scale=.75]{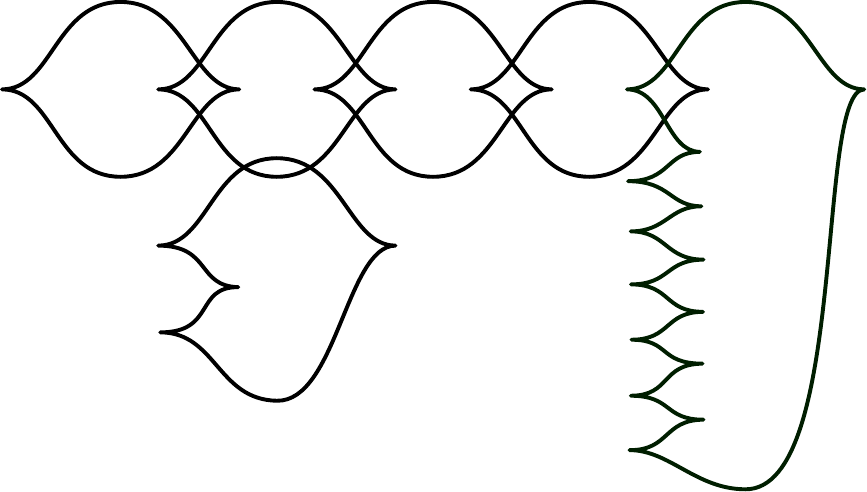}
	\caption{Stein handlebody for the plumbing inducing the canonical contact structure on its boundary.}
	\label{fig:E3Stein}
\end{figure}

\begin{proof}
	Tosun~\cite[Theorem~1.1(b)]{Tosun} classified tight contact structures on small Seifert fibered spaces with $e_0=-2$ and $r_1+r_2+r_3<2$, showing that all such tight contact structures all arise from a Legendrian surgery picture associated to the plumbing~\eqref{e:64/7}, as in Figure~\ref{f:Legendrian-link-plumbing}.
	Since $(Y_C,\xi_C)$ is fillable, $\xi_C$ is tight~\cite{EliashbergGromov}. As explained above $Y_C$ is the small Seifert fibered space with $(e_0;r_1,r_2,r_3)=(-2;1/2,1/3,15/22)$. Since, $r_1+r_2+r_3 = \frac{50}{33} < 2$, Tosun's classification result applies.
	
	\begin{figure}
		\centering
		\includegraphics[scale=.75]{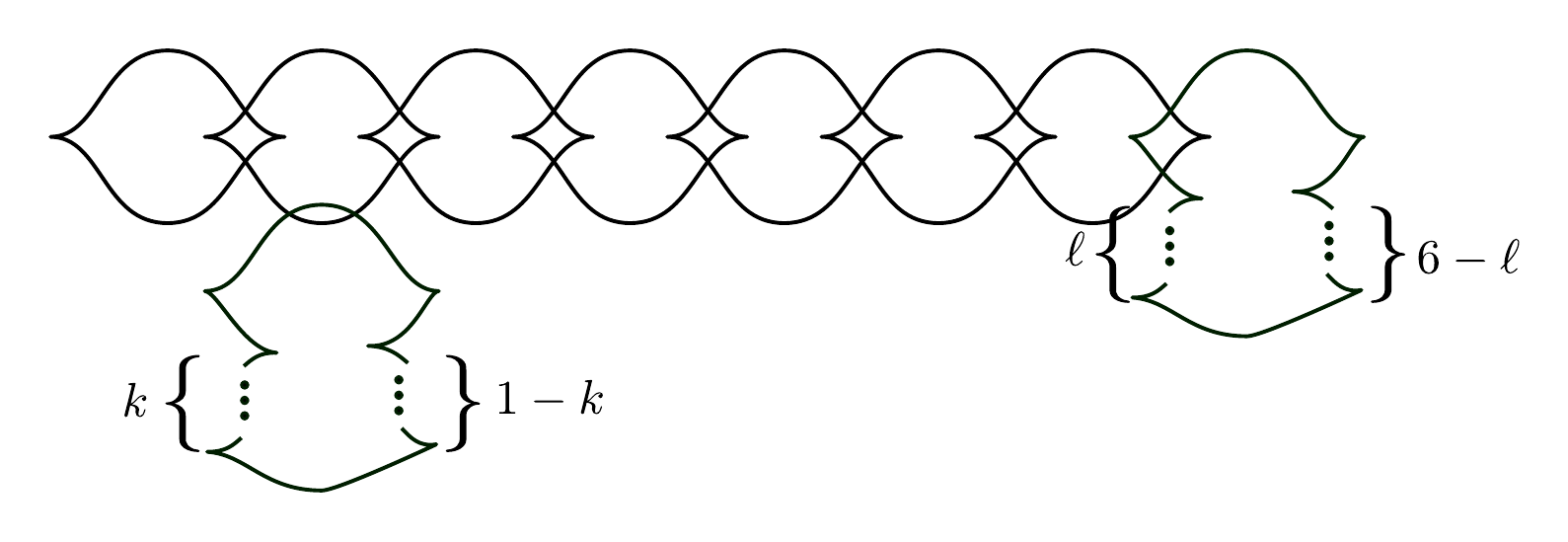}
		\caption{Allowing some number of the stabilizations of unknotted components of Figure~\ref{fig:E3Stein} to be negative provides Legendrian surgery diagrams for all other tight contact structures on this 3-manifold by a result of Tosun.}
		\label{f:Legendrian-link-plumbing}
	\end{figure}
	
	We will prove that the canonical contact structure is the unique tight contact structure (on the Seifert fibered space at hand) whose associated spin$^c$ structure $\ft$ extends to a rational homology ball.
	We do this by computing the Heegaard Floer correction term associated to the \spinc classes of all tight contact structures, using Tosun's classification.
	
	We begin by observing that, since $(Y_C,\xi_C)$ has a rational homology ball filling, then the correction term $d(Y_C,\ft)$ vanishes.
	We set out to prove that this characterises the canonical contact structure (up to conjugation).

	Since torus knots are L-space knots (in fact, lens space knots~\cite{Moser}), and the surgery coefficient is 64, which is larger than $2g(T(3,22))-2 = 40$, $Y$ is an L-space~\cite[Proposition~4.1]{LiscaStipsicz-I}. (Now we could use Ghiggini's classification result from~\cite{Ghiggini-Lclassification} instead of Tosun's, as mentioned in the introduction.)
	
	Since every tight contact structure on $Y_C$ arises as the boundary of a Stein handlebody as in Figure~\ref{f:Legendrian-link-plumbing}, we can use this Stein filling to compute the degree of the contact invariant for each contact structure. The contact invariant is also the generator of the associated tower in $\HFplus(Y)$, so its degree is the correction term associated to the corresponding contact structure.
	
	The possible realizations of the link $L$ are distinguished by the rotation numbers $s, t$ of the components with framing $-3$ and $-8$:
	these can take values $\pm1$ and $\pm6$, $\pm4$, $\pm2$, or $0$, respectively.
	All other rotation numbers vanish, since the framing is $-2$.
	Moreover, reversing both signs corresponds to conjugation.
	Let $J_{s,t}$ be the complex structure on the Stein plumbing $P$ associated to Legendrian realization of the link with rotation numbers $s$ and $t$, and $\xi_{s,t}$ the associated contact structure on $Y_C$.
	
	Since $c_1(J_{s,t})$ evaluates on a sphere in the plumbing as the rotation number of the corresponding unlink, we can compute $c_1^2(J_{s,t})$ from the intersection matrix $Q$ of the plumbing $P$.
	$Q$ and its inverse $Q^{-1}$ are:
	\[
	Q = \left(
	\begin{matrix}
		-2 & 1 & 1 & 1 & 0 & 0\\
		1 & -2 & 0 & 0 & 0 & 0\\
		1 & 0 & -3 & 0 & 0 & 0\\
		1 & 0 & 0 & -2 & 1 & 0\\
		0 & 0 & 0 & 1 & -2 & 1\\
		0 & 0 & 0 & 0 & 1 & -8\\
	\end{matrix}
	\right), \qquad
	Q^{-1} = -\frac{1}{64}\left(
	\begin{matrix}
		132 & 66 & 44 & 90 & 48 & 6\\
		66 & 65 & 22 & 45 & 24 & 3\\
		44 & 22 & 36 & 30 & 16 & 2\\
		90 & 45 & 30 & 105 & 56 & 7\\
		48 & 24 & 16 & 56 & 64 & 8\\
		6 & 3 & 2 & 7 & 8 & 9
	\end{matrix}
	\right).
	\]
	Then $c_1^2(J_{s,t}) = vQ^{-1}v^t$, where $v = (0,0,s,0,0,t)$.
	
	For convenience, call $P^* = P\setminus B^4$, a cobordism from $S^3$ to $Y$.
	Then,
	\[
	\frac{c_1^2(J_{s,t}) - 2\chi(P^*) - 3\sigma(P^*)}4 = \frac{384-36s^2-4st-9t^2}{256},
	\]
	and it is easily verified that the minimum of this function, with $s = \pm1$ and $t\in\{0,\pm2,\pm4,\pm6\}$ is $0$, attained only at $(s,t) = \pm(1,6)$;
	that is, only at the canonical contact structure and its conjugate.
\end{proof}

\begin{remark}
	We note that when understanding symplectic fillings, identifying a contact structure up to contactomorphism and conjugation suffices. This is because, for every symplectic filling of a contact manifold, the same filling with the conjugate almost complex structure (negating the symplectic form) gives a symplectic filling of the conjugate contact manifold. Therefore the fillings of a contact manifold and its conjugate are in bijective correspondence.
\end{remark}

Next, we consider the second exceptional case, where $C$ has a unique singularity of type $(6,43)$ and self-intersection number $16^2=256$.
In this case, $Y_C = -S^3_{256}(T(6,43))$, so $Y_C$ is a small Seifert fibered space with Seifert parameters $(-2;1/2,1/6,36/43)$.
A negative definite plumbing with boundary $Y_C$ is the following.
\begin{equation}\label{e:256/x}
	\begin{aligned}
		\xygraph{
			!{<0cm,0cm>;<0.9cm,0cm>:<0cm,0.9cm>::}
			!~-{@{-}@[|(2.5)]}
			!{(0,0.8) }*+{\bullet}="c"
			!{(1.5,0.8) }*+{\bullet}="d"
			!{(3,0.8) }*+{\bullet}="e"
			!{(4.5,0.8) }*+{\bullet}="f"
			!{(6,0.8) }*+{\bullet}="x"
			!{(7.5,0.8) }*+{\bullet}="y"
			!{(9,0.8) }*+{\bullet}="z"
			!{(10.5,0.8) }*+{\bullet}="a"
			!{(1.5,-0.7) }*+{\bullet}="h"
			!{(0,1.2) }*+{-2}
			!{(1.5,-1.1) }*+{-6}
			!{(1.5,1.2) }*+{-2}
			!{(3,1.2) }*+{-2}
			!{(4.5,1.2) }*+{-2}
			!{(6,1.2) }*+{-2}
			!{(7.5,1.2) }*+{-2}
			!{(9,1.2) }*+{-2}
			!{(10.5,1.2) }*+{-8}
			"c"-"d"
			"d"-"e"
			"e"-"f"
			"f"-"x"
			"d"-"h"
			"x"-"y"
			"y"-"z"
			"z"-"a"
		}
	\end{aligned}
\end{equation}

\begin{prop}\label{p:canonical643}
	When $C$ is the rational cuspidal curve with a single $T(6,43)$ cusp and self-intersection number $256$, the contact manifold $(Y_C,\xi_C)$ is contactomorphic to the canonical contact structure on $Y_C$ associated with the plumbing~\eqref{e:256/x}. It can be presented as a Legendrian surgery diagram as in Figure~\ref{fig:E6Stein}.
\end{prop}

\begin{figure}
	\centering
	\includegraphics[scale=.75]{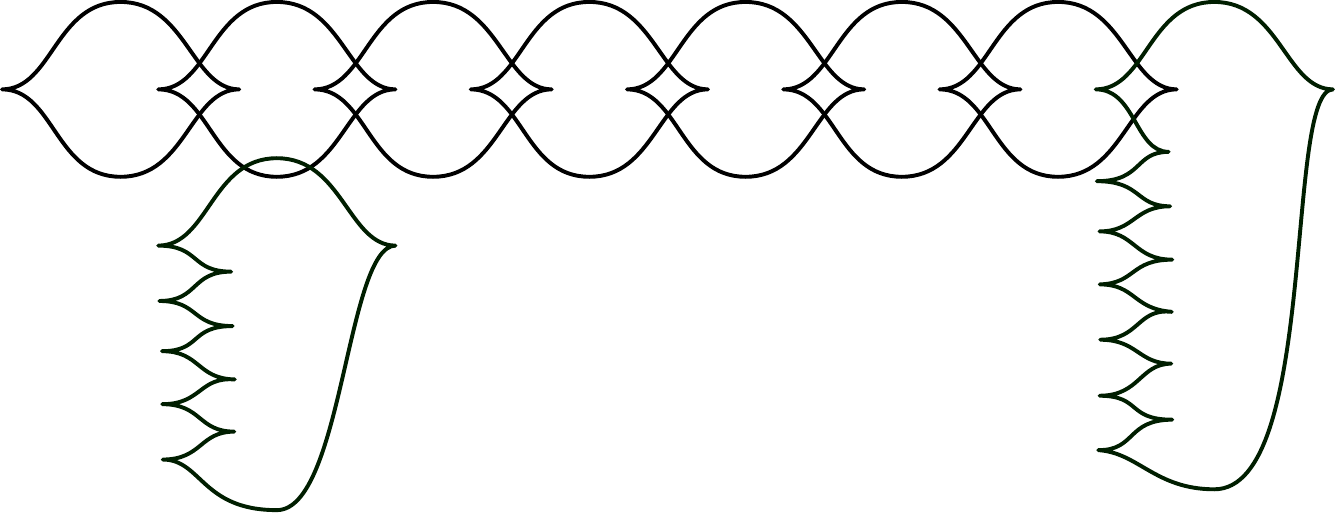}
	\caption{Stein handlebody for the plumbing inducing the canonical contact structure on its boundary.}
	\label{fig:E6Stein}
\end{figure}

\begin{proof}
	Since the proof is very similar to that of Proposition~\ref{p:canonical322}, we omit the details that would be repeated and only include the calculations where the two cases differ.
	
	Again, in this case we have $r_1+r_2+r_3 = \frac{194}{129} < 2$, so Tosun's classification result still applies. Moreover, the surgery coefficient is again sufficiently large, so $Y_C$ is again an L-space.
	
	In the plumbing graph there are still only two rotation numbers $s$ and $t$ that vary, corresponding to the vertices of weight $-6$ and $-8$, respectively. $s$ takes values in $\{\pm2, 0\}$, while $t$ takes values in $\{\pm4, \pm2, 0\}$.
	
	The intersection matrix $Q$ associated to the plumbing is:
	\[
	Q = \left(
	\begin{matrix}
		-2 & 1 & 1 & 1 & 0 & 0 & 0 & 0 & 0\\
		1 & -2 & 0 & 0 & 0 & 0 & 0 & 0 & 0\\
		1 & 0 & -6 & 0 & 0 & 0 & 0 & 0 & 0\\
		1 & 0 & 0 & -2 & 1 & 0 & 0 & 0 & 0\\
		0 & 0 & 0 & 1 & -2 & 1 & 0 & 0 & 0\\
		0 & 0 & 0 & 0 & 1 & -2 & 1 & 0 & 0\\
		0 & 0 & 0 & 0 & 0 & 1 & -2 & 1 & 0\\
		0 & 0 & 0 & 0 & 0 & 0 & 1 & -2 & 1\\
		0 & 0 & 0 & 0 & 0 & 0 & 0 & 1 & -8\\
	\end{matrix}
	\right),
	\]
	and its inverse is
	\[
	Q^{-1} = -\frac1{256}\left(
	\begin{matrix}
		516 & 258 & 86 & 432 & 348 & 264 & 180 & 96 & 12\\
		258 & 257 & 43 & 216 & 174 & 132 &  90 & 48  & 6\\
		86 &  43 &  57 &  72 &  58 &  44 &  30 &  16 &   2\\
		432 & 216 &  72 & 576 &  464 & 352 & 240 & 128 &  16\\
		348 & 174 &  58 & 464 & 580 & 440 & 300 & 160 &  20\\
		264 & 132 &  44 & 352 & 440 & 528 & 360 & 192 &  24\\
		180 &  90 &  30 & 240 & 300 & 360 & 420 & 224 &  28\\
		96 &  48 &  16 & 128 & 160 & 192 & 224 & 256 &  32\\
		12 &   6 &   2 &  16&   20 &  24 &  28 &  32 &  36
	\end{matrix}
	\right).
	\]
	Therefore, in the same notation as in the proof of Proposition~\ref{p:canonical322},
	\[
	\frac{c_1^2(J_{s,t}) - 2\chi(P^*) - 3\sigma(P^*)}4 = \frac{2304-57s^2-4st-36t^2}{256},
	\]
	and the minimum of this function as $s$ and $t$ vary in the corresponding ranges is $0$, attained only at $(s,t) = \pm(4,6)$;
	that is, only at the canonical contact structure and its conjugate.
\end{proof}


\section{Background on symplectic configuration of curves} \label{s:background}
We will briefly collect some background related to rational cuspidal curves and embedding classifications that we will use repeatedly in the remaining sections of the paper. For further details, see~\cite[Sections~2 and~3]{bigGS} (see also~\cite{Li}).

By~\cite{Wall}, every curve singularity can be resolved by blowing up (sufficiently many times), and the diffeomorphism type of the link determines the topology of the resolution.
There are two natural stopping points when resolving a singularity:
the \emph{minimal resolution} is the smallest resolution such that the proper transform $\widetilde C$ of $C$ is smooth;
the \emph{normal crossing resolution} is the smallest resolution such that the total transform $\overline C$ of $C$ is a normal crossing divisor, i.e. all singularities are double points.

The \emph{multiplicity} $m_p$ of a singular point $p$ of a curve $C$ is the minimal intersection of a germ of a divisor $D$ at $p$ with $C$. In terms of the resolution, $m_p$ is the algebraic intersection number of the exceptional divisor and the proper transform after blowing up at $p$, and we have $[\widetilde C] = [C]-m_p[E]$.
The \emph{multiplicity sequence} of a singularity $p$ is defined as the sequence of multiplicities of the curve at $p$ and each of its proper transforms in the sequence of blow-ups leading to the minimal resolution of the singularity.

Recall that an isolated singularity of a curve at $q$ can be smoothed to its Milnor fiber.
Let $\mu(q)$ denote the first Betti number of the Milnor fiber. If $q$ has $r$ local branches, define $\delta(q)$ by $2\delta(q) = \mu(q) + r - 1$. 
If $q$ is a cusp, then $\delta(q)$ is the genus of its Milnor fiber.
It follows from the adjunction formula that for a singularity with multiplicity sequence $[m_1,\dots, m_n]$
\begin{equation}\label{e:deltamulti}
	\delta(p) = \frac12\sum m_j(m_j-1).
\end{equation}
There are two different types of notions of the genus of a singular curve $C$: the \emph{geometric genus} and the \emph{arithmetic genus}.
The geometric genus $p_g(C)$ is the genus of the proper transform of the curve in the minimal resolution; by definition, rational curves have geometric genus zero. The arithmetic genus $p_a(C)$ is given by
	\begin{equation}\label{e:degree-genus}
		p_a(C) := p_g(C) + \sum_{p \in \Sing(C)} \delta(p)
	\end{equation}

Next, we will state various results about embedded surfaces in symplectic $4$-manifolds that we used heavily in our embedding classification results in~\cite{bigGS}, beginning with the following important result of McDuff.

\begin{theorem}[\cite{Mc}] \label{thm:mcduff}
	If $(X,\omega)$ is a closed symplectic $4$-manifold and $C_0\subset X$ is a smooth symplectic sphere of self-intersection number $+1$, then there is a symplectomorphism of $(X,\omega)$ to a symplectic blow up of $(\cptwo,\lambda\omega_{\rm{FS}})$ for some $\lambda > 0$, such that $C_0$ is identified with $\cpone$.
\end{theorem}

This theorem motivates our focus on surfaces embedded in $\cptwo\#N\cptwobar$. We will use the standard basis $h,e_1,\dots, e_N$ for $H_2(\cptwo \# N \cptwobar)$ with $h^2=1$ and $e_i^2=-1$. The following is a useful lemma to find embedded exceptional spheres which intersect a given collection of symplectic surfaces positively.

\begin{lemma}[\cite{Mc}, {\cite[Lemma~3.5]{bigGS}}] \label{l:blowdown}
Suppose $\mathcal{C}$ is a configuration of positively intersecting symplectic surfaces in $\cptwo \# N\cptwobar$. Let $e_{i_1},\dots,e_{i_\ell}$ be exceptional classes which have non-negative algebraic intersections with each of the symplectic surfaces in the configuration $\mathcal{C}$. Then there exist disjoint exceptional spheres $E_{i_1},\dots, E_{i_\ell}$ representing the classes $e_{i_1},\dots, e_{i_\ell}$ respectively such that any geometric intersections of $E$ with $\mathcal{C}$ are positive.
\end{lemma}

Using the previous lemma, information about an embedding of surfaces in $\cptwo\# N\cptwobar$ can be reduced to symplectic isotopy classes of curve configurations in $\cptwo$ together with the information of the homology classes represented by the surface components. Using the adjunction formula, we have the following restrictions on the homology classes that can represent a symplectic sphere in $\cptwo\# N\cptwobar$.

\begin{lemma}[{\cite[Lemma~3.7]{bigGS}}]\label{l:hom}
	Suppose $\Sigma$ is a smooth symplectic sphere in $\cptwo\#N\cptwobar$ intersecting $\cpone$ non-negatively. Then writing $[\Sigma]=a_0h+a_1e_1+\dots+a_Ne_N$ (so $a_0\geq 0$), we have:
	\begin{enumerate}
		\item \label{adj} $\sum (a_i^2+a_i)=2+a_0^2-3a_0$.
		\item \label{e:exceptional-disjoint} If $a_0=0$, there is one $i_0$ such that $a_{i_0}=1$ and all other $a_i\in \{0,-1\}$.
		\item \label{sgn} If $a_0\neq 0$, then for all $i\geq 1$, $a_i\leq 0$.
	\end{enumerate}
	Some particular cases which we will use often are:
	\begin{enumerate}\setcounter{enumi}{3}
		\item If $a_0=1$ or $a_0=2$, $a_i\in \{0,-1\}$ for all $1\leq i\leq N$.
		\item If $a_0=3$, then there exists a unique $i_0$ such that $a_{i_0}=-2$, and $a_i\in \{0,-1\}$ for all other $i$.
	\end{enumerate}
	The self-intersection number of $\Sigma$ can be used to compute how many $a_i$ have coefficient $0$ versus $-1$.
\end{lemma}

In the following lemmas from~\cite{bigGS}, $C_i$ and $C_j$ are smooth symplectic spheres in $\cptwo\#N\cptwobar$ such that $[C_i]\cdot h=[C_j]\cdot h=0$. These are easy consequences of Lemma~\ref{l:hom} and the hypothesized intersection relations.

\begin{lemma}[{\cite[Lemmas~3.8 and 3.10]{bigGS}}]\label{l:consecutive} \label{l:share2}
	If $[C_i]\cdot[C_j]=1$ (and $[C_i]\cdot h=[C_j]\cdot h=0$), there is exactly one exceptional class $e_i$ which appears with non-zero coefficient in both $[C_i]$ and $[C_j]$. The coefficient of $e_i$ is $+1$ in one of $[C_i],[C_j]$ and $-1$ in the other.
	
	If $[C_i]\cdot [C_j]=0$, then either there is no exceptional class which appears with non-zero coefficients in both, or there are exactly two exceptional classes $e_m$ and $e_n$ appearing with non-zero coefficients in both. One of these classes $e_m$ has coefficient $-1$ in both $[C_i]$ and $[C_j]$ and the other $e_n$ appears with coefficient $+1$ in one of $[C_i]$ or $[C_j]$ and coefficient $-1$ in the other.
\end{lemma}

\begin{lemma}[{\cite[Lemma~3.9]{bigGS}}]\label{l:pos}
	If $e_m$ appears with coefficient $+1$ in $[C_i]$ then it does not appear with coefficient $+1$ in $[C_j]$.
\end{lemma}

\begin{lemma}[{\cite[Lemmas~3.11 and 3.12]{bigGS}}]\label{l:2chain}\label{l:2chainfix}
	Suppose $C_1,\dots, C_k$ are a linear chain of symplectic spheres of self-intersection $-2$ disjoint from $\cpone$ in $\cptwo\# N\cptwobar$ ($[C_i]\cdot[C_{j}]=1$ $|i-j|=1$ and $0$ otherwise). Then the homology classes are given by one of the following two options up to re-indexing the exceptional classes:
	\begin{enumerate}[label=(\Alph*)]
		\item \label{i:to} $[\Sigma_i]=e_i-e_{i+1}$ for $i=1,\dots, k$.
		\item \label{i:from} $[\Sigma_i]=e_{i+1}-e_i$ for $i=1,\dots, k$.
	\end{enumerate}
	The homology class of any surface disjoint from the chain has the same coefficient for $e_1,\dots, e_{k+1}$. 
	
	If the chain is attached to another symplectic sphere $C_0$ which does intersect $\cpone$, option~\ref{i:from} can only occur if $e_2,\dots, e_{k+1}$ all appear with coefficient $-1$ in $[C_0]$. In particular if $[C_0]\cdot h=1$, option~\ref{i:from} can only occur if $[C_0]^2\leq 1-k$.
\end{lemma}


\section{Rational blow-down relations for fillings of unicuspidal contact manifolds}\label{s:qbd}
	
	Next we study the relationships between different fillings of contact manifolds $(Y_C,\xi_C)$, where $C$ is a symplectic curve in $\cptwo$ with a single $T(a,b)$ cusp. The relationship we will focus on is symplectic \emph{rational blow-down} (in its most general sense) which means replacing a plumbing of symplectic spheres with a symplectic rational homology ball filling. The first examples of rational blow-down were introduced by Fintushel and Stern~\cite{FintushelStern}, shown to be symplectic operations by Symington~\cite{Symington1} and were generalized in~\cite{Park,Symington2,SSW}. In~\cite{BhupalStipsicz}, Bhupal and Stipsicz show that these examples are all the plumbings which can be symplectically rationally blown down, and explicitly list out all families of non-linear plumbings which admit symplectic rational blow-downs.
	
	Linear plumbings which can be symplectically rationally blown down have boundaries which are the lens spaces $L(p^2,pq-1)$. The associated plumbing graphs for this family all arise recursively as an iterated ``$2$-expansion'' of $(-4)$, where a $2$-expansion of a linear plumbing with weights $(-a_1,\dots, -a_{n-1})$ can be either $(-a_1-1,\dots, -a_{n-1},-2)$ or $(-2,-a_1,\dots, -a_{n-1}-1)$. This can be understood as starting with a $-4$-sphere and a $-1$-sphere intersecting at two points transversally, and iteratively blowing up at one of the two intersection points on a $-1$-sphere, and then looking at the linear plumbing that results from removing the $-1$-sphere after all of the blow-ups. Observe that ``interior'' vertices cannot be changed by 2-expansions. In particular, an interior vertex of square strictly less than $-2$ survives after 2-expansions. 
	
	
	Now we proceed to study when different fillings of $(Y_C,\xi_C)$ are related by sequences of symplectic rational blow-downs.
	Note that among the curves in $\cptwo$ with a single $T(a,b)$ cusp, the two exceptional cases are the only cases where the symplectic rational blow-down relations have not already been established in prior work. In the $\Ac_p$ family, there is a unique minimal symplectic filling so there are no pairs to relate. In the $\Bc_p$ case, there are two fillings which are related by a single rational blow-down of a $-4$-sphere~\cite[Proposition 6.6]{bigGS}. The Fibonacci families have $Y_C$ a lens space or connected sum of lens spaces. Symplectic rational blow-down relations between lens spaces were established in~\cite{BhupalOzbagci}. The minimal symplectic fillings of the connected sum of two lens spaces are all Stein because the contact manifolds are planar~\cite{Wendl}. Thus they are boundary sums of Stein fillings of the two lens space summands~\cite{Eliash}, so the results of~\cite{BhupalOzbagci} apply in this case as well.
	
	We begin with the first exceptional case where $C$ has a cusp of type $T(3,22)$ and self-intersection number $64$. The minimal symplectic fillings of $(Y_C,\xi_C)$ are each obtained as the complement of a concave neighborhood of an embedding of $C$ (or any of its resolutions) into a closed symplectic $4$-manifold. In~\cite{bigGS} we classified all relatively minimal symplectic embeddings of the minimal resolution of $C$.
	The minimal resolution of $C$ in this case is described by the graph below where the triple edge indicates a tangency of multiplicity three between the $+1$-sphere and the $-1$-sphere.
	
	\[
	\xygraph{
		!{<0cm,0cm>;<1.25cm,0cm>:<0cm,1cm>::}
		!{(0,0) }*+{\bullet^{-2}}="c"
		!{(1,0) }*+{\bullet^{-2}}="l3a"
		!{(2,0) }*+{\bullet^{-2}}="l3b"
		!{(3,0) }*+{\bullet^{-2}}="l3c"
		!{(4,0) }*+{\bullet^{-2}}="l3d"
		!{(5,0) }*+{\bullet^{-2}}="l3e"
		!{(-1,0) }*+{\bullet^{-1}}="l2a"
		!{(-2,0) }*+{\bullet^{+1}}="l2b"
		"c"-"l2a"
		"c"-"l3a"
		"l2a"-"l2b"
		"l3a"-"l3b"
		"l3b"-"l3c"
		"l3c"-"l3d"
		"l3d"-"l3e"
		"l2a"-@/_.1cm/"l2b"
		"l2a"-@/^.1cm/"l2b"
	}
	\]
	
	In~\cite[Section 6]{bigGS}, we prove that there are exactly three symplectic isotopy classes of relatively minimal embeddings of this minimal resolution into closed symplectic manifolds. The three isotopy classes are distinguished by the maps induced by the embeddings on second homology which are given by the three possibilities below. Note that in each case the embedding is into a blow-up of $\cptwo$ and we use the standard generators $h,e_0,\dots, e_{N-1}\in H_2(\cptwo\#N\cptwobar;\Z)$ represented by $\cpone\subset \cptwo$ and the exceptional spheres. The vertices in the graph above correspond to generators of the second homology of the concave neighborhood of the minimal resolution, and the three possibilities for their images in $H_2(\cptwo\#N\cptwobar;\Z)$ under different symplectic embeddings are (with $N=13,8,7$ respectively):
	\begin{align*}
		&h,3h-2e_0-e_1-e_2-e_3-e_4-e_5-e_6, e_1-e_7,e_7-e_8,e_8-e_9,e_9-e_{10},e_{10}-e_{11},e_{11}-e_{12},\\
		&h,3h-2e_0-e_1-e_2-e_3-e_4-e_5-e_6, e_1-e_7, e_2-e_1, e_3-e_2, e_4-e_3, e_5-e_4, e_6-e_5,\\
		&h,3h-2e_0-e_1-e_2-e_3-e_4-e_5-e_6, e_0-e_1, e_1-e_2, e_2-e_3, e_3-e_4, e_4-e_5, e_5-e_6.\\
	\end{align*}
	
	The fillings of $(Y_C,\xi_C)$ are the complements of these embeddings, which have $b_2=N-7$. Therefore we have three fillings, $V_6$, $V_1$, and $V_0$, where the indices are chosen so that $b_2(V_k) = k$. As a corollary to Proposition~\ref{p:canonical322}, $V_6$ is symplectic deformation equivalent to the negative definite symplectic plumbing~\eqref{e:64/7}.
	
	\begin{prop}\label{p:322blowdowns}
		There is a linear symplectic rational blow-down from $V_6$ to $V_1$ and a non-linear symplectic rational blow-down from $V_6$ to $V_0$, but no rational blow-down from $V_1$ to $V_0$.
	\end{prop}

	Note that any sequence of symplectic rational blow-downs will yield a sequence of symplectic fillings. Since we have a classification of the fillings of $(Y_C,\xi_C)$, it suffices to understand when any pair of these symplectic fillings is related by a single symplectic rational blow-down.

	\begin{proof}[Proof of Proposition~\ref{p:322blowdowns}]
		It is apparent from the plumbing description of $V_6$, that there is a linear plumbing of symplectic spheres with self-intersections $(-8,-2,-2,-2,-2)$. This is one of Fintushel--Stern's original rational blowdowns $L(p^2,p-1)$ with $p=6$. Rationally blowing down this chain yields a symplectic filling with $b_2=1$, so it is necessarily $V_1$.
		
		Since $V_6$ itself is a plumbing, its replacement by the rational homology ball filling $V_0$ is itself a rational blow-down relation, and indeed this plumbing coincides with the plumbing of Figure 1(f) with $q=2$ in~\cite{BhupalStipsicz}.

		We now show that there is no rational blow-down from $V_1$ to $V_0$. There is a unique plumbing with $b_2=1$ which can be rationally blown down which is a $-4$-sphere. The non-torsion part of $H_2(V_1)$ is generated by the orthogonal complement of the classes in the corresponding embedding, namely $3e_0 - e_1 - \dots - e_7$. Therefore the intersection form of $V_1$ is $\langle -16 \rangle$, so no homology class has self-intersection $-4$. Thus there can be no rational blow-down from $V_1$ to $V_0$.
	\end{proof}

	Next we consider the second exceptional case, where $C$ has a singularity of type $(6,43)$ and self-intersection number $16^2$. Again, the minimal symplectic fillings of $(Y_C,\xi_C)$ are precisely the complements of the relatively minimal symplectic embeddings of a given resolution of $C$ into closed symplectic manifolds. We will consider a resolution which is between the minimal and minimal normal crossing resolution indicated by the graph below. Here the $+1$- and $-4$-spheres intersect tangentially with multiplicity $3$, and the $-1$-sphere intersects these two at the same point transversally.
	
	\[
	\xygraph{
		!{<0cm,0cm>;<1.25cm,0cm>:<0cm,1cm>::}
		!{(3,0) }*+{\bullet^{-2}}="c"
		!{(2,0) }*+{\bullet^{-2}}="l3a"
		!{(1,0) }*+{\bullet^{-2}}="l3b"
		!{(0,0) }*+{\bullet^{-2}}="l3c"
		!{(-1,0) }*+{\bullet^{-4}}="l3d"
		!{(-2,0) }*+{\bullet^{+1}}="l3e"
		!{(4,0) }*+{\bullet^{-2}}="l2a"
		!{(5,0) }*+{\bullet^{-2}}="l2b"
		!{(1,1) }*+{\bullet^{-2}}="l1a"
		!{(0,1) }*+{\bullet^{-2}}="l1b"
		!{(-1,1) }*+{\bullet^{-1}}="l1c"
		"c"-"l2a"
		"c"-"l3a"
		"l2a"-"l2b"
		"l3a"-"l3b"
		"l3b"-"l3c"
		"l3c"-"l3d"
		"l3d"-"l3e"
		"l3d"-@/_.1cm/"l3e"
		"l3d"-@/^.1cm/"l3e"
		"l1a"-"l1b"
		"l1b"-"l1c"
		"l1c"-"l3e"
		"l1c"-"l3d"
	}
	\]

	Again, we have results from~\cite[Section 6]{bigGS} classifying the symplectic isotopy classes of relatively minimal embeddings of this resolution. In this case, there are six such embeddings distinguished by the maps they induce on second homology which are given as follows.

Let $W_9$ denote the filling complementary to the embedding of the resolution into $\cptwo\#18\cptwobar$ with homology classes:
$$\begin{array}{llllllll}	& h-e_0-e_{16},& e_{16}-e_{17}, &e_{17}-e_{18}& & & \\
	h, & 3h-2e_0-e_1-\dots -e_9, & e_1-e_{10}, & e_{10}-e_{11}, & e_{11}-e_{12}, & e_{12}-e_{13}, & e_{13}-e_{14}, & e_{14}-e_{15}
\end{array}$$

Let $W_6$ denote the filling complementary to the embedding of the resolution into $\cptwo\#16\cptwobar$ with homology classes:	
$$\begin{array}{llllllll}	& h-e_8-e_9, & e_8-e_7, & e_9-e_8& & & \\
	h, & 3h-2e_0-e_1-\dots -e_9, & e_1-e_{10}, & e_{10}-e_{11}, & e_{11}-e_{12}, & e_{12}-e_{13}, & e_{13}-e_{14}, & e_{14}-e_{15}
\end{array}$$

Let $Z_6$ denote the filling complementary to the embedding of the resolution into $\cptwo\#16\cptwobar$ with homology classes:	
$$\begin{array}{llllllll}	&  h-e_8-e_9, & e_8-e_7, & e_7-e_6& & & \\
	h, & 3h-2e_0-e_1-\dots -e_9, & e_1-e_{10}, & e_{10}-e_{11}, & e_{11}-e_{12}, & e_{12}-e_{13}, & e_{13}-e_{14}, & e_{14}-e_{15}
\end{array}$$

Let $W_4$ denote the filling complementary to the embedding of the resolution into $\cptwo\#14\cptwobar$ with homology classes:	
$$\begin{array}{llllllll}	& h-e_0-e_{11},& e_{11}-e_{12},& e_{12}-e_{13}& & & \\
	h,& 3h-2e_0-e_1-\dots-e_9,& e_1-e_{10},& e_2-e_1,& e_3-e_2,& e_4-e_3,& e_5-e_4,& e_6-e_5
\end{array}$$

Let $W_1$ denote the filling complementary to the embedding of the resolution into $\cptwo\#11\cptwobar$ with homology classes:	
$$\begin{array}{llllllll}	& h-e_8-e_9, & e_8-e_7, & e_9-e_8& & & \\
	h,& 3h-2e_0-e_1-\dots-e_9,& e_1-e_{10},& e_2-e_1,& e_3-e_2,& e_4-e_3,& e_5-e_4,& e_6-e_5
\end{array}$$

Let $W_0$ denote the filling complementary to the embedding of the resolution into $\cptwo\#10\cptwobar$ with homology classes:	
$$\begin{array}{llllllll}	& h-e_8-e_9 &  e_8-e_7 & e_9-e_8& & & \\
	h,& 3h-2e_0-e_1-\dots-e_9,& e_0-e_1,& e_1-e_2,& e_2-e_3,& e_3-e_4,& e_4-e_5,& e_5-e_6
\end{array}$$

		Observe that the subscript indicates the second Betti number. A posteriori, due to Proposition~\ref{p:canonical643}, $W_9$ is the plumbing of symplectic spheres according to the graph in~\eqref{e:256/x}.
		
		\begin{prop}\label{p:643blowdowns} Let $C$ be the rational cuspidal curve with a single $T(6,43)$ singularity and self-intersection number $16^2$.
		There exists a rational blow-down of a connected embedded plumbing of symplectic spheres from one filling of $(Y_C,\xi_C)$ to another if and only if there is an edge between the fillings in following graph.
			\[
		\xygraph{
			!{<0cm,0cm>;<1cm,0cm>:<0cm,0.75cm>::}
			!{(-1,2) }*+{W_9}="w9"
			!{(-1,1) }*+{W_6}="w6"
			!{(2,1) }*+{Z_6}="z6"
			!{(-4,0) }*+{W_4}="w4"
			!{(-3,-1) }*+{W_1}="w1"
			!{(2,-2) }*+{W_0}="w0"
			"w9"-"w6"
			"w9"-"z6"
			"w9"-"w4"
			"w9"-"w0"
			"w6"-"w1"
			"w4"-"w1"
		}
		\]
		\end{prop}
		
		\begin{proof}
		First, we observe the existence of the rational blow-downs from $W_9$ to $W_6$ and $W_4$, and from $W_6$ and $W_4$ to $W_1$.
		Using the fact that $W_9$ is necessarily the plumbing~\eqref{e:256/x} (by Proposition~\ref{p:canonical643}), we can visibly see sub-plumbings of the form $(-8,-2,-2,-2,-2)$ and $(-6,-2,-2)$. Note that these symplectic embeddings are not unique, but it is possible to realize disjoint embeddings of both the $(-8,-2,-2,-2,-2)$ and $(-6,-2,-2)$ chains as disjoint sub-plumbings. Rationally blowing down a $(-8,-2,-2,-2,-2)$ chain yields a symplectic filling with $b_2=4$ which is necessarily $W_4$. Rationally blowing down a $(-6,-2,-2)$ chain yields symplectic fillings with $b_2=6$, which must be either $W_6$ or $Z_6$. Rationally blowing down disjoint plumbings $(-8,-2,-2,-2,-2)$ and $(-6,-2,-2)$ yields a symplectic filling with $b_2=1$, so it must be $W_1$. Using the homological properties that define $Z_6$, we will show below that $Z_6$ does not admit any further rational blow-downs. Therefore, rationally blowing down a $(-6,-2,-2)$ plumbing in $W_9$ which is disjoint from a $(-8,-2,-2,-2,-2)$ plumbing, necessarily results in $W_6$. This shows as well that there is a further rational blow-down from $W_6$ to $W_1$. If we exchange the order of which of the two disjoint plumbings we rationally blow-down first, we see the sequence of rational blow-downs from $W_9$ to $W_4$ to $W_1$.
		
		Next, we will show there is a rational blow-down from $W_9$ to $Z_6$. Consider the $(-6,-2,-2)$ chain embedded as a sub-plumbing of $W_9$ given by turning right instead of left at the $3$-valent vertex. This sub-plumbing is not disjoint from the $(-8,-2,-2,-2,-2)$ chain, and thus potentially yields a different result from the rational blow-down of the other embedding of the $(-6,-2,-2)$ chain (where you turn left at the $3$-valent vertex). We will now verify that rationally blowing down this embedding of the $(-6,-2,-2)$ plumbing is not $W_6$, and thus must be $Z_6$.
		
		First, we argue that the result of this rational blow-down is simply connected.
		Since $W_9$ is a tree plumbing of spheres, it is simply connected.
		The fundamental group of the rational homology ball which replaces the $(-6,-2,-2)$ plumbing is normally generated by the meridian of the last $-2$-sphere in the chain, so it suffices to show this curve is null-homotopic in the complement of the plumbing in $W_9$.
		This meridian can be realized as the equator of the next $-2$-sphere $S$ in the $W_9$ plumbing, so it bounds a disk (the other half of the $-2$-sphere $S$) in the complement of the $(-6,-2,-2)$ chain. Consequently the result of this rational blow-down is a simply-connected filling, and in particular it has trivial $H_1$.
		
		We will next show that $W_6$ necessarily has non-trivial $H_1$. To see this, consider the long exact sequence of the pair $(W_6,\partial W_6)$:
		\[
		0 = H_2(\partial W_6) \to H_2(W_6) \to H_2(W_6,\partial W_6)\cong H^2(W_6) \to H_1(\partial W_6)\cong \Z/{256}\Z
		\]
		Here we identify $H_2(W_6,\partial W_6)\to H^2(W_6)$ by Alexander--Lefschetz duality. Under this identification, the map $H_2(W_6)\to H^2(W_6)$ in the sequence is described by the intersection form $Q_{W_6}$ of $W_6$, by sending $a\in H_2(W_6)$ to $Q_{W_6}(a,\cdot)$. The size of the cokernel of this map is the determinant of $Q_{W_6}$. In~\cite[Proof of Theorem~6.5, p. 52]{bigGS}, we computed the intersection forms for $W_6$ and $Z_6$, by finding an integral homology basis for the orthogonal complement in $H_2(\cptwo\#16\cptwobar)$
		of the corresponding embeddings of the resolutions for $W_6$ and $Z_6$ (which are listed above). We calculated that $\det(Q_{W_6})=64$ and $\det(Q_{Z_6})=256$. (Note, this distinguishes $W_6$ from $Z_6$.) On the other hand, $H_1(\partial W_6) \cong \Z/{256}\Z$ since $\partial W_6 = -S^3_{256}(T(6,43))$. Since the co-kernel of the map $H_2(W_6)\to H^2(W_6)$ has order $64$, the map from $H^2(W_6)$ to $H_1(\partial W_6)\cong \Z/{256}\Z$ cannot be surjective. Therefore, there are some elements which are not in the kernel of the map $H_1(\partial W_6)\to H_1(W_6)$, so $H_1(W_6)\neq 0$. This shows that the rational blow-down of this embedding of the $(-6,-2,-2)$ chain is necessarily $Z_6$.
		
		Finally, the entirety of the $W_9$ plumbing can be rationally blown down in a non-linear way (this plumbing is~\cite[Figure~1(j)]{BhupalStipsicz} for $q=4$) to $W_0$.
		
		Next, we will show that fillings which are not connected by edges are not related by a symplectic rational blow-down.
		Note that because rational blow-down strictly decreases $b_2$, we only need to obstruct rational blow-downs from larger fillings to smaller fillings.
		
		We will next show that $Z_6$ cannot be symplectically rationally blown down to any other filling. To do this, we start by looking at the classes in $H_2(Z_6)$ which could be represented by a symplectic sphere.
		Classes in $H_2(Z_6)$ must be in the orthogonal complement of the classes listed above for the $Z_6$ embedding of the resolution in $H_2(\cptwo\#16\cptwobar)$. Classes in this orthogonal complement are classes of the form $\sum_{i=0}^{15} a_ie_i$ where
	\[
		a_1=a_{10}=\dots = a_{15}, \qquad a_6=a_7=a_8=-a_9, \qquad 2a_0+a_1+\dots+a_9=0
	\]
		where, using the equalities above, the last equation can equivalently be written as
		\begin{equation}\label{eq:orthz6}
			2a_0+a_1+\dots+a_5+2a_6=0.
		\end{equation}
		If the class represents a symplectic sphere, by Lemma~\ref{l:hom}, there exists an index $i_0\in \{0,\dots, 15\}$ such that $a_{i_0}=1$ and $a_i\in \{-1,0\}$ for all $i\neq i_0$. Combining this with the forced equalities of coefficients above, we see that $i_0\notin\{1,6,7,8,10,\dots, 15\}$. Furthermore $i_0\neq 9$ because if $a_9=1$, $a_6=a_7=a_8=-1$, but then there are no solutions to Equation~\eqref{eq:orthz6} if $a_0,\dots,a_5\in \{0,-1\}$. Therefore the only possibilities are that $i_0\in \{0,2,3,4,5\}$, and we get the following possibilities:
		\[
		\begin{array}{ll}
			e_0 - e_1 - e_{10} - \dots - e_{15} - e_i, & 2 \le i \le 5\\
			e_0 - e_i - e_j, & 2 \le i < j \le 5\\
			e_i-e_1-e_{10} - \dots - e_{15}, & 2\le i \le 5\\
			e_i - e_j, & 2 \le i \neq j \le 5.
		\end{array}
		\]
		Note that classes of the first type have square $-9$, the second type have square $-3$ and the third type have square $-8$, and the last have square $-2$. In particular, any symplectic plumbing of spheres which embeds in $Z_6$ can only include spheres with self-intersection numbers in the set $\{-2,-3,-8,-9\}$. 
		
		We start with a general observation: any plumbing of spheres which can be symplectically rationally blown down contains at least one sphere with self-intersection strictly less than $-3$. For linear plumbings this follows from the fact that they are all obtained as $2$-expansions of $(-4)$, and for non-linear plumbings this follows from inspection of the families in~\cite{BhupalStipsicz} (note the parameters $p,q,r$ must be non-negative).		
		
		Therefore, the only possible plumbings that may be symplectically rationally blown down which can embed in $Z_6$ must include at least one sphere of self-intersection $-8$ or $-9$. Because $b_2(Z_6)=6$, any embeddable plumbing must have $b_2\leq 6$. The only $2$-expansion of $(-4)$ of length $\leq 6$ which includes a $-9$-sphere is $(-9,-2,-2,-2,-2,-2)$. Those that include a $-8$-sphere are $(-2,-8,-2,-2,-2,-3)$ and $(-8,-2,-2,-2,-2)$. For the non-linear plumbings in~\cite{BhupalStipsicz}, we can immediately rule out embeddings of plumbings in Figure 1(a),(e),(h),(i),(j) and 2(a),(b),(c) because they contain either $-4$- or $-6$-spheres. We rule out 1(b),(c),(d),(g) because any plumbing in these families has $b_2>6$ when self-intersections are restricted to lie in $\{-2,-3,-8,-9\}$. Thus the only non-linear plumbing which could be rationally blown down with $b_2\leq 6$, and self-intersection numbers in this class is the case of~\cite[Figure~1(f) with $q=2$]{BhupalStipsicz}, which happens to be the plumbing~\eqref{e:64/7}. We will rule out these remaining cases now using the pairwise intersections of these classes (we will freely use the lemmas of Section~\ref{s:background}). If there were an embedding of the linear plumbing $(-9,-2,-2,-2,-2,-2)$ in $Z_6$, up to permuting the indices $\{2,3,4,5\}$, the first sphere in the chain would represent $e_0-e_1-e_{10}-\dots-e_{15}-e_2$, and the next three must represent $e_2-e_3$, $e_3-e_4$, $e_4-e_5$ (in order for these spheres to have pairwise intersections according to the linear chain). This leaves no possibility $e_i-e_j$, $i,j\in\{2,3,4,5\}$ for the fourth $-2$-sphere in the chain which has intersection $0$ with the first two $-2$-spheres and $1$ with the third. 
		An embedding of $(-8,-2,-2,-2,-2)$ would necessarily (up to a permutation of $\{2,3,4,5\}$) have the first four spheres in homology classes $e_2-e_1-e_{10}-\cdots-e_{15}$, $e_3-e_2$, $e_4-e_3$, $e_5-e_4$, leaving no possible $i,j\in \{2,3,4,5\}$ such that $e_i-e_j$ could have intersection $+1$ with $e_5-e_4$ and intersection $0$ with the first three classes. Note that the non-linear plumbing~\eqref{e:64/7} which is case 1(f) on the Bhupal--Stipsicz list, contains the linear chain $(-8,-2,-2,-2,-2)$ as a sub-plumbing so this non-linear plumbing is also obstructed from symplectically embedding into $Z_6$. Similarly, an embedding of the linear chain $(-2,-8,-2,-2,-2,-3)$ would (up to permutation of $\{2,3,4,5\}$) have the first four spheres in classes, $e_2-e_3$, $e_3-e_1-e_{10}-\cdots - e_{15}$, $e_4-e_3$, $e_5-e_4$, leaving no options for the fifth sphere in the chain with the correct intersection number.
		Thus there is no embedding into $Z_6$ of any plumbing which can be symplectically rationally blown down.
		
		Note at this point, we have established that $Z_6$ cannot be the filling which results from rationally blowing down the $(-6,-2,-2)$ chain in $W_9$ which is disjoint from the $(-8,-2,-2,-2,-2)$ chain (since $Z_6$ admits no further rational blow-downs). Because $W_6$ is the only other symplectic filling with $b_2=6$, $W_6$ is necessarily the result of the rational blow-down of this disjoint $(-6,-2,-2)$ chain in $W_9$. This establishes that $W_6$ can be rationally blown down to $W_1$, and gives another way to see that $W_6$ and $Z_6$ are not symplectomorphic.
		
		Next, we will similarly obstruct rational blow-downs from $W_6$ to $W_4$ and $W_0$, though the obstruction to a rational blow-down to $W_0$ will be significantly more subtle. In this case, the orthogonality relations imply classes in $H_2(W_6;\Z)$ are precisely those of the form $\sum_{i=0}^{15}a_ie_i$ where
		$$a_7=a_8=a_9=0, \qquad a_1=a_{10}=\dots = a_{15}, \qquad 2a_0+a_1+\dots+a_6=0.$$
		Proceeding as before, the classes satisfying these constraints which are represented by symplectic spheres are
		\[
		\begin{array}{ll}
			e_0 - e_1 - e_{10} + \cdots - e_{15} - e_i, & 2 \le i \le 6\\
			e_0 - e_i - e_j, & 2 \le i < j \le 6\\
			e_i - e_1 - e_{10} + \cdots - e_{15}, & 2 \le i \le 6\\
			e_i - e_j, & 2 \le i \neq j \le 6.
		\end{array}
		\]
		Note that the squares of these classes are $\{-9,-3,-8,-2\}$. Since the only plumbing with $b_2=2$ which can be rationally blown down is $(-5,-2)$ and there are no symplectic $-5$-spheres in $W_6$, there can be no symplectic rational blow-down from $W_6$ to $W_4$. We can obstruct the embeddings into $W_6$ of symplectic linear plumbings of with $b_2=6$ as before. The only two possibilities are $(-9,-2,-2,-2,-2,-2)$ and $(-2,-8,-2,-2,-2,-3)$. In the former case, the first five spheres in the chain would necessarily (up to permuting indices $\{2,\dots, 6\}$) represent $e_0-e_1-e_{10}-\cdots-e_{15}-e_2$, $e_2-e_3$, $e_3-e_4$, $e_4-e_5$, $e_5-e_6$, leaving no option for the sixth sphere in the chain. In the latter case, the chain would need to take the form $e_3-e_2$, $e_2-e_1-e_{10}-\cdots-e_{15}$, $e_4-e_2$, $e_5-e_4$, $e_6-e_5$, $e_0-e_6-e_j$ for some $j\in\{2,\dots, 5\}$, but any value of $j$ would result in a non-zero intersection of the last sphere in the chain with one of the first four spheres. 
		
		Next we consider potential non-linear plumbings which can be rationally blown down. In fact, there is a unique non-linear plumbing from the Bhupal-Stipsicz list with $b_2=6$ which we cannot rule out with homology classes alone. This is the plumbing~\cite[Figure~1(f)]{BhupalStipsicz} for $q=2$, which happens to be the plumbing~\eqref{e:64/7} (the filling of the other exceptional cuspidal contact manifold). This consists of a linear chain $(-8,-2,-2,-2,-2)$ with an additional $-3$-sphere intersecting the second to last $-2$-sphere. Up to permutation of indices $\{2,\dots, 6\}$, the spheres in the linear chain necessarily represent $e_2-e_1-e_{10}-\dots-e_{15}$, $e_3-e_2$, $e_4-e_3$, $e_5-e_4$, $e_6-e_5$ and the additional $-3$-sphere would represent $e_0-e_5-e_6$. Since all of these classes could be represented by symplectic spheres in $W_6$, obstructing this rational blowdown is a little more subtle.
		
		\begin{figure}
			\centering
			\includegraphics[scale=1.1]{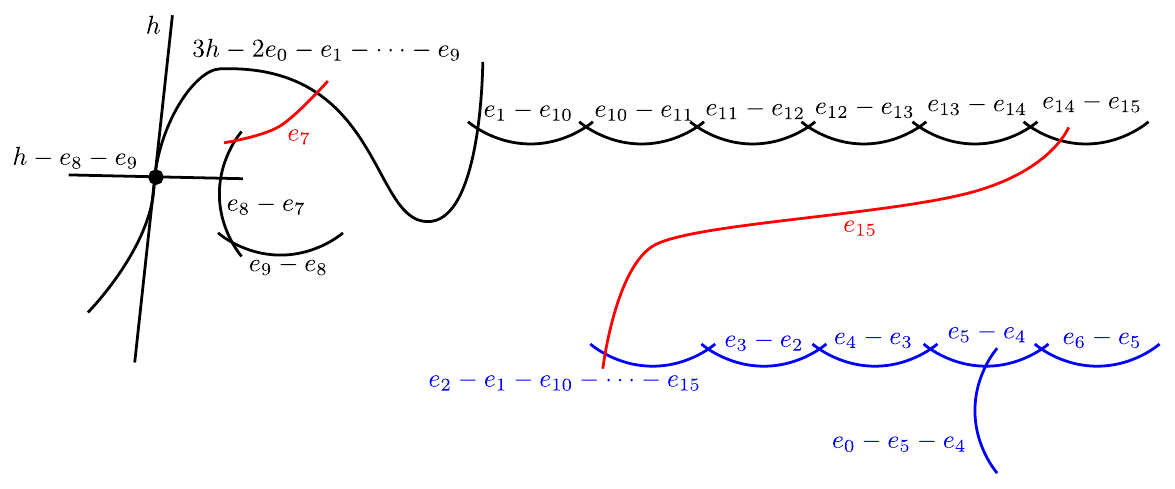}
			\caption{Embedding into $\cptwo\# 16\cptwobar$ of the resolution complementary to $W_6$, together with a hypothesized disjoint plumbing with the homology classes of each component specified, as well as exceptional spheres representing $e_7$ and $e_{15}$.}
			\label{fig:disjoint}
		\end{figure}
		
		\begin{figure}
			\centering
			\includegraphics[scale=1]{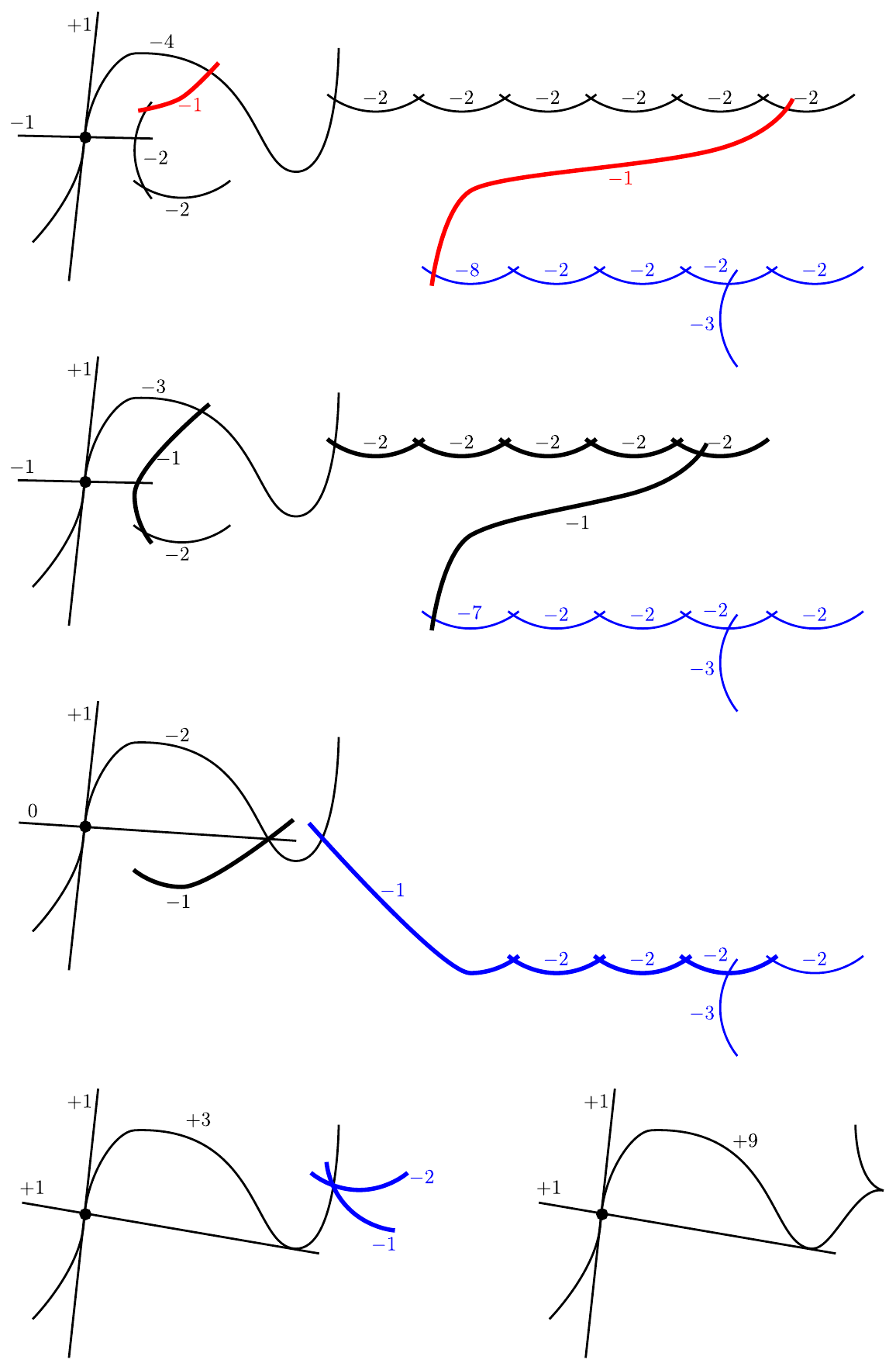}
			\caption{Sequence of blow-downs from $\cptwo\# 16\cptwobar$ to $\cptwo$, starting with the configuration of Figure~\ref{fig:disjoint}, and tracking the image under the blow-downs, ending with a cuspidal cubic and two lines (one tangent to order $3$ and the other tangent to order $2$ and passing through the inflection point). At each stage, the thickened curves indicate the curves which are blown down to get to the next figure in the sequence.}
			\label{fig:blowdownseq}
		\end{figure}
		
		Suppose there exists such a symplectic plumbing in $W_6$. By gluing the concave neighborhood of the resolution, this is equivalent to assuming that in $\cptwo\#16\cptwobar$, there exists an embedding of the resolution configuration (with homology classes as specified for $W_6$) and a disjoint embedding of the plumbing~\eqref{e:64/7} representing homology classes as described above. Observe that the classes $e_7$ and $e_{15}$ intersect all of the components in these two configurations non-negatively. Therefore, by Lemma~\ref{l:blowdown}, there exist embedded exceptional spheres representing $e_7$ and $e_{15}$ which intersect all curve components in these configurations non-negatively (so the geometric and algebraic intersection numbers match). Consequently, we have an embedding into $\cptwo\#16\cptwobar$ of a configuration of curves as in Figure~\ref{fig:disjoint}. Blowing down symplectic $-1$-spheres in this configuration representing $e_i$ classes repeatedly as in Figure~\ref{fig:blowdownseq}, we eventually reach a configuration of curves in $\cptwo$ consisting of the following components: a cubic $C$ with a simple cusp at a point $r$, a line $L$ which tangentially intersects $C$ at an inflection point $p$ (with multiplicity $3$), and a line $T$ which intersects $C$ and $L$ transversally at $p$, and intersects $C$ tangentially (with multiplicity $2$) at an additional point $q$. (Here $r,p,q$ are all distinct points on $C$.) In fact, such a configuration of symplectic curves cannot exist, by the Riemann--Hurwitz formula. We fix an almost complex structure $J$ such that $C$, $L$, and $T$ are $J$-holomorphic. Then, using the pencil of $J$-holomorphic lines through $p$ gives a degree-2 map $\pi\colon\cpone\to\cpone$ with at least two ramification points (corresponding to the inflection line and the cusp respectively).
		Therefore Riemann--Hurwitz reads: $2 = 2\cdot 2 - \sum (e_\pi(p)-1)$, which implies that these are the only two ramification points, from which we deduce that there is no other tangent drawn to the cubic from the inflection point. (In fact, this exact argument appears in~\cite[p. 51]{bigGS}.) Thus, we reach a contradiction, and the plumbing~\eqref{e:64/7} cannot embed symplectically in $W_6$, so we conclude there is no rational blow-down from $W_6$ to $W_0$.

		To see there is no rational blow-down from $W_4$ to $W_0$, observe that $H_2(W_4)$ consists of classes $\sum_{i=0}^{13} a_ie_i$ where
		\[
		a_1=a_2=a_3=a_4=a_5=a_6=a_{10}, \qquad a_{11}=a_{12}=a_{13} = -a_0, \qquad 2a_0+a_1+\dots+a_9=0.
		\]
		The classes of this form which can be represented by symplectic spheres are
		\[
		\begin{array}{ll}
			e_0 - e_i - e_j - e_{11} - e_{12} - e_{13}, & 7 \le i < j \le 9\\
			e_i - e_j, & 7 \le i \neq j \le 9.
		\end{array}
		\]
		These classes have squares $-6$ or $-2$. The only plumbing with $b_2\leq 4$ which can be symplectically rationally blown down and involves only $-6$- and $-2$-spheres is the linear chain $(-6,-2,-2)$ which connects $W_4$ to $W_1$. There is no such plumbing of length $4$ so there is no symplectic rational blow-down from $W_4$ to $W_0$.
		
		The classes in $H_2(W_1)$ have the form $\sum_{i=0}^{10}a_ie_i$ such that
		\[
		a_1=a_2=a_3=a_4=a_5=a_6=a_{10}, \qquad a_7=a_8=a_9=0, \qquad 2a_0+a_1+\dots+a_9=0.
		\]
		These are all integer multiples of $3e_0-e_1-\dots-e_6-e_{10}$ which has square $-16$, so there is no $-4$-sphere in $W_1$ to rationally blow down.

		\end{proof}
		
		\begin{remark}
		Similar computations as seen in the previous proof which analyze the homology classes represented by symplectic spheres in $V_6$ proves that there is no linear rational blow-down from $V_6$ to $V_0$ in the case of $\Ec$, thus slightly strengthening the statement of Proposition~\ref{p:322blowdowns}. Similarly, one can obstruct the existence of a linear rational blow-down from $W_9$ to $W_0$. Such computations show that these non-linear rational blow-downs are genuinely new symplectic cut-and-paste operations, rather than just hidden reformulations of the prior known linear rational blow-downs.
		\end{remark}

\section{Bounds on self-intersection numbers of rational cuspidal curves}\label{ss:bounds}
		In the earlier sections, we discussed symplectic fillings of contact manifolds $(Y_C,\xi_C)$ arising on the concave boundary of a neighborhood of a rational cuspidal curve $C$ whose algebraic genus and self-intersection number are determined by a degree which would allow the curve $C$ to symplectically embed in $\cptwo$. In this section and the next, we study symplectic fillings for more general contact manifolds $(Y_C,\xi_C)$ where self-intersection need not be $d^2$ and the arithmetic genus need not be a triangular number $\frac12(d-1)(d-2)$.
		
		$C$ will be a singular symplectic curve with positive self-intersection number (and thus admitting a concave neighborhood). Recall that $(Y_C,\xi_C)$ depends only on the singularity types, geometric genus, and self-intersection number of $C$. In this section, we will fix the singularity types for a rational (geometric genus equal to zero) curve, and vary the self-intersection number $s$. The goal is to prove that for certain values of $s$ (depending on the singularities we fixed), $(Y_C,\xi_C)$ is not symplectically fillable.
		
		We begin with an easy, quite general remark.
			
		\begin{lemma}\label{l:high-four}
			Let $C$ and $C'$ be singular curves with the same geometric genus and configuration of singularities (not necessarily cuspidal), and $0<C'\cdot C' \le C\cdot C$. If $(Y_C,\xi_C)$ is strongly symplectically fillable, then $(Y_{C'},\xi_{C'})$ is strongly symplectically fillable as well.
		\end{lemma}
			
		\begin{proof}
			Suppose $(Y_C,\xi_C)$ has a strong symplectic filling $W$. Let $N_C$ denote a standard concave neighborhood of $C$, and let $(X,\omega)$ denote the closed symplectic manifold which results from symplectically gluing $N_C$ to $W$ along their common contact boundary. Symplectically blow up $X$ at $n:=C\cdot C- C'\cdot C'$ smooth points of $C\subset X$. The proper transform of $C$ yields a symplectic embedding of $C'$ in $X\# n\cptwobar$. Since $C'\cdot C'>0$, there exists a concave neighborhood of $C'$~\cite[Theorem 2.13]{bigGS} in $X\#n\cptwobar$ whose contact boundary is $(Y_{C'},\xi_{C'})$. The complement of this neighborhood is a strong symplectic filling of $(Y_{C'},\xi_{C'})$.
		\end{proof}
			
		Recall the definitions of the multiplicity sequence and $\delta(p)$ from Section~\ref{s:background}, and that,
		for a cuspidal point $p$, $2\delta(p)$ is the sum of $m(m-1)$ over all elements $m$ in the multiplicity sequence for $p$. Let $M(p)$ sum of all squares in the multiplicity sequence. Let $\ell(p)$ denote the last (and smallest) entry in the multiplicity sequence. 
		(Note that in our convention, $\ell(p) > 1$ because we define the multiplicity sequence using the minimal resolution rather than the normal crossing resolution.)
		Note that $M(p)$ (respectively, $M(p) + \ell(p)$) is the defect in self-intersection when taking the minimal smooth (resp. normal crossing) resolution of $p$.
		For instance, when the singularity at $p$ is of type $T_{m,km+1}$, the multiplicity sequence is $[m^{[k]}]$, so that $M(p) = km^2$, $\ell(p) = m$, and $\delta(p) = \frac12 km(m-1)$.
						
		\begin{prop}\label{p:high-five}
			Suppose that $C$ is a rational curve with cusp singularities $p_1, \dots, p_\mu$ ($\mu \ge 1$), and with reducible singularities $q_1,\dots,q_\nu$ ($\nu \ge 0)$, and satisfying $C\cdot C \ge \sum M(p_i) + \sum M(q_j) + 2\min \ell(p_i) + 2$.
			Then $(Y_C,\xi_C)$ is not strongly symplectically fillable.
			If additionally $p_1,\dots, p_\mu$ are the only singularities of $C$ (i.e. $C$ is cuspidal), then $\xi_C$ is not even weakly symplectically fillable.
		\end{prop}

		\begin{proof}
			Without loss of generality, suppose that $p_1$ is a singular cusp with $\ell(p_1) = \ell = \min \ell(p_i)$.
			By Lemma~\ref{l:high-four} above, it suffices to prove the statement for $C\cdot C = \sum M(p_i) + \sum M(q_i) + 2\ell + 2$.
				
			Let $(W,\omega_W)$ be a strong symplectic filling of $(Y_C,\xi_C)$, and $(X',\omega')$ be the closed symplectic manifold obtained by gluing a standard concave neighborhood of $C$ to $(W,\omega_W)$.
			We now view $C$ as a symplectic curve in $X'$.
				
			Blow up $C$ at all its singular points until we obtain the minimal smooth resolution of $C$, i.e. the proper transform $\widetilde{C}$ of $C$ is a smooth sphere of self-intersection $\widetilde{C}\cdot\widetilde{C} = C\cdot C - \sum M(p_i) - \sum M(q_j) = 2\ell + 2$.
			Note that, by definition, at $p_1$ the last blow-up creates an exceptional divisor that intersects $\widetilde{C}$ at a single point with multiplicity $\ell$.
				
			Blow up $\ell$ more times at $p_1$, so that at $p_1$ we get to the normal crossing resolution of the singularity $(C,p_1)$, and then $\ell+1$ more times at the intersection between (the proper transform of) $C$ and the last exceptional divisor.
			Let $(X,\omega)$ denote the corresponding blow-up of $(X',\omega')$.
				
			The 4-manifold $X$ contains the following configuration of symplectic spheres: a $+1$-sphere $C'$, the proper transform of $C$, an exceptional divisor intersecting $C$ transversely once, and a string of $-2$-curves departing from it; this is depicted in~\eqref{e:high-five}.
			\begin{equation}\label{e:high-five}
			\xygraph{
					!{<0cm,0cm>;<0.9cm,0cm>:<0cm,0.9cm>::}
					!~-{@{-}@[|(2.5)]}
					!{(0,0.8) }*+{\bullet}="vl"
					!{(1.5,0.8) }*+{\dots}="vd"
					!{(3,0.8) }*+{\bullet}="vf"
					!{(4.5,0.8) }*+{\bullet}="c"
					!{(6,0.8) }*+{\bullet}="wf"
					!{(4.5,-0.7) }*+{\bullet}="b"
					!{(7.5,0.8) }*+{\dots}="wd"
					!{(9,0.8) }*+{\bullet}="wl"
					!{(10.5,0.8) }*+{\bullet}="e"
					!{(12,0.8) }*+{\bullet}="h"
					!{(0,1.2) }*+{-2}
					!{(4.5,-1.1) }*+{-\ell-1}
					!{(3,1.2) }*+{-2}
					!{(4.5,1.2) }*+{-2}
					!{(6,1.2) }*+{-2}
					!{(9,1.2) }*+{-2}
					!{(10.5,1.2) }*+{-1}
					!{(12,1.2) }*+{+1}
					"vl"-"vd"
					"vf"-"vd"
					"wl"-"wd"
					"wf"-"wd"
					"vf"-"c"
					"wf"-"c"
					"c"-"b"
					"wl"-"e"
					"e"-"h"
				}
			\end{equation}
			The left leg contains exactly $\ell-1$ vertices of weight $-2$, while the right leg contains $\ell$ of them. In total, the chain of $-2$-vertices has length $2\ell$.
				
			By Theorem~\ref{thm:mcduff}, $X$ is symplectomorphic to a blow-up of $\cptwo$, and $C'$ can be identified with a line in $\cptwo$. Using the standard basis for $H_2(\cptwo\#N\cptwobar)$, by Lemma~\ref{l:hom} the $-1$-sphere adjacent to $C'$ is in the homology class $h-e_0-e_1$.
			Since the chain of $-2$s is of length $2\ell > 2$, by Lemma~\ref{l:2chainfix}, up to relabeling the $e_i$, the homology classes in the chain are $e_1-e_2, \dots, e_{2\ell}-e_{2\ell+1}$.
			The central vertex in the plumbing is in the homology class $e_{\ell+1}-e_{\ell+2}$.
			
			The $(-\ell-1)$-sphere in~\eqref{e:high-five} is disjoint from the line $C'$, so by Lemma~\ref{l:hom} its homology class is of the form $e_i - e_{j_1} - \cdots - e_{j_\ell}$ for some $i, j_1, \dots, j_\ell$.
			It is also disjoint from the two halves of the chain of $-2$s,
			so the coefficients appearing in the homology class of the $(-\ell-1)$-sphere of $e_1,\dots,e_{\ell+1}$ are all equal---namely they are either all $0$ or all $-1$.
			Since the $(-\ell-1)$-sphere intersects the class $e_{\ell+1}-e_{\ell+2}$ once positively, either $i=\ell+2$ or $j_k=\ell+1$ for some $k$. The former possibility is ruled out by Lemma~\ref{l:pos}, therefore $e_1,\dots, e_{\ell+1}$ must all appear with coefficient $-1$ in the homology class for the $(-\ell-1)$-sphere. However, the self-intersection of this class is at most $-\ell-2$, a contradiction.
			This proves the first assertion.

			Note that when $C$ is a rational cuspidal curve, i.e. if $\nu = 0$, the 3-manifold $Y_C$, that is the boundary of a regular neighborhood of $C$, is a rational homology sphere; in particular, since every weak symplectic filling of $\xi_C$ can be deformed to a strong symplectic filling~\cite[Lemma~1.1]{OhtaOno2}, the last assertion follows.
		\end{proof}
			
		In fact we note here that for unicuspidal curves with a singularity of type $(\ell,\ell+1)$ the maximal self-intersection allowed by the proposition is $\ell^2 + 2\ell + 1 = (\ell+1)^2$, which is exactly the degree of the rational cuspidal curve $\{x^{\ell+1}-y^\ell z=0\}$ of type $\Ac_\ell$. This shows that the inequality is sharp.
			
		We give a small refinement of the proposition above for curves with more than one cusp.
			
		\begin{prop}\label{p:high-six}
			Suppose that $C$ is a rational curve with singularities $p_1, \dots, p_\mu$ ($\mu > 1$) such that $\ell(p_1) = \ell(p_2) = \min\ell(p_i)$, with reducible singularities $q_1,\dots,q_\nu$ ($\nu \ge 0)$, and self-intersection $C\cdot C$ at least $\sum M(p_i) + \sum M(q_j) + 2\min \ell(p_i) + 1$.
			Then the associated cuspidal contact structure $\xi_C$ is not strongly symplectically fillable.
			If additionally $p_1,\dots, p_\mu$ are the only singularities of $C$, then $\xi_C$ is not even weakly symplectically fillable.
		\end{prop}
			
		\begin{proof}
			Let $\ell = \ell(p_1) = \ell(p_2)$. Again, by Lemma~\ref{l:high-four}, we can assume that $s(C) = \sum M(p_i) + 2\min \ell(p_i) + 1$.
			As above, suppose that there is a filling $(W,\omega_W)$ of $(Y_C,\xi_C)$ and glue it to a standard concave neighborhood of $C$ to obtain the closed manifold $(X',\omega')$.
			Blow up $\ell$ times at each of the points $p_1$ and $p_2$ to get to the normal crossing divisor resolution at $p_1$ and $p_2$, in the blown-up manifold $(X,\omega)$.
			By assumption, the proper transform of $C$ is a symplectic $+1$-sphere, and by Theorem~\ref{thm:mcduff} we can identify it with a line in a blow-up of $\cptwo$.
			However, in $X$ we see the following configuration.
			\begin{equation}\label{e:high-six}
				\xygraph{
					!{<0cm,0cm>;<0.9cm,0cm>:<0cm,0.9cm>::}
					!~-{@{-}@[|(2.5)]}
					!{(3,-0.2) }*+{\bullet}="c"
					!{(1.5,-0.2) }*+{\bullet}="el"
					!{(0.2,0.55) }*+{\bullet}="2l"
					!{(0.2,-0.95) }*+{\bullet}="3l"
					!{(4.5,-0.2) }*+{\bullet}="er"
					!{(5.8,0.55) }*+{\bullet}="2r"
					!{(5.8,-0.95) }*+{\bullet}="3r"
					!{(3,0.2) }*+{+1}
					!{(1.5,0.2) }*+{-1}
					!{(4.5,0.2) }*+{-1}
					!{(0.2,0.95) }*+{-\ell-1}
					!{(0.2,-0.55) }*+{-2}
					!{(5.8,0.95) }*+{-\ell-1}
					!{(5.8,-0.55) }*+{-2}
					"c"-"el"
					"c"-"er"
					"el"-"2l"
					"el"-"3l"
					"er"-"2r"
					"er"-"3r"
				}
				\quad
				=
				\qquad
				\xygraph{
					!{<0cm,0cm>;<0.9cm,0cm>:<0cm,0.9cm>::}
					!~-{@{-}@[|(2.5)]}
					!{(2.8,-0.2) }*+{\bullet}="c"
					!{(1.3,-0.2) }*+{\bullet}="el"
					!{(0,0.55) }*+{\bullet}="2l"
					!{(0,-0.95) }*+{\bullet}="3l"
					!{(4.3,-0.2) }*+{\bullet}="er"
					!{(5.6,0.55) }*+{\bullet}="2r"
					!{(5.6,-0.95) }*+{\bullet}="3r"
					!{(2.8,0.2) }*+{\vphantom{e'}h}
					!{(1.3,0.2) }*+{\vphantom{e'}e}
					!{(4.3,0.2) }*+{e'}
					!{(0,0.95) }*+{\vphantom{e'}v}
					!{(0,-0.55) }*+{\vphantom{e'}w}
					!{(5.6,0.95) }*+{v'}
					!{(5.6,-0.55) }*+{w'}
					"c"-"el"
					"c"-"er"
					"el"-"2l"
					"el"-"3l"
					"er"-"2r"
					"er"-"3r"
				}
			\end{equation}
				
			By Lemma~\ref{l:hom} and~\ref{l:pos}, the vertices $e$, $v$, and $w$ must be in the homology classes $h-e_1-e_2$, $e_2-e_3-\dots-e_{\ell+2}$, and $e_1-e_0$, respectively.
			The class $e'$ can either be in the homology class $h-e_1-e_j$ or in the homology class $h-e_2-e_j$, for some $j>2$.
			In either case, we get a contradiction, because neither of the classes of $v'$ or $w'$ can be in a class $e_2-\sum e_k$ or $e_1-\sum e_k$, by Lemma~\ref{l:pos}.
				
			This proves the first assertion.
			The second assertion follows verbatim as in the case of Proposition~\ref{p:high-five}.
		\end{proof}

	The two propositions we have just proven give an upper bound for the self-intersection of a curve $C$ with given singularities to exists in some closed symplectic 4-manifold.
	In particular, combining either of them with Lemma~\ref{l:high-four}, we obtain that, if we fix the singularity types of a rational curve $C$, the set of integers $s$ such that $s = C\cdot C$ for some $C$ in a closed symplectic 4-manifold with those singularities is either empty or an interval $(-\infty, s_0]$.

	In the rational and unicuspidal case, 	the existence of the Puiseux expansion shows that every singularity type is realized as one singularity of a rational plane curve, say $C_0$. Blowing up all the other singularities of $C_0$ (except for the desired one), one obtains a rational curve $C_1$ in a blow-up of $\cptwo$ that has one singularity of prescribed type. This shows in particular that, in the unicuspidal case, we always have an interval $(-\infty,s_0]$ of realized self-intersections, and Proposition~\ref{p:high-five} gives an upper bound on $s_0$.
	
	Giving explicit lower bounds on $s_0$, however, is less easy.
	In the case of singularities of type $(p,q)$, we can find the bound $s_0 > pq$.
	Indeed, the curve $\{x^pz^{q-p} - y^q=0\}$ has degree $q$, and two singularities: one of type $(p,q)$ at $(0:0:1)$, and the other of type $(q-p,q)$ at $(1:0:0)$.
	Taking the minimal smooth resolution of the latter, we obtain a curve $C$ of self-intersection strictly larger than $q^2-(q-p)q = pq$ (see~\cite[Lemma~2.4]{bigGS}), whose unique singularity is of type $(p,q)$.
	Note that the self-intersection of the proper transform of $C$ in the minimal smooth resolution of its only singularity, too, is positive (again, by~\cite[Lemma~2.4]{bigGS}).
			
		\begin{remark}
			A version of the two obstructions above was already known in the algebro-geometric context;
			indeed, Hartshorne proved in~\cite[Theorem 4.1]{Hartshorne} that if an algebraic surface $X$ contains a smooth genus-$g$ curve $C'$ of self-intersection at least $4g+6$, then $X$ is ruled, and $C'$ is a section. The analogue result in the symplectic context has been proven by K\"utle~\cite{Fabien-thesis}.
			For instance, this applies to the case of a rational curve $C$ whose unique singularity is of type $(2,3)$, and whose self-intersection is at least 10.
			To see this, observe that smoothing the singularity yields a genus-1 curve $C'$ satisfying the requirements of Hartshorne's theorem, so $C'$ is a section of a ruled surface $X \to T^2$. Blow up once at the cusp of $C$ and consider the map from the proper transform of $C$ to $T^2$ obtained by composing the blow-down with the projection. This is a degree-1 map from a sphere to $T^2$, which gives a contradiction.
			This specific case was also known in the symplectic context, by work of Ohta and Ono~\cite{OhtaOno}.

			In general, the bound obtained by applying Hartshorne's or K\"utle's result is weaker than the one obtained by applying Proposition~\ref{p:high-five}.
			This can be seen in the following family of examples: look at rational curves with only one singularity at $p$, which is of type $(F_k, F_{k+1})$. (As usual, $F_k$ is the $k^{\rm th}$ Fibonacci number.) On the one hand, we have $4p_g(C) + 5 = 2(F_{k}-1)(F_{k+1}-1) + 5$, so Hartshorne's bound guarantees that there are no such curves as soon as the self-intersection $s$ satisfies:
			\[
			s \ge 2F_{k}F_{k+1} - 2F_{k} - 2F_{k+1} + 7.
			\]
			On the other hand, the recursive definition of the Fibonacci numbers gives $\ell(p) = F_2 = 2$, so our bound implies that there can be no such rational curve when
			\[
			s \ge M(p) + 2\cdot 2 + 2 = \sum_{j = 2}^k F_h^2 + 6 = F_kF_{k+1} + 4,
			\]
			which is much smaller than the previous bound if $k$ is large. (Here we have used the remarkable identity $\sum_{j=0}^k F_j^2 = F_kF_{k+1}$.)

		\end{remark}

\section{Rational cuspidal curves of low arithmetic genus}
{\label{s:lowgenus}
In this section we look at symplectic filling classifications for contact manifolds $(Y_C,\xi_C)$ where $C$ is a rational cuspidal curve with low complexity, where we take arithmetic genus as our measure of complexity.

Recall from Section~\ref{s:background} that the arithmetic genus of a rational curve is determined by the multiplicity sequences of its singularities; each entry $m$ in the sequence(s) contributes $\frac12 m(m-1)$ to the genus. If we restrict to low genus cases, this significantly restricts the types of singularities that may arise.
We can vary the self-intersection number of the curve, $s$, freely (the curve will admit a concave neighborhood if $s>0$), however we can only utilize techniques from~\cite{bigGS} to classify fillings if $s$ is sufficiently large to ensure that the proper transform of the curve in the minimal smooth resolution has positive self-intersection. Using the notation from Section~\ref{ss:bounds}, this means that we need $s>\sum_p M(p)$, where the sum is taken over all singular points $p$. Note that when the self-intersection gets sufficiently large so that it exceeds the bounds of Proposition~\ref{p:high-five} or~\ref{p:high-six}, the corresponding contact manifold is not fillable.

In this section, we consider all possible rational cuspidal curves $C$ whose algebraic genus is at most three, with self intersection large enough so that the proper transform of the curve in the minimal resolution has strictly positive self-intersection.	For each such curve, we classify the minimal symplectic fillings of the corresponding contact manifold $(Y_C,\xi_C)$.

The reader is encouraged to review techniques and results from~\cite{bigGS}, as we will use them frequently in this section.

\subsection{Genus 1}

A rational cuspidal curve of arithmetic genus $1$ can only have a single multiplicity sequence $[2]$ singularity, which corresponds to type $(2,3)$, i.e. a simple cusp.
Suppose $C$ is a rational cuspidal curve with a simple cusp as its unique singularity and self-intersection $s := C\cdot C$. For convenience, throughout this subsection we will denote the corresponding contact manifold by $(Y_s,\xi_s)$ (though of course the contact manifold depends on the singularities specified).

The minimal smooth resolution results from a single blow-up at the cusp. The total transform consists of an exceptional divisor $E$ of self-intersection $-1$ which is simply tangent to the smooth proper transform of $C$ which has self-intersection $s-4$. Therefore we constrain ourselves to the cases when $s\geq 5$.

\begin{prop}
	Let $(Y_s,\xi_s)$ be the contact boundary of a concave neighborhood of a rational cuspidal curve with a unique simple cusp, and self-intersection $s$.
	\begin{itemize}
		\item When $s=5,6,7,9$, $(Y_s,\xi_s)$ has a unique minimal symplectic filling $W$ and $b_2(W_s)=9-s$.
		\item When $s=8$, $(Y_s,\xi_s)$ has exactly two minimal symplectic fillings $W_A$ and $W_B$, and $b_2(W_A)=b_2(W_b)=1$.
		\item When $s\geq 10$, $(Y_s,\xi_s)$ has no symplectic fillings.
	\end{itemize}
\end{prop}

\begin{proof}
	Since the multiplicity sequence for the unique singularity is $[2]$, following the notation from Section~\ref{ss:bounds}, we have $M(p)=4$ and $\ell(p)=2$. Therefore, by Proposition~\ref{p:high-five}, when $s\geq 10$, $(Y_s,\xi_s)$ has no symplectic fillings.
	
	For any minimal symplectic fillings $W$ of $(Y_s,\xi_s)$, we can glue $W$ to a concave neighborhood of $C$ to obtain a closed symplectic manifold $X$ such that the pair $(X,C)$ is relatively minimal. Conversely, for any relatively minimal symplectic embedding of $C$ into a closed symplectic manifold $X$, there is a concave neighborhood of $C$ whose complement is a minimal symplectic filling of $X$.
	Therefore, to analyze the symplectic fillings when $5\leq s\leq 9$, we classify relatively minimal symplectic embeddings of $C$ into a closed symplectic manifold $X$. If $C$ embeds in $X$, then we can blow-up $X$, once at the singular point $p$ to resolve the singularity of $C$, and $s-5$ times at another point $q$, so that the resulting proper transform $\tilde C$ has self-intersection $1$. See Figure~\ref{fig:genus1}.
	
	\begin{figure}
		\centering
		\includegraphics[scale=.6]{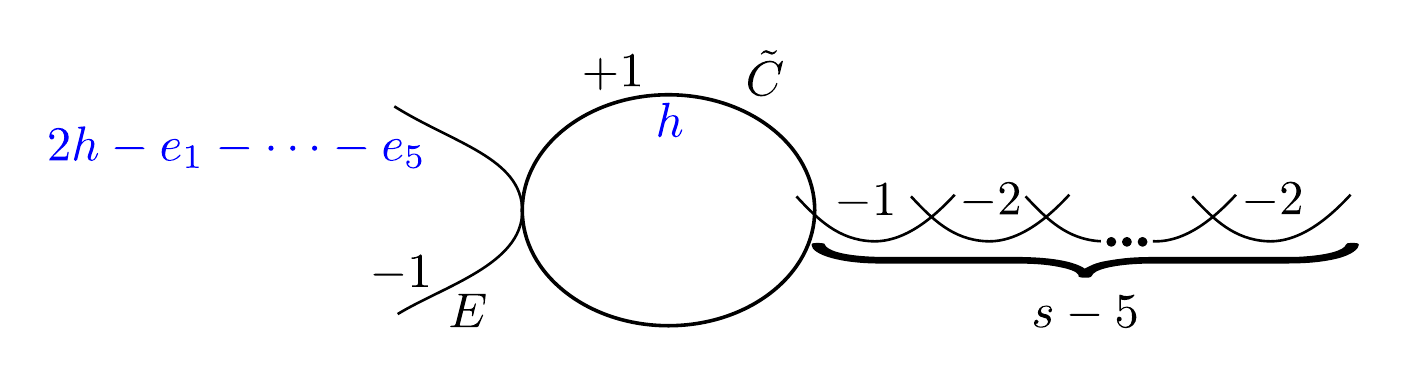}
		\caption{A blow-up of a curve $C$ with a single simple cusp and self-intersection $s$, so that the proper transform $\tilde{C}$ has self-intersection $+1$.}
		\label{fig:genus1}
	\end{figure}

	By Theorem~\ref{thm:mcduff}, there exists a symplectomorphism of $X\#(s-4)\cptwobar$ to $\cptwo\#N\cptwobar$ identifying the proper transform $\tilde C$, a smooth symplectic sphere of self-intersection $+1$ with $\cpone\subset\cptwo$.
	Using the standard basis $\{h,e_1,\cdots, e_N\}$ for $H_2(\cptwo\#N \cptwobar)$, $[C]=h$. We determine the possible homology classes for the other components using the lemmas from section~\ref{s:background}. By Lemma~\ref{l:hom}, the tangent exceptional sphere $E$ represents $2h-e_1-\cdots - e_5$. When $s>5$, there is a chain of $s-5$ additional exceptional divisors. The first sphere in the chain is a $-1$-sphere intersecting $C$ transversally and disjoint from $E$, and thus represents the class $h-e_1-e_2$ (up to relabeling). By Lemma~\ref{l:2chain}, the remaining spheres in the chain represent (a truncation of) $e_2-e_3$, $e_3-e_4$, $e_4-e_5$ or in the case that $s=8$ (so there are exactly two $-2$-spheres in the chain) they can represent $e_2-e_3$, $e_1-e_2$. By~\ref{l:blowdown}, we can blow down exceptional divisors representing $e_1,\dots, e_5$ which either intersect the configuration positively or are components of the image of the configuration after some blow-downs. After blowing down all exceptional classes, the image of the total transform of $C$ descends to a conic (the image of $E$) with one tangent line (the image of $\tilde C$), and, when $s>5$ another line (the image of the first exceptional sphere in the chain) in $\cptwo$. 
	
	By~\cite[Theorem~1.5]{bigGS}, such a configuration has a unique symplectic isotopy class in $\cptwo$. Therefore there is exactly one relatively minimal symplectic embedding (up to symplectic isotopy) of the total transform of $C$ into $X\#{(s-4)}\cptwobar \cong \cptwo\#5\cptwobar$ for each homological embedding. By deleting a concave regular neighborhood of the total transform of $C$, we get exactly one minimal symplectic filling of $(Y_s,\xi_s)$ for each homological embedding. When $s=5,6,7,9$, this shows that there is a unique minimal symplectic filling, corresponding to the unique relatively minimal symplectic embedding of $C$ into $\cptwo\#(5-(s-4))\cptwobar$. Since $(Y_s,\xi_s)$ is a rational homology sphere, the Mayer--Vietoris long exact sequence implies the complementary filling has $b_2=9-s$. By additivity of the signature, it is also negative definite.
	
	When $s=8$, there are two homological embeddings, (where the final $-2$-sphere in the chain can either represent $e_3-e_4$ or $e_1-e_2$). To see that the two corresponding fillings $W_A$ and $W_B$ are different, we look at their intersection forms.
	The torsion-free part of the former is generated by the class $e_1-e_2-e_3-e_4+2e_5$, while that of the latter is generated by the class $e_4-e_5$.
	In particular, the intersection forms of $W_A$ and $W_B$ are non-isomorphic, so $W_A$ and $W_B$ are non-diffeomorphic.
\end{proof}

\begin{remark}
	Note that in the $s=8$ case, the two homological embeddings of the total transform of $C$, correspond to two symplectic embeddings of $C$, one into $\cptwo\#\cptwobar$ and the other into $S^2\times S^2$.
	The former embedding is obtained by blowing up a cuspidal cubic in $\cptwo$ at a non-singular point. The latter is obtained by taking a cuspidal cubic $D$ and a line $L$ meeting $D$ transversely in three points, blowing up $\cptwo$ at two of them, and then contracting the proper transform of $L$.
	In our analyses, this will be a common source of multiple fillings of the same cuspidal contact manifold with the same Betti numbers.
\end{remark}
	
\subsection{Genus 2}
	For a rational cuspidal curve of algebraic genus $2$, there can either be two simple cusps (each with multiplicity sequence $[2]$), or a single cusp which is a cone on a $(2,5)$-torus knot (multiplicity sequence $[2,2]$). In both cases, the minimal resolution results from two blow-ups at points of multiplicity $2$, so we can classify symplectic fillings for such curves when the self-intersection number $s\geq 9$.
	
	\begin{prop}
		Let $C$ be a rational cuspidal curve with self-intersection number $s$, such that either it has exactly two simple cusps or it has a unique singularity of type $(2,5)$, and let $(Y_s,\xi_s)$ denote the corresponding cuspidal contact manifold.
		\begin{itemize}
			\item When $9\leq s \leq 11$, $(Y_s,\xi_s)$ has a unique minimal symplectic filling $W_s$, and $b_2(W_s)=13-s$.
			\item When $s=12$, $(Y_s,\xi_s)$ has exactly two minimal symplectic fillings, $W_A$ and $W_B$, and $b_2(W_A)=b_2(W_B)=1$.
			\item When $s\geq 13$, $(Y_s,\xi_s)$ admits no symplectic fillings.
		\end{itemize}
	\end{prop}

	\begin{proof}
		We proceed with the same methods as used in the genus-$1$ case. The resolutions we embed into $X\#(s-7)\cptwobar \cong \cptwo\#N\cptwobar$ are shown in Figures~\ref{fig:genus2a} and~\ref{fig:genus2b} for the two types of singularity configurations. In each case, we blow up a sufficient number of times so that the proper transform $\tilde C$ has self-intersection $1$, thus identifying $\tilde C$ with $\cpone\subset \cptwo$ using Theorem~\ref{thm:mcduff}.
		
		\begin{figure}
			\centering
			\includegraphics[scale=.6]{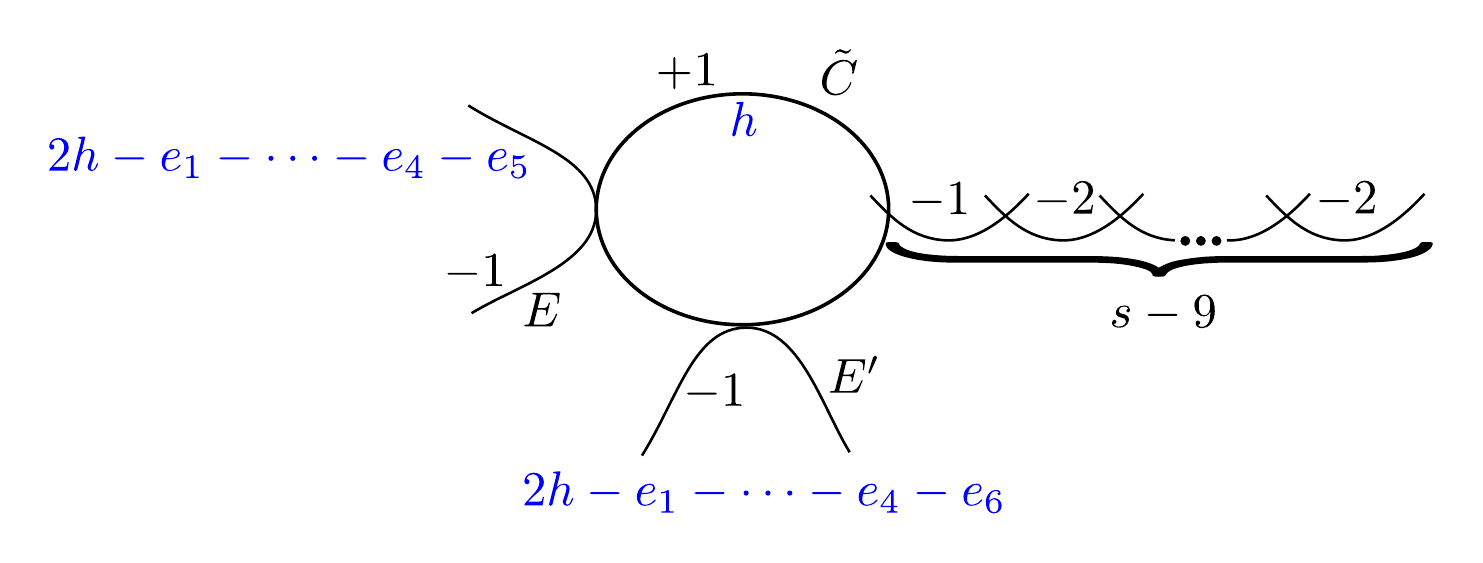}
			\caption{A blow-up of a curve $C$ with a two simple cusps and self-intersection $s$, so that the proper transform $\tilde{C}$ has self-intersection $+1$.}
			\label{fig:genus2a}
		\end{figure}
	
		\begin{figure}
			\centering
			\includegraphics[scale=.6]{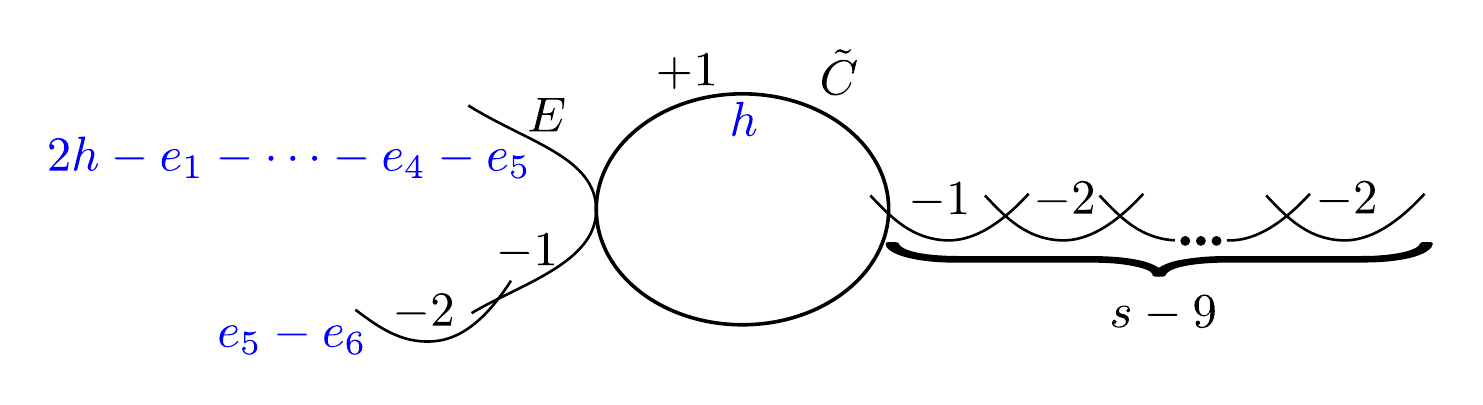}
			\caption{A blow-up of a curve $C$ with a one cusp of type $(2,5)$ and self-intersection $s$, so that the proper transform $\tilde{C}$ has self-intersection $+1$.}
			\label{fig:genus2b}
		\end{figure}

		In the case of two simple cusps, there are two tangent exceptional divisors of self-intersection $-1$ representing classes $2h-e_1-\cdots -e_4-e_5$ and $2h-e_1-\cdots-e_4-e_6$ by Lemma~\ref{l:hom}. By Lemmas~\ref{l:hom} and~\ref{l:2chain}, the exceptional divisors in the chain of length $s-9$ represent classes $h-e_1-e_2, e_2-e_3, e_3-e_4$, or $h-e_1-e_2, e_2-e_3, e_1-e_2$ (when $s<12$, the sequence is truncated and there is a unique option). Using Lemma~\ref{l:blowdown}, we blow down exceptional spheres in the $e_i$ classes which either intersect the configuration positively or are contained in the image of the configuration after blow-down. In the first case, the embedding descends to two conics intersecting transversally with one line tangent to both, and another line passing through two of the intersection points of the conics. In the second case, our two conics intersect at one simple tangency and two transverse intersection points, and the last line is tangent to the two conics at the point where they are tangent to each other. In both cases, without the last line, the configuration has a unique symplectic isotopy class by~\cite[Theorem~1.5]{bigGS}, so adding in the last line, the entire configuration has a unique symplectic isotopy class by~\cite[Proposition~5.1]{bigGS}.  Therefore, there is a unique relatively minimal symplectic embedding when $9\leq s\leq 11$ and there are two relatively minimal symplectic embeddings when $s=12$. 
		
		When $9\leq s\leq 11$, the unique relatively minimal symplectic embeddings correspond to unique minimal symplectic fillings of $(Y_s,\xi_s)$ and the Betti number calculation follows from the Mayer--Vietoris long exact sequence and the fact that $Y_s$ is a rational homology sphere. To distinguish the two fillings when $s=12$, we can again look at the intersection form. For the homological embedding where the last sphere represents $e_3-e_4$, the generator of the non-torsion part of the homology of the complement is $e_1-e_2-e_3-e_4+2e_5+2e_6$ which has self-intersection $-12$. For the homological embedding where the last sphere represents $e_1-e_2$, the generator of the non-torsion part of the homology of the complement is $e_4-e_5-e_6$ which has self-intersection $-3$.
		
		We observe that it is not possible to find a homological embedding if there are more than two $-2$-spheres in the chain (i.e. when $s>12$), since such a $(-2)$-sphere must represent a class $e_i-e_j$ which intersects $e_3-e_4$ (or $e_1-e_2$) once positively, and has intersection zero with all other curves in the configuration. Therefore, there are no symplectic embeddings of $C$ into a closed symplectic manifold when $s\geq 13$, so $(Y_s,\xi_s)$ has no symplectic fillings when $s\geq 13$. Note that Proposition~\ref{p:high-five} would only imply this for $s\geq 14$.

		In the case of one cusp of type $(2,5)$, there is one exceptional divisor $E$ of self-intersection $-1$ tangent to the proper transform $\tilde C$, and another exceptional divisor $E'$ of self-intersection $-2$ which intersects $E$ once transversally and is disjoint from all other curves. Up to relabeling, $[E]=2h-e_1-\cdots -e_5$ and $[E']=e_5-e_6$. The exceptional divisors in the chain of length $s-9$ again represent classes $h-e_1-e_2, e_2-e_3, e_3-e_4$ or $h-e_1-e_2, e_2-e_3, e_1-e_2$. Blowing down exceptional spheres, this configuration descends to a single conic with two lines, which has a unique symplectic isotopy class by~\cite[Theorem~1.5]{bigGS}. Thus this yields a unique symplectic filling when $9\leq s\leq 11$. The two symplectic fillings corresponding to the two homological embeddings when $s=12$ can be distinguished by their intersection forms, as their non-torsion homology is generated by $e_1-e_2-e_3-e_4+2e_5+2e_6$ or $e_4-e_5-e_6$ respectively as in the previous case. Similarly, we obtain no possible homological embeddings when there are more than two $-2$-spheres in the chain so when $s\geq 13$, there are no symplectic fillings of $(Y_s,\xi_s)$.
	\end{proof}

\subsection{Genus 3}

When we allow the algebraic genus of the rational cuspidal curve to increase to $3$, there are more options for the types of cusps. The first option is to have a single cusp of type $(3,4)$ (multiplicity sequence $[3]$), where we consider self-intersection numbers $s\geq 10$. The second is to have a single cusp of type $(2,7)$ (multiplicity sequence $[2,2,2]$). The third option is to have two cusps, one a simple cusp (multiplicity sequence $[2]$) and the other of type $(2,5)$ (multiplicity sequence $[2,2]$). The fourth option is to have three simple cusps (each with multiplicity sequence $[2]$). In the second, third, and fourth options, we can classify fillings when the self-intersection number satisfies $s\geq 13$.

\begin{prop}
	Let $C$ be a rational cuspidal curve with self-intersection $s$ and a cusp of type $(3,4)$ being the unique singularity, and let $(Y_s,\xi_s)$ be the corresponding contact manifold.
	\begin{itemize}
		\item When $10\leq s \leq 16$, $(Y_s,\xi_s)$ has a unique minimal symplectic filling $W$ and $b_2(W)=16-s$.
		\item When $s\geq 17$, $(Y_s,\xi_s)$ is not symplectically fillable.
	\end{itemize}
\end{prop}

\begin{proof}
	That there are no fillings when $s\geq 17$ follows from Proposition~\ref{p:high-five}.
	
	\begin{figure}
		\centering
		\includegraphics[scale=.6]{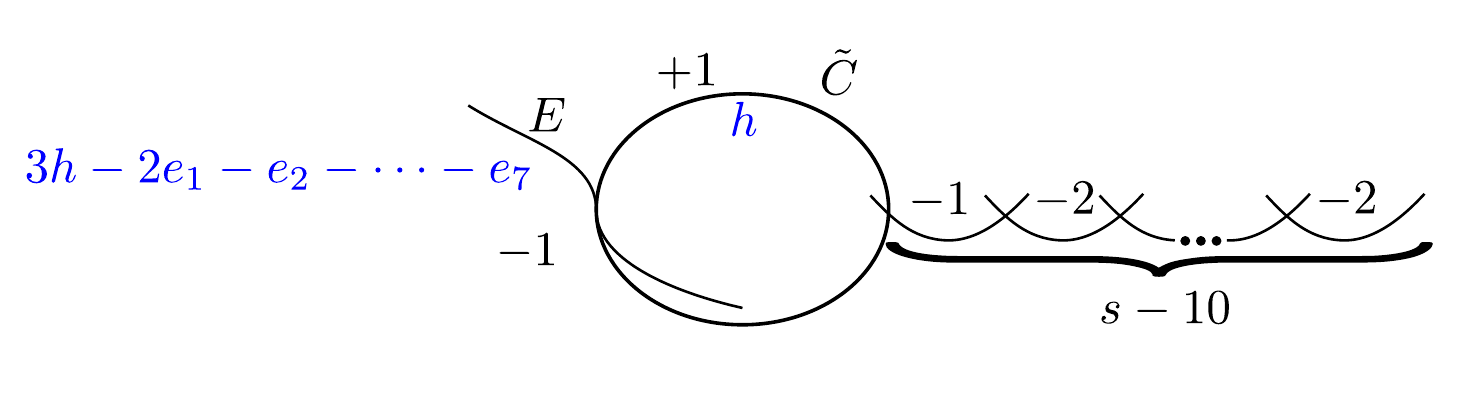}
		\caption{A blow-up of a curve $C$ with a one cusp of type $(3,4)$ and self-intersection $s$, so that the proper transform $\tilde{C}$ has self-intersection $+1$.}
		\label{fig:genus3a}
	\end{figure}
	
	For $10\leq s\leq 16$, we proceed as in the lower genus cases. Blow up an embedding of $C$ into $X$ once at the cusp to get the minimal smooth resolution so that the exceptional sphere $E$ is tangent with multiplicity $3$ to the proper transform of $C$. We blow-up $s-10$ additional times at a different point so that $\tilde C$ has self-intersection $1$, yielding a chain of $s-10$ exceptional spheres as in Figure~\ref{fig:genus3a}. Identifying $\tilde C$ with $\cpone$ by Theorem~\ref{thm:mcduff}
	so $[\tilde C]=h$,  by Lemma~\ref{l:hom}, we see that $[E]=3h-2e_1-e_2-\cdots-e_7$, and the chain of $s-8$ exceptional spheres represent the classes in (a truncation of) the sequence $h-e_1-e_2, e_2-e_3, e_3-e_4, e_4-e_5, e_5-e_6, e_6-e_7$. Blowing down exceptional spheres representing $e_1,\cdots, e_7$ using Lemma~\ref{l:blowdown}, $E$ descends to a cubic curve with either a node or a cusp, $\tilde C$ descends to an inflection line, and the first divisor in the chain descends to a line passing through the singular point of the cubic and one other point. Without the last line, this configuration has a unique symplectic isotopy class by~\cite[Proposition~5.11]{bigGS}. We get a unique symplectic isotopy class for the configuration which includes the last line by~\cite[Proposition~5.1]{bigGS}. Thus there is a unique relatively minimal symplectic embedding of this blow-up of $C$ into a closed manifold and that manifold must be $\cptwo\#7\cptwobar$. We conclude that $(Y_s,\xi_s)$ has a unique minimal symplectic filling with $b_2 = 16-s$.
\end{proof}

\begin{prop}
	Let $C$ be a rational cuspidal curve with self-intersection $s$ and either a unique $(2,7)$ cusp, or one $(2,3)$ cusp and one $(2,5)$ cusp, and let $(Y_s,\xi_s)$ be the corresponding contact manifold.
	\begin{itemize}
		\item When $13\leq s \leq 15$, $(Y_s,\xi_s)$ has two minimal symplectic fillings $W$ and $W'$. $b_2(W)=16-s$ and $b_2(W')=17-s$.
		\item When $s=16$, $(Y_s,\xi_s)$ has three minimal symplectic fillings $W$, $W'$ and $W''$. $b_2(W)=16-s$ and $b_2(W')=b_2(W'')=17-s$.
		\item When $s\geq 17$, $(Y_s,\xi_s)$ is not symplectically fillable.
	\end{itemize}
\end{prop}

\begin{proof}
	\begin{figure}
		\centering
		\includegraphics[scale=.6]{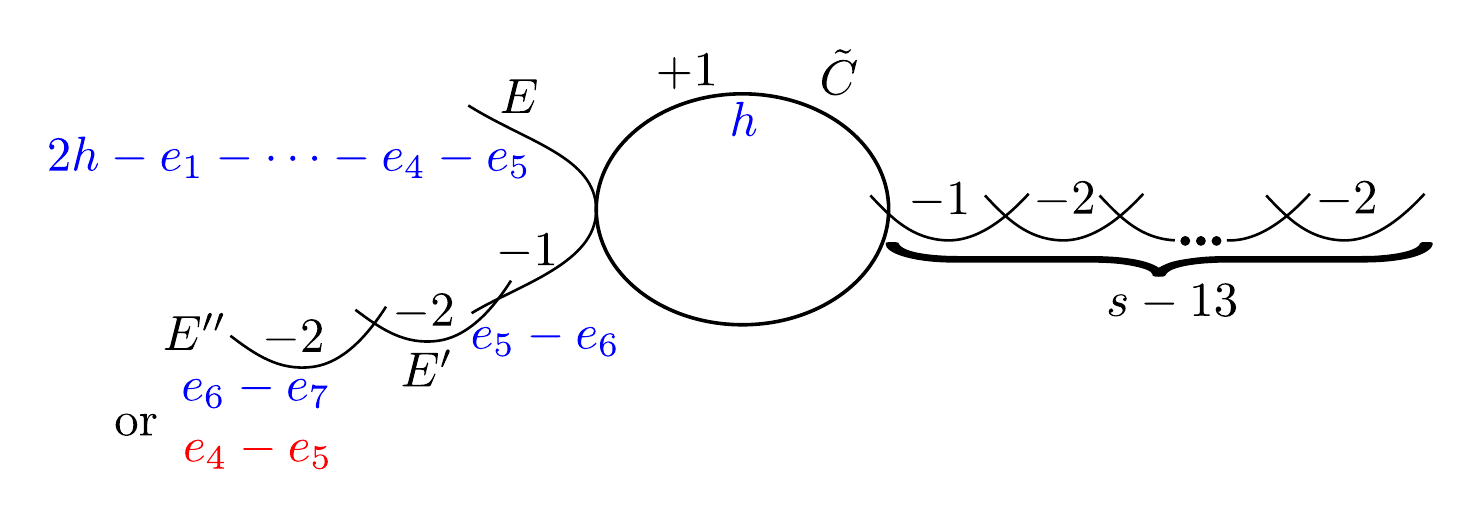}
		\caption{A blow-up of a curve $C$ with a one cusp of type $(2,7)$ and self-intersection $s$, so that the proper transform $\tilde{C}$ has self-intersection $+1$.}
		\label{fig:genus3b}
	\end{figure}

	\begin{figure}
		\centering
		\includegraphics[scale=.6]{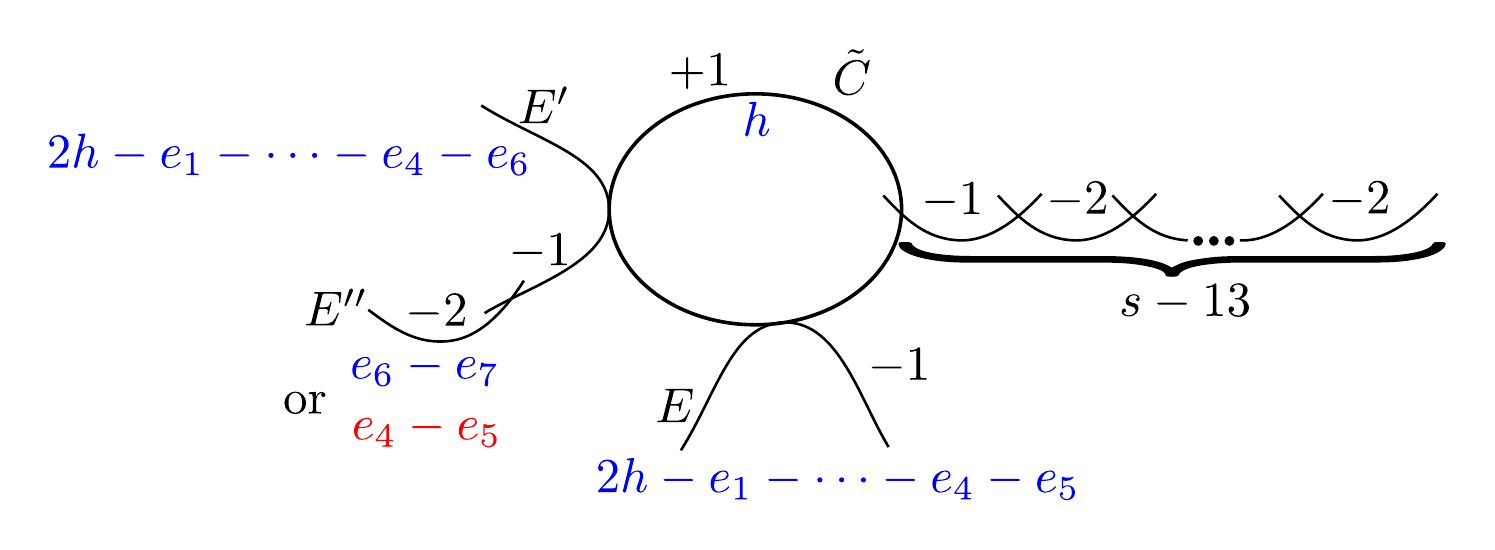}
		\caption{A blow-up of a curve $C$ with a one cusp of type $(2,5)$ and one of type $(2,3)$ and self-intersection $s$, so that the proper transform $\tilde{C}$ has self-intersection $+1$.}
		\label{fig:genus3c}
	\end{figure}
	
	Proceeding as in prior cases, we blow up an embedding of $C$ to its minimal resolution, then blow up $s-13$ additional times at another point until $\tilde C$ has self-intersection $1$ and can be identified with $\cpone$. The three exceptional divisors from the minimal resolution must have homology classes as specified by the options shown in Figures~\ref{fig:genus3b} and~\ref{fig:genus3c} (depending on whether there are $1$ or $2$ cusps). Note that in both cases, there are two options for the homology class of the exceptional divisor $E''$: $e_6-e_7$ or $e_4-e_5$. Also for both cases, the homology classes of the divisors in the extra chain must represent either $h-e_1-e_2,e_2-e_3,e_3-e_4$ or $h-e_1-e_2,e_2-e_3,e_1-e_2$. If $[E'']=e_4-e_5$, only the second option is possible. When $[E'']=e_6-e_7$, both options are possible, but note that they only differ when all three components are included corresponding to the $s=16$ case. Thus, we get three cases when $s=16$ and two cases when $13\leq s\leq 15$. Note, there are no allowed homological embeddings if the chain is longer, so $(Y_s,\xi_s)$ has no fillings when $s\geq 17$. Each configuration descends under blow-downs to a configuration in $\cptwo$ with a unique symplectic isotopy class by~\cite[Theorem~1.5 and Proposition~5.1]{bigGS}, so each homological embedding corresponds to a unique minimal filling. If $[E'']=e_6-e_7$, we have a relatively minimal symplectic embedding into $\cptwo\#7\cptwobar$. If $[E'']=e_4-e_5$ the relatively minimal embedding is into $\cptwo\#6\cptwobar$. Symplectic fillings complementary an embedding in $\cptwo\#7\cptwobar$ will have $b_2=17-s$, and those which are complementary to an embedding in $\cptwo\#6\cptwobar$ will have $b_2=16-s$. Thus, when $13\leq s\leq 15$, the two fillings have different $b_2$ so they are clearly not diffeomorphic. When $s=16$, there are two potential fillings with $b_2=17-s$. To distinguish these, we look at the square of the generator for the homology of the complement. In one case the generator is $e_1-e_2-e_3-e_4+2e_5+2e_6+2e_7$ of square $-16$ and in the other case the generator is $e_4-e_5-e_6-e_7$ of square $-4$. 
\end{proof}
	
	\begin{remark}
	When $s=16$, in the minimal symplectic filling with $b_2=17-s=1$ where the homology class of $E''$ agrees with the option for the filling with $b_2=16-s=0$, we see that there is a symplectic $-4$-sphere in this filling representing the class $e_4-e_5-e_6-e_7$. Rationally blowing down this $-4$-sphere will yield another filling which is necessarily the unique rational homology ball filling. The other minimal symplectic filling with $b_2=17-s=1$ cannot be rationally blown-down (since there is no $-4$-sphere in a manifold with intersection form $\langle -16\rangle$).
	\end{remark}
	
	\begin{prop}
		Let $C$ be a rational cuspidal curve with self-intersection $s$ and three simple cusps, and let $(Y_s,\xi_s)$ be the corresponding contact manifold.
		\begin{itemize}
			\item When $13\leq s \leq 16$, $(Y_s,\xi_s)$ has a unique minimal symplectic filling $W$ and $b_2(W)=16-s$.
			\item When $s\geq 17$, $(Y_s,\xi_s)$ is not symplectically fillable.
		\end{itemize}
	\end{prop}

	\begin{proof}
		The determination of homological embeddings has a similar classification to the previous case. What is different in this situation is that not all of these homological embeddings correspond to actual geometric embeddings, and we need to prove a symplectic isotopy result for the one which does. The possible homological embeddings are shown in Figure~\ref{fig:genus3d}, where the additional chain of $s-13$ exceptional divisors represent classes $h-e_1-e_2$, $e_2-e_3$, and either $e_3-e_4$ or $e_1-e_2$. Note that when $[E'']=2h-e_1-e_2-e_3-e_5-e_6$, and $s=16$, the last divisor in the chain must represent $e_1-e_2$ to be disjoint from $E''$. Therefore we have three homological embeddings. 
		
			\begin{figure}
			\centering
			\includegraphics[scale=.6]{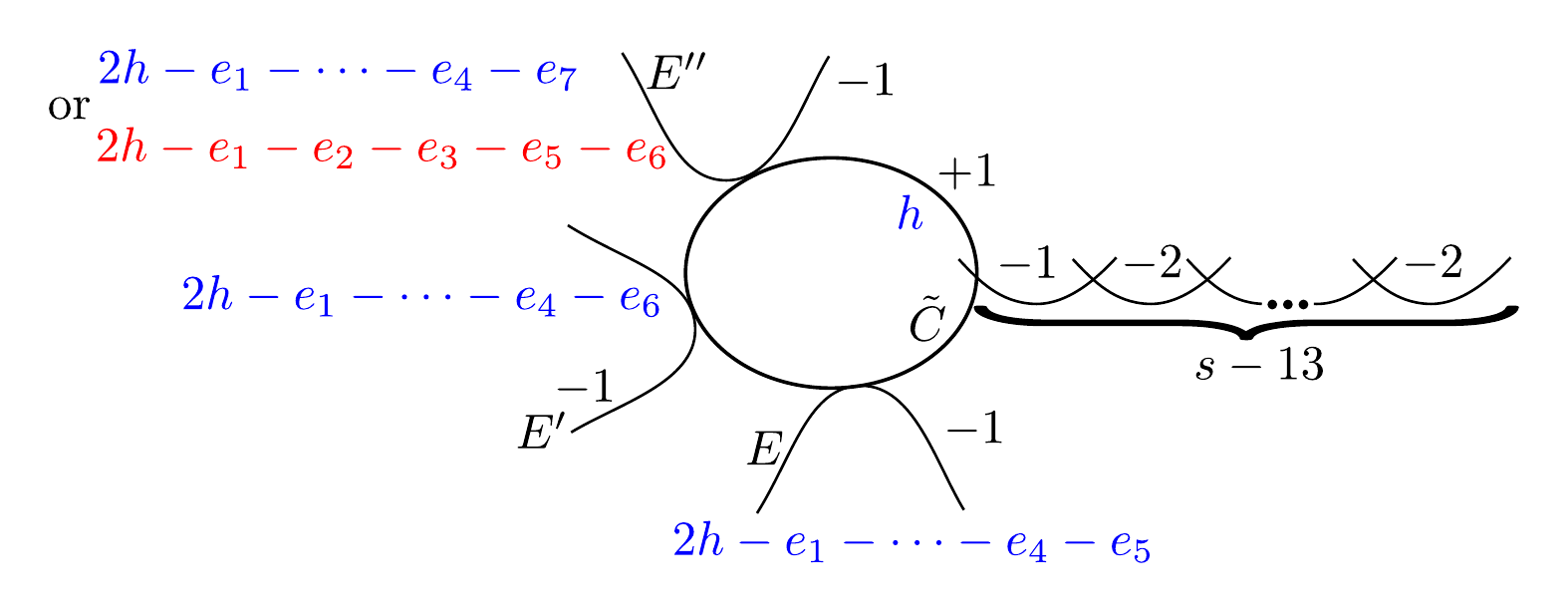}
			\caption{A blow-up of a curve $C$ with a three cusps of type $(2,3)$ and self-intersection $s$, so that the proper transform $\tilde{C}$ has self-intersection $+1$.}
			\label{fig:genus3d}
		\end{figure}
	
		For the two homological embeddings where $[E'']=2h-e_1-\cdots-e_4-e_7$, if we suppose this can be realized symplectically in $\cptwo\#7\cptwobar$, then we can blow down exceptional spheres so that the image in $\cptwo$ of the resolution becomes three symplectic conics in a pencil (all three intersecting at 4 points transversally or all three intersecting at one point tangentially and two points transversally) with a symplectic line tangent to all three conics. This configuration is obstructed by~\cite[Proposition~5.23]{bigGS} (note that the statement refers to a pencil with four transverse intersections, but the same proof applies to the case where the pencil has one base point of multiplicity two--in the birational derivation blow up once at each base point in the pencil). Therefore these homological embeddings cannot be symplectically realized, so there are no corresponding symplectic fillings.
		
		For the last remaining homological embedding where $[E'']=2h-e_1-e_2-e_3-e_5-e_6$ and the chain represents a (truncation of) $h-e_1-e_2$, $e_2-e_3$, $e_1-e_2$, for any symplectic realization of this embedding into $\cptwo\# 6\cptwobar$, there will exist exceptional spheres so that the result blows down to a configuration in $\cptwo$ of three symplectic conics intersecting at one common triple tangency (the image of $e_1$ and $e_2$), one transverse triple point (the image of $e_3$) and a transverse double point for each of the three pairs (the images of $e_4,e_5,e_6$), together with one symplectic line tangent at the common tangency of the conics and another a symplectic line tangent to all three conics at non-intersection points. To verify this configuration has a unique symplectic isotopy class in $\cptwo$, we add a symplectic line through the two points which pass through all three conics--this does not change the symplectic isotopy classification by~\cite[Proposition~5.1]{bigGS}. Then we find a birational equivalence as in Figure~\ref{fig:birat} by (1) blowing up twice at the tangency of the conics and once at the triple intersection of the conics, then (2) blowing down the proper transforms of the two lines which passed through intersection points of the conics (which became $-1$-spheres after step (1)), and (3) blowing down the remaining $-1$-sphere which is the image under the blow-downs of step (2) of an exceptional $-2$-sphere introduced in step (1). The resulting configuration contains a single conic (the image of the line which had three tangencies to the original conics), and five lines---four of which are tangent to the conic (the images of the three original conics and the image of one of the exceptional divisors), and one of which passes transversally through the tangency point of the fourth line with the conic (and otherwise has generic transverse double point intersections). This configuration has a unique symplectic isotopy class because it can be built from a single conic by iteratively applying~\cite[Proposition~5.1]{bigGS}. Therefore, there is a unique minimal symplectic filling corresponding to this homological embedding.
		\end{proof}
		
		\begin{figure}
			\centering
			\includegraphics[scale=.425]{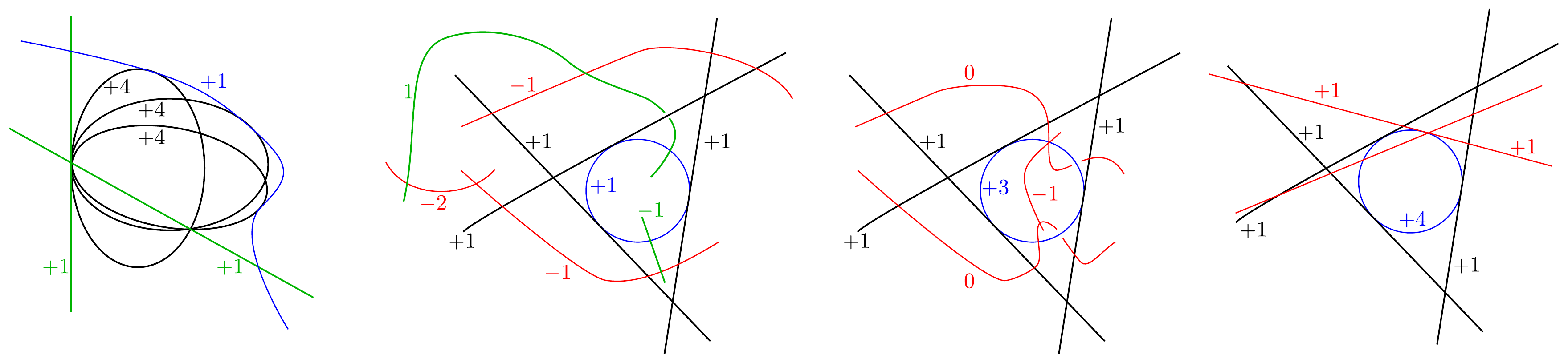}
			\caption{Birational equivalence.}
			\label{fig:birat}
		\end{figure}
}

\section{Further speculations and questions}\label{s:speculations}

	We collect in this section some ideas and questions for further investigation.
	
	\subsection{Handlebodies}
	In Section~\ref{s:AB} we gave explicit Stein handlebody descriptions for all rational homology ball fillings of the cuspidal contact structures associated to the $\Ac_p$ and $\Bc_p$ families. Similar handlebody descriptions are known for rational homology balls bounded by lens spaces or by connected sums of lens spaces~\cite{LekiliMaydanskiy}, thus covering also the two Fibonacci families of~\cite[Theorem 1.1]{FdBLMHN} described in the introduction.
	
	\begin{question}
	Can one find an explicit Stein handlebody description of the rational homology ball fillings of $\Ec$ and $\Dc$?
	\end{question}
	
	In~\cite{AGLL}, there is a (somewhat non-explicit) handle decomposition of a topological rational homology ball bounding the corresponding Seifert fibered spaces.
	Starting from~\cite[Figure~7]{AGLL}, one can recover a handlebody diagram for a rational homology ball bounding the cuspidal manifold of type $\Ec$, as shown in Figure~\ref{f:E3_tophandles}. Does this handlebody support a Stein structure?
	
	\begin{figure}
	\includegraphics[width=0.4\textwidth]{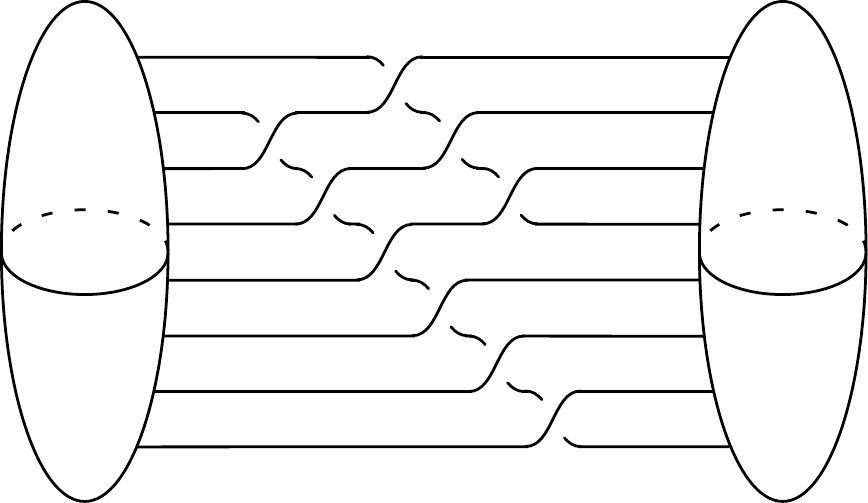}
	\caption{The handlebody description of a rational homology ball whose boundary is the cuspidal manifold of type $\Ec$. The framing of the $2$-handle, in Gompf's convention from~\cite{Gompf2}, is $-9$.}\label{f:E3_tophandles}
	\end{figure}
	
	Note that it would suffice to Legendrian-realize the attaching curve of the 2-handle so that the contact framing is one less than the topological framing: that is, we are looking to find a Legendrian realization of the link of Figure~\ref{f:E3_tophandles} where the 2-handle has Thurston--Bennequin invariant $-8$ (in Gompf's convention~\cite{Gompf2}). By the argument of Proposition~\ref{p:canonical322}, the contact structure obtained on the boundary would automatically be the canonical one (which is the cuspidal one, up to conjugation), and by the classification of fillings, such a diagram would represent the unique rational homology ball filling of $(Y_C,\xi_C)$ where $C$ is a curve of type $\Ec$.
	
	\subsection{Lefschetz fibrations and open books}
	In Sections~\ref{s:AB} and~\ref{s:E3E6} we have given explicit Stein handlebody decompositions for at least one filling for each curve of type $\Ac_p$, $\Bc_p$, $\Ec$, and $\Dc$. For the Fibonacci families, such descriptions were already known to Honda~\cite{Honda} and Giroux~\cite{Giroux} (see also~\cite{LekiliMaydanskiy}).
	
	\begin{question}
	Can one construct explicit Lefschetz fibrations on these fillings?
	\end{question}
	
	As a byproduct, one would also be able to recover an open book decomposition for the cuspidal contact structures in this case. In fact, we conjecture that the diagrams in Figures~\ref{fig:ApLF} and~\ref{fig:BpLF} depict the vanishing cycles of the fillings in the $\Ac_p$ and $\Bc_p$ cases. We verified that the underlying topology is the expected one.
	
	Moreover, adding the orange vanishing cycle in Figure~\ref{fig:ApLF} (respectively the orange vanishing cycle on the left of Figure~\ref{fig:BpLF}) corresponds to attaching a Weinstein $2$-handle along a Legendrian knot in $S^1\times S^2$ which is in the same smooth isotopy class as the Legendrian knot in Figure~\ref{fig:ApStein} (respectively Figure~\ref{fig:BpStein}). We see this by observing that Lefschetz fibrations without these orange curves destabilize to the trivial Lefschetz fibration on the annulus. In the $\Ac_p$ case, the Legendrian represented by the orange vanishing cycle is stabilized $p-1$ times, and the unstabilized version would correspond to the rational homology ball appearing in the standard rational blow-down corresponding to a daisy relation. Starting from a Stein handlebody diagram of this standard rational blow-down, it is not clear whether the Lefschetz fibration depicted represents positive or negative stabilizations of the Legendrian in that diagram, so there is some ambiguity remaining on whether this is the correct Lefschetz fibration. Similar arguments can be made in the $\Bc_p$ case, and a similar ambuiguity remains.
	
			\begin{figure}[h!]
			\centering
			\includegraphics[scale=.45]{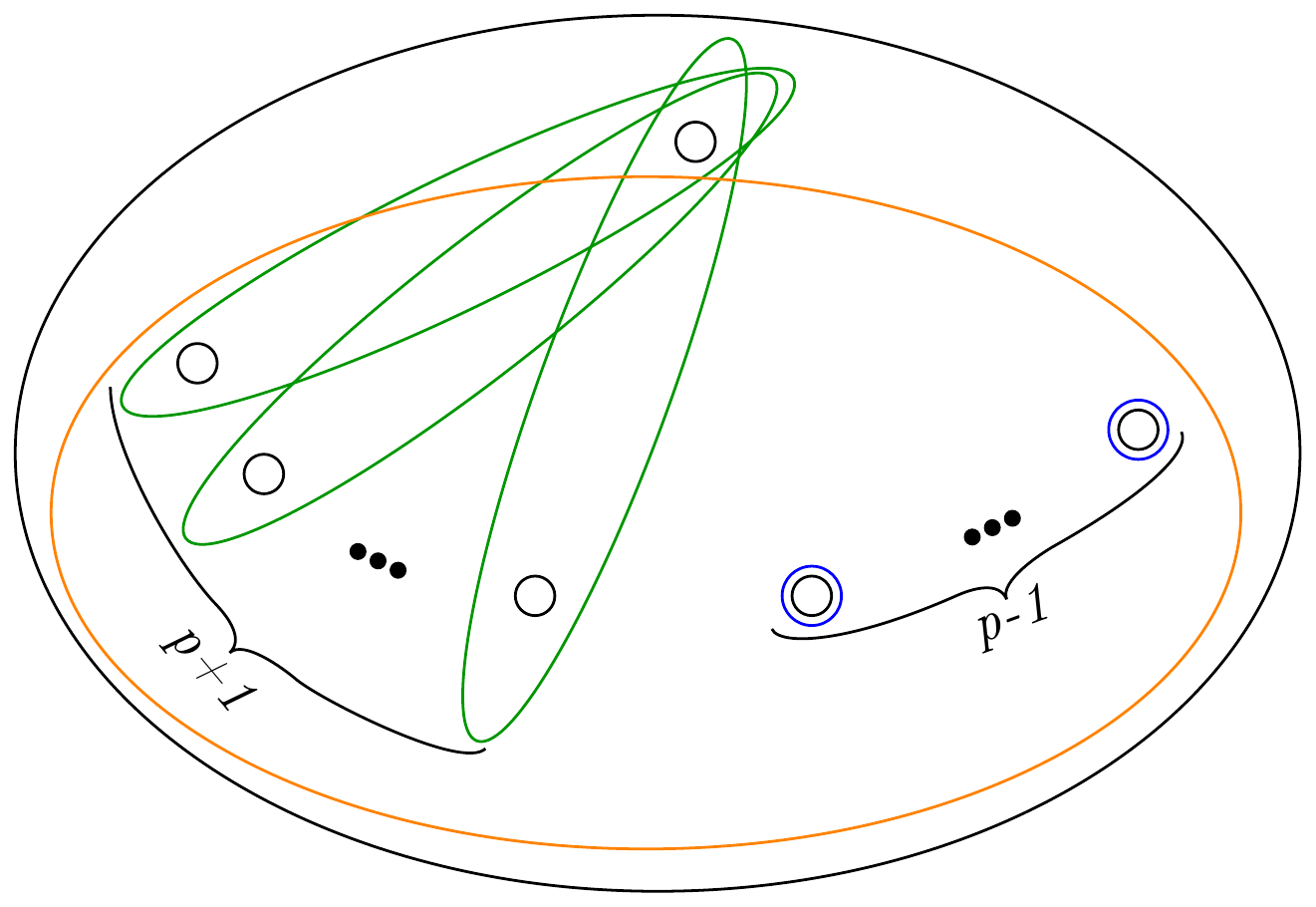}
			\caption{Lefschetz fibration for the filling corresponding to the $\Ac_p$ family.}
			\label{fig:ApLF}
		\end{figure}
		
		\begin{figure}[h!]
			\centering
			\includegraphics[scale=.45]{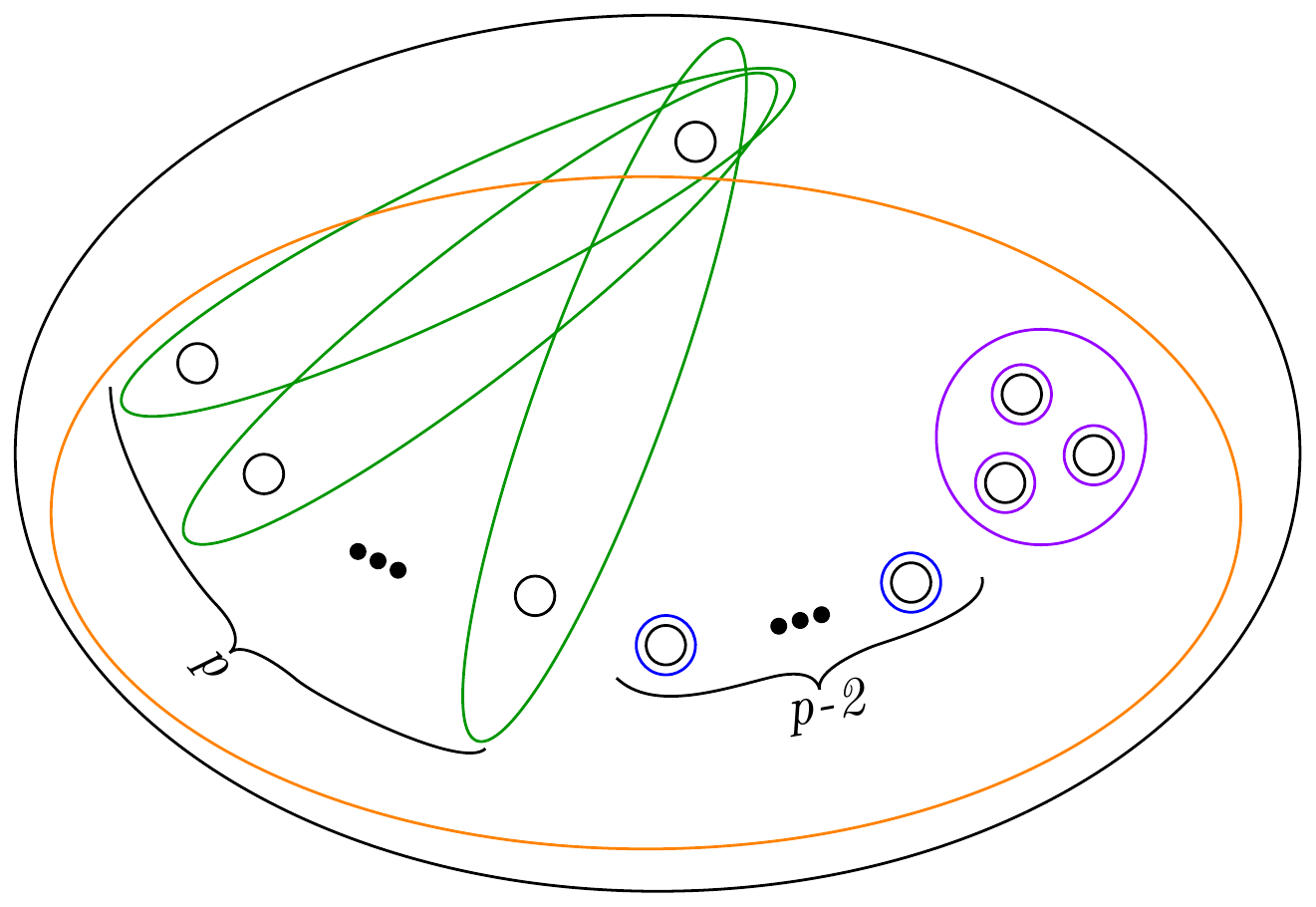} \hspace{.5cm} \includegraphics[scale=.5]{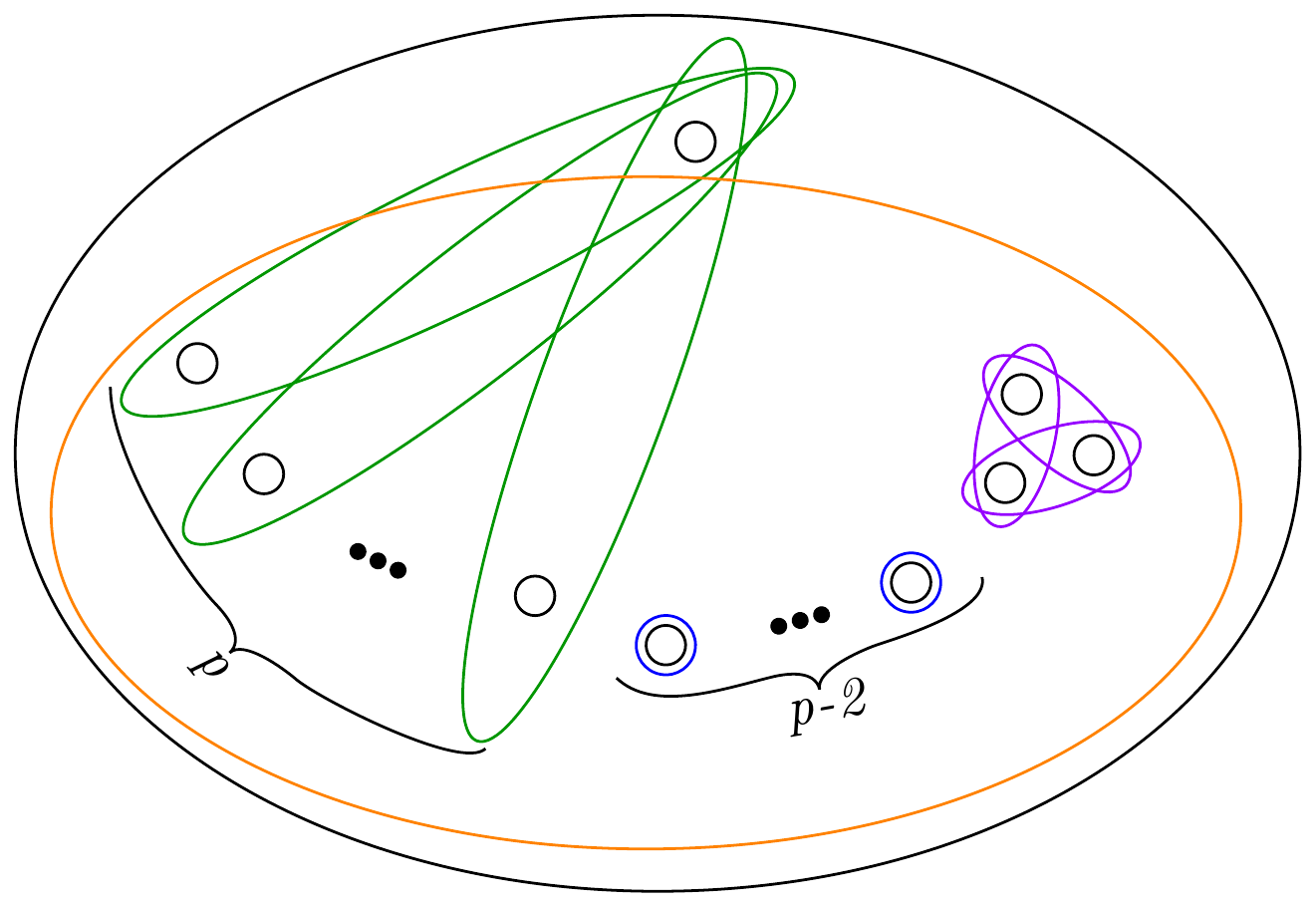}
			\caption{Lefschetz fibrations for the two fillings corresponding to the $\Bc_p$ family.}
			\label{fig:BpLF}
		\end{figure}

	If the proposed open book decompositions do indeed describe the correct contact structures, we would also obtain that all cuspidal contact structures of types $\Ac_p$, $\Bc_p$, or Fibonacci are supported by a planar open book. By contrast, by~\cite[Theorem~1.2]{GhigginiGPlamenevskaya}, we know that in the cases of $\Ec$ and $\Dc$ the cuspidal contact structures \emph{cannot} be planar. On the other hand, Etnyre and Ozbagci have proved that these contact structures are supported by an open book of genus 1~\cite{EtnyreOzbagci}.

\bibliography{cuspidal}

\begin{thebibliography}{FLMN07}

\bibitem[AGLL20]{AGLL}
Paolo Aceto, Marco Golla, Kyle Larson, and Ana~G. Lecuona.
\newblock Surgerires on torus knots, rational balls, and cabling.
\newblock Preprint arXiv:2008.06760, 2020.

\bibitem[BCG16]{BCG}
J\'ozsef Bodn\'ar, Daniele Celoria, and Marco Golla.
\newblock Cuspidal curves and {H}eegaard {F}loer homology.
\newblock {\em Proc. Lond. Math. Soc.}, 112(3):512--548, 2016.

\bibitem[BO16]{BhupalOzbagci}
Mohan Bhupal and Burak Ozbagci.
\newblock Symplectic fillings of lens spaces as {L}efschetz fibrations.
\newblock {\em J. Eur. Math. Soc. (JEMS)}, 18(7):1515--1535, 2016.

\bibitem[BS11]{BhupalStipsicz}
Mohan Bhupal and Andr\'{a}s~I. Stipsicz.
\newblock Weighted homogeneous singularities and rational homology disk
  smoothings.
\newblock {\em Amer. J. Math.}, 133(5):1259--1297, 2011.

\bibitem[EG91]{EliashbergGromov}
Yakov Eliashberg and Mikhael Gromov.
\newblock Convex symplectic manifolds.
\newblock In {\em Several complex variables and complex geometry, Part 2},
  Proceedings of Symposia in Pure Mathematics, pages 135--162, 1991.

\bibitem[Eli90]{Eliash}
Yakov Eliashberg.
\newblock Filling by holomorphic discs and its applications.
\newblock In {\em Geometry of low-dimensional manifolds, 2 ({D}urham, 1989)},
  volume 151 of {\em London Math. Soc. Lecture Note Ser.}, pages 45--67.
  Cambridge Univ. Press, Cambridge, 1990.

\bibitem[EO06]{EtnyreOzbagci}
John~B. Etnyre and Burak Ozbagci.
\newblock Open books and plumbings.
\newblock {\em Int. Math. Res. Not. IMRN}, 2006:72710, 2006.

\bibitem[FLMN07]{FdBLMHN}
Javier {{F}ern\'andez de Bobadilla}, Ignacio Luengo, Alejandro {{M}elle
  Hern\'andez}, and Andras N\'emethi.
\newblock Classification of rational unicuspidal projective curves whose
  singularities have one {P}uiseux pair.
\newblock In {\em Real and complex singularities}, Trends Math., pages 31--45.
  Birkh\"auser, Basel, 2007.

\bibitem[FS97]{FintushelStern}
Ronald Fintushel and Ronald~J. Stern.
\newblock Rational blowdowns of smooth {$4$}-manifolds.
\newblock {\em J. Differential Geom.}, 46(2):181--235, 1997.

\bibitem[FS06]{FintushelSternDN}
Ronald Fintushel and Ronald~J. Stern.
\newblock Double node neighborhoods and families of simply connected
  4-manifolds with {$b^+=1$}.
\newblock {\em J. Amer. Math. Soc.}, 19(1):171--180, 2006.

\bibitem[GGP20]{GhigginiGPlamenevskaya}
Paolo Ghiggini, Marco Golla, and Olga Plamenevskaya.
\newblock Surface singularities and planar contact structures.
\newblock {\em Ann. Inst. Fourier (Grenoble)}, 70(4):1791--1823, 2020.

\bibitem[Ghi08]{Ghiggini-Lclassification}
Paolo Ghiggini.
\newblock On tight contact structures with negative maximal twisting number on
  small {S}eifert manifolds.
\newblock {\em Algebr. Geom. Topol.}, 8(1):381--396, 2008.

\bibitem[Gir00]{Giroux}
Emmanuel Giroux.
\newblock Structures de contact en dimension trois et bifurcations des
  feuilletages de surfaces.
\newblock {\em Invent. Math.}, 141(3):615--689, 2000.

\bibitem[GK20]{GKutle}
Marco Golla and Fabien K{\"u}tle.
\newblock Symplectic isotopy of rational cuspidal sextics and septics.
\newblock Preprint arXiv:2008.10923, 2020.

\bibitem[Gom98]{Gompf2}
Robert~E. Gompf.
\newblock Handlebody construction of {S}tein surfaces.
\newblock {\em Ann. of Math. (2)}, 148(2):619--693, 1998.

\bibitem[GS99]{GompfStipsicz}
Robert~E. Gompf and Andr\'{a}s~I. Stipsicz.
\newblock {\em {$4$}-manifolds and {K}irby calculus}, volume~20 of {\em
  Graduate Studies in Mathematics}.
\newblock American Mathematical Society, Providence, RI, 1999.

\bibitem[GS19]{bigGS}
Marco Golla and Laura Starkston.
\newblock Symplectic isotopy of rational curves.
\newblock Preprint, arXiv:1907.06787, 2019.

\bibitem[Har69]{Hartshorne}
Robin Hartshorne.
\newblock Curves with high self-intersection on algebraic surfaces.
\newblock {\em Publ. Math. I.H.E.S.}, 36:111--125, 1969.

\bibitem[Hon00]{Honda}
Ko~Honda.
\newblock On the classification of tight contact structures, {I}.
\newblock {\em Geom. Topol.}, 4(1):309--368, 2000.

\bibitem[Kap79]{Kaplan}
Steve~J. Kaplan.
\newblock Constructing framed {$4$}-manifolds with given almost framed
  boundaries.
\newblock {\em Trans. Amer. Math. Soc.}, 254:237--263, 1979.

\bibitem[{K}{\"u}t21]{Fabien-thesis}
Fabien {K}{\"u}tle.
\newblock {\em Courbes symplectiques de haute auto-intersection dans les
  surfaces symplectiques}.
\newblock PhD thesis, Universit\'e de Nantes, 2021.
\newblock Preprint, arXiv:2111.04836.

\bibitem[Lis08]{Li}
Paolo Lisca.
\newblock On symplectic fillings of lens spaces.
\newblock {\em Trans. Amer. Math. Soc.}, 360(2):765--799, 2008.

\bibitem[Liu14]{Liu}
Tiankai Liu.
\newblock {\em On planar rational cuspidal curves}.
\newblock PhD thesis, Massachussets Institute of Technology, 2014.

\bibitem[LM14]{LekiliMaydanskiy}
Yank{\i} Lekili and Maksim Maydanskiy.
\newblock The symplectic topology of some rational homology balls.
\newblock {\em Comment. Math. Helv.}, 89(3):571--596, 2014.

\bibitem[LS04]{LiscaStipsicz-I}
Paolo Lisca and Andr\'{a}s~I. Stipsicz.
\newblock Ozsv\'{a}th-{S}zab\'{o} invariants and tight contact three-manifolds.
  {I}.
\newblock {\em Geom. Topol.}, 8:925--945, 2004.

\bibitem[Mat18]{Matkovic}
Irena Matkovi\v{c}.
\newblock Classification of tight contact structures on small {S}eifert fibered
  {$L$}-spaces.
\newblock {\em Algebr. Geom. Topol.}, 18(1):111--152, 2018.

\bibitem[McD90]{Mc}
Dusa McDuff.
\newblock The structure of rational and ruled symplectic {$4$}-manifolds.
\newblock {\em J. Amer. Math. Soc.}, 3(3):679--712, 1990.

\bibitem[Mos71]{Moser}
Louise Moser.
\newblock Elementary surgery along a torus knot.
\newblock {\em Pacific J. Math.}, 38(3):737--745, 1971.

\bibitem[OO99]{OhtaOno2}
Hiroshi Ohta and Kaoru Ono.
\newblock Simple singularities and topology of symplectically filling
  4-manifold.
\newblock {\em Comment. Math. Helv.}, 74:575--590, 1999.

\bibitem[OO05]{OhtaOno}
Hiroshi Ohta and Kaoru Ono.
\newblock Symplectic $4$-manifolds containing singular rational curves with
  $(2,3)$-cusp.
\newblock {\em S\'emin. Congr.}, 10:233--241, 2005.

\bibitem[OS12]{OwensStrle}
Brendan Owens and Sa{\v{s}}o Strle.
\newblock Dehn surgeries and negative-definite four-manifolds.
\newblock {\em Selecta Math. (N.S.)}, 18(4):839--854, 2012.

\bibitem[O{\relax Sz}03]{OSz-correctionterms}
Peter~S. Ozsv\'ath and Zolt\'an {\relax Sz}ab\'o.
\newblock Absolutely graded {F}loer homologies and intersection forms for
  four-manifolds with boundary.
\newblock {\em Adv. Math.}, 173(2):179--261, 2003.

\bibitem[O{\relax Sz}04]{OSz-HF}
Peter~Steven Ozsv{\'a}th and Zolt{\'a}n {\relax Sz}ab{\'o}.
\newblock Holomorphic disks and topological invariants for closed
  three-manifolds.
\newblock {\em Ann. of Math. (2)}, 159(3):1027--1158, 2004.

\bibitem[O{\relax Sz}05]{OSz-contact}
Peter~S. Ozsv\'ath and Zolt\'an {\relax Sz}ab\'o.
\newblock Heegaard {F}loer homology and contact structures.
\newblock {\em Duke Math. J.}, 129(1):39--61, 2005.

\bibitem[Par97]{Park}
Jongil Park.
\newblock Seiberg-{W}itten invariants of generalised rational blow-downs.
\newblock {\em Bull. Austral. Math. Soc.}, 56(3):363--384, 1997.

\bibitem[Par05]{ParkExotic}
Jongil Park.
\newblock Simply connected symplectic 4-manifolds with {$b^+_2=1$} and
  {$c^2_1=2$}.
\newblock {\em Invent. Math.}, 159(3):657--667, 2005.

\bibitem[PSS05]{ParkStipsiczSzabo}
Jongil Park, Andr\'{a}s~I. Stipsicz, and Zolt\'{a}n Szab\'{o}.
\newblock Exotic smooth structures on {$\Bbb{CP}^2\#5\overline{\Bbb{CP}^2}$}.
\newblock {\em Math. Res. Lett.}, 12(5-6):701--712, 2005.

\bibitem[SS05]{StipsiczSzabo}
Andr\'{a}s~I. Stipsicz and Zolt\'{a}n Szab\'{o}.
\newblock An exotic smooth structure on {$\Bbb C\Bbb P^2\#6\overline{\Bbb C\Bbb
  P^2}$}.
\newblock {\em Geom. Topol.}, 9:813--832, 2005.

\bibitem[SSW08]{SSW}
Andr\'{a}s~I. Stipsicz, Zolt\'{a}n Szab\'{o}, and Jonathan Wahl.
\newblock Rational blowdowns and smoothings of surface singularities.
\newblock {\em J. Topol.}, 1(2):477--517, 2008.

\bibitem[Sym98]{Symington1}
Margaret Symington.
\newblock Symplectic rational blowdowns.
\newblock {\em J. Differential Geom.}, 50(3):505--518, 1998.

\bibitem[Sym01]{Symington2}
Margaret Symington.
\newblock Generalized symplectic rational blowdowns.
\newblock {\em Algebr. Geom. Topol.}, 1:503--518, 2001.

\bibitem[Tos20]{Tosun}
B\"{u}lent Tosun.
\newblock Tight small {S}eifert fibered spaces with $e_0 = -2$.
\newblock {\em Algebr. Geom. Topol.}, 20(1):1--27, 2020.

\bibitem[Tur02]{Turaev}
Vladimir Turaev.
\newblock {\em Torsions of {$3$}-dimensional manifolds}, volume 208 of {\em
  Progress in Mathematics}.
\newblock Birkh\"{a}user Verlag, Basel, 2002.

\bibitem[Wal04]{Wall}
C.~T.~C. Wall.
\newblock {\em Singular points of plane curves}, volume~63.
\newblock Cambridge University Press, 2004.

\bibitem[Wen10]{Wendl}
Chris Wendl.
\newblock Strongly fillable contact manifolds and {$J$}-holomorphic foliations.
\newblock {\em Duke Math. J.}, 151(3):337--384, 2010.

\end{thebibliography}
\bibliographystyle{alpha}

\end{document}